\documentclass[11pt]{amsart}
\usepackage[colorlinks=true,citecolor=black!60!green,linkcolor=black!60!red,filecolor=black!60!cyan,urlcolor=black!60!magenta]{hyperref}
\usepackage[text={6.1in,8.5in},centering]{geometry}
\usepackage{amssymb,amsmath,amsthm,mathtools,extpfeil,units,empheq}
\usepackage{tikz,caption}
\usepackage[shortlabels]{enumitem}

\newcommand*\circled[1]{\tikz[baseline=(char.base)]{
    \node[shape=circle,draw,inner sep=1pt] (char) {#1};}}

\newtheorem{thm}[equation]{Theorem}
\newtheorem*{thm*}{Theorem}
\newtheorem*{lem*}{Lemma}
\newtheorem{thmA}{Theorem}

\newtheorem{lem}[equation]{Lemma}
\newtheorem{prop}[equation]{Proposition}
\newtheorem*{prop*}{Proposition}
\newtheorem{cor}[equation]{Corollary}

\theoremstyle{definition}
\newtheorem*{defn}{Definition}
\newtheorem{rmk}[equation]{Remark}

\numberwithin{equation}{section}

\DeclareMathOperator{\config}{config}
\DeclareMathOperator{\id}{id}

\DeclareMathOperator{\sgn}{sign}
\DeclareMathOperator{\cel}{cell}
\DeclareMathOperator{\ucel}{ucell}

\DeclareMathOperator{\spin}{spin}

\DeclareMathOperator{\no}{no}
\DeclareMathOperator{\FI}{FI}
\DeclareMathOperator{\FB}{FB}
\DeclareMathOperator{\FBW}{FBW}

\DeclareMathOperator{\length}{length}
\DeclareMathOperator{\codim}{codim}

\newcommand{\epsi}{\varepsilon}
\newcommand{\concat}{\mid}
\newcommand{\del}{\partial}
\newcommand*\strong[1]{\textbf{\textit{#1}}}
\newcommand{\abs}[1]{\left\lvert #1 \right\rvert}
\newcommand{\co}{\colon\thinspace}

\newcommand{\tendigitwidth}{1.7cm}

\begin{document}
\title[Disks in a strip, twisted algebras, persistence]{Configuration spaces of
  disks in a strip, twisted algebras, persistence, and other stories}
\author[H.~Alpert]{Hannah Alpert}
\address[H.~Alpert]{Department of Mathematics, Auburn University, Auburn, Alabama, United States}
\email{hcalpert@auburn.edu}
\author[F.~Manin]{Fedor Manin}
\address[F.~Manin]{Department of Mathematics, UCSB, Santa Barbara, California, United States}
\email{manin@math.ucsb.edu}
\begin{abstract}
  We give $\mathbb{Z}$-bases for the homology and cohomology of the configuration
  space of $n$ unit disks in an infinite strip of width $w$, first studied by
  Alpert, Kahle and MacPherson.  We also study the way these spaces evolve both
  as $n$ increases (using the framework of representation stability) and as $w$
  increases (using the framework of persistent homology).  Finally, we include
  some results about the cup product in the cohomology and about the
  configuration space of unordered disks.
\end{abstract}
\maketitle

\setcounter{tocdepth}{1}
\tableofcontents

\section{Introduction}

The configuration space of $n$ labeled unit-diameter disks in an infinite strip of width $w$ is denoted $\config(n, w)$; Figure~\ref{fig-def} depicts an example configuration.  Specifically, parametrizing the configurations in terms of the centers of the disks, $\config(n, w)$ is the set of points $(x_1, y_1, \ldots, x_n, y_n) \in \mathbb{R}^{2n}$, such that $(x_i - x_j)^2 + (y_i - y_j)^2 \geq 1$ for all $i$ and $j$, and such that $\frac{1}{2} \leq y_i \leq w - \frac{1}{2}$ for all $i$.  We would like to describe the topology of $\config(n, w)$.

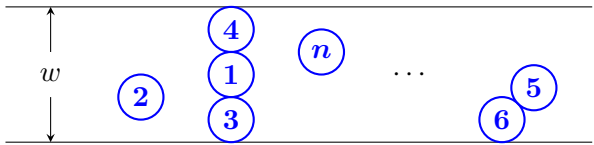
\begin{figure}[h!]
\begin{center}
\begin{tikzpicture}[scale=.6, emp/.style={inner sep = 0pt, outer sep = 0pt, blue}, >=stealth]
\draw (-1, 0)--(12, 0);
\draw (-1, 3)--(12, 3);
\node at (0, 1.5) {$w$};
\draw[->] (0, 2)--(0, 3);
\draw[->] (0, 1)--(0, 0);
\node[emp] (d2) at (2, 1) {\bf 2};
\node[emp] (d3) at (4, .5) {\bf 3};
\node[emp] (d1) at (4, 1.5) {\bf 1};
\node[emp] (d4) at (4, 2.5) {\bf 4};
\node[emp] (dn) at (6, 2) {$\boldsymbol{n}$};
\node at (8, 1.5) {$\cdots$};
\node[emp] (d6) at (10, .5) {\bf 6};
\node[emp] (d5) at (10.707, 1.207) {\bf 5};
\draw[blue, thick] (d2) circle (.5);
\draw[blue, thick] (d3) circle (.5);
\draw[blue, thick] (d1) circle (.5);
\draw[blue, thick] (d4) circle (.5);
\draw[blue, thick] (dn) circle (.5);
\draw[blue, thick] (d6) circle (.5);
\draw[blue, thick] (d5) circle (.5);
\end{tikzpicture}
\end{center}
\caption{The configuration space $\config(n, w)$ is the set of ways to arrange $n$ disjoint labeled disks of width $1$ in $\mathbb{R} \times [0, w]$.}\label{fig-def}
\end{figure}

The topology of the configuration space of $n$ points in the plane has been
well-understood since the work of Arnol'd~\cite{Arnold69} and
F.~Cohen~\cite[Ch.~III]{CLM}; see~\cite{Sinha} for an overview.  Recently, there
has been interest in more ``physical'' models in which the points have thickness
and are constrained to lie in a bounded region, drawing inspiration from both
statistical physics~\cite{Diaconis09} and robotics~\cite{Farber08}.  In the
first, one imagines molecules of a substance as hard balls and extracts
information about states of matter from the way they move past each other.  In
the latter, one can imagine a number of robots coordinating their movements so
that they can travel to different points in a constrained region without bumping
against each other; this amounts to a motion planning problem in a disk
configuration space.

The topology of disk configuration spaces was first studied mathematically by Baryshnikov, Bubenik, and Kahle in~\cite{BBK14} and experimentally by Carlsson et
al.~\cite{CGKM12}.  While these papers represent real progress, trying to fully
understand even the connected components of these spaces seems daunting;
see~\cite{Kahle12}.  Further work has taken two approaches to simplifying the
question.  One is to replace the disks by polygons, such as squares or hexagons,
as in \cite{Alp20} and \cite{ABKMS}.  The topology of the resulting configuration
spaces is closely related to that of disk configuration spaces; in particular, it
captures any of their topology that ``survives for a long time'' as disks grow or
shrink.  This observation can be formalized using persistent homology.


While it is simpler in some respects than that of disk configuration spaces, the
topology of polygon configuration spaces still seems very difficult to
understand.  A more radical simplification, it turns out, is to remove the side
walls of the rectangle, replacing it with the infinite strip, an idea introduced
in~\cite{AKM}.  That paper defines the spaces $\config(n, w)$ and computes the
asymptotic growth of the rank of $H_j(\config(n, w))$ as $n$ increases, up
to a constant factor depending on $j$ and $w$.  This turns out to be exponential
unless the strip is wide compared to $j$.

In this paper, we present a number of results about the topology of
$\config(n, w)$ and related spaces.  The majority of the paper seeks to
understand $H_*(\config(n, w); \mathbb{Z})$, and its dependence on $n$ and $w$, in greater
resolution and from a more algebraic perspective.  For fixed $n$ and $w$, we find
a geometrically motivated basis for this homology.  All the classes in this basis
can be assembled out of a small number of classes involving a small number of
disks, depending only on $w$; we make this idea precise using some algebraic
machinery due to Sam and Snowden.  We also track the appearance and disappearance
of homology classes as the width of the strip changes, using the machinery of
persistent homology.

The paper includes several additional results proven using similar methods.  We
give a basis for the cohomology of $\config(n,w)$ and make progress in
understanding its cup product.  We explore the possibility of $\FI_d$-module
structures on the homology of hard disk configuration spaces and no-$(k+1)$-equal
spaces.  Finally, we discuss the topology of the configuration space of unordered
disks in a strip.

\subsection{Main results}
We now discuss our main results in greater detail, starting with a description of
the homology of $\config(n,w)$ for fixed $w$ and all $n$.
\begin{thmA} \label{thm:FG}
  For fixed $w$, $H_*(\config({-},w); \mathbb{Z})$ forms a finitely generated, noncommutative
  twisted algebra whose generators live in
  $H_{\leq \frac{3}{2}w-2}(\config(\leq 3w/2,w); \mathbb{Z})$.
\end{thmA}
This contrasts with the classical family of configuration spaces of points in the
plane, which forms a commutative, but infinitely generated twisted algebra.

Informally, Theorem \ref{thm:FG} means that $H_j(\config(n,w); \mathbb{Z})$ is spanned by
cycles built as follows:
\begin{enumerate}
\item Separate the $n$ disks into groups of at most $3w/2$.
\item Place the groups in some order along the strip.
\item Label the disks in some way using the numbers $1$ through $n$.
\item Let each group do its own thing, without interacting with the others.
\end{enumerate}
The things a group can do---\emph{elementary cycles}---come in two types: $1$ to
$w$ disks can form a \strong{wheel}, and $w+1$ to $3w/2$ disks can form a
\strong{filter}. 
If $z_1$ and $z_2$ are elementary cycles, we refer to the act of placing them
next to each other as the \strong{concatenation product}, denoted
$z_1 \concat z_2$.

A wheel of $k$ disks has $k-1$ circular degrees of freedom generated by the
rotation of concentric disks, making for a cycle represented by a $T^{k-1}$.
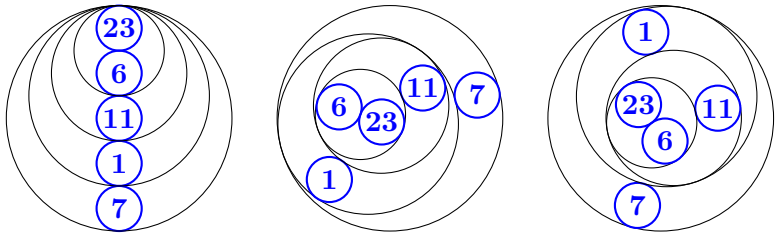
\begin{figure}[h]
  \begin{tikzpicture}[scale=0.6]
    \draw[blue, thick] (6,0) circle (0.5);
    \draw[blue, thick] (6,1) circle (0.5);
    \draw[blue, thick] (6,2) circle (0.5);
    \draw[blue, thick] (6,3) circle (0.5);
    \draw[blue, thick] (6,4) circle (0.5);
    \draw(6,3.5) circle (1);
    \draw(6,3) circle (1.5);
    \draw(6,2.5) circle (2);
    \draw(6,2) circle (2.5);
    \node[blue] at (6,0) {\bf 7};
    \node[blue] at (6,1) {\bf 1};
    \node[blue] at (6,2) {\bf 11};
    \node[blue] at (6,3) {\bf 6};
    \node[blue] at (6,4) {\bf 23};

    \draw(12,2) circle (2.5);
    \coordinate (A) at (11.517,1.871);
    \draw(A) circle (2);
    \draw(A) ++ (235:1.5) coordinate (B);
    \draw[blue,thick](B) circle (0.5);
    \node[blue] at (B) {\bf 1};
    \draw(A) ++ (235:-0.5) coordinate (B2);
    \draw(B2) circle (1.5);
    \draw(B2) ++ (23:1) coordinate (C);
    \draw[blue,thick](C) circle (0.5);
    \node[blue] at (C) {\bf 11};
    \draw(B2) ++ (23:-0.5) coordinate (C2);
    \draw(C2) circle (1);
    \draw(C2) ++ (160.5:0.5) coordinate(D1);
    \draw(C2) ++ (160.5:-0.5) coordinate(D2);
    \draw[blue,thick](D1) circle (0.5);
    \node[blue] at (D1) {\bf 6};
    \node[blue] at (D2) {\bf 23};
    \draw[blue,thick](D2) circle (0.5);
    \draw[blue, thick] (13.933,2.514) circle (0.5);
    \node[blue] at (13.933,2.514) {\bf 7};

    \draw(18,2) circle (2.5) coordinate(A);
    \draw(A)++(255:2) coordinate(A2);
    \draw[blue, thick] (A2) circle (0.5);
    \draw(A) ++(255:-0.5) coordinate(B);
    \draw (B) circle(2);
    \draw (B) ++(108:1.5) coordinate(C1);
    \draw (B) ++(108:-0.5) coordinate(C2);
    \draw[blue,thick] (C1) circle (0.5);
    \draw (C2) circle (1.5);
    \draw (C2) ++(12.5:1) coordinate(D1);
    \draw (C2) ++(12.5:-0.5) coordinate(D2);
    \draw[blue,thick] (D1) circle (0.5);
    \draw (D2) circle (1);
    \draw (D2) ++(306:0.5) coordinate(E1);
    \draw (D2) ++(306:-0.5) coordinate(E2);
    \draw[blue,thick] (E1) circle (0.5);
    \draw[blue,thick] (E2) circle (0.5);
    \node[blue] at (A2) {\bf 7};
    \node[blue] at (C1) {\bf 1};
    \node[blue] at (D1) {\bf 11};
    \node[blue] at (E1) {\bf 6};
    \node[blue] at (E2) {\bf 23};
  \end{tikzpicture}
  \caption{Some configurations of a wheel with $5$ disks.  The first
    configuration gives a canonical (up to switching the first two) ordering of
    the disks.
  } \label{fig:wheel}
\end{figure}

A filter consists of $r \geq 3$ wheels with $k>w$ disks total, such that each
wheel is made of at least $k-w$ disks; the wheels are ordered.  The filter can
move as follows.  Every wheel can perform its rotations independently, for a
total of $k-r$ degrees of freedom; each wheel can also move back and forth along
the strip, crossing over each other in order.  Any $r-1$ wheels can have the same
$x$-coordinate (since they contain at most $w$ disks total) but all $r$ cannot.
This creates an $S^{r-2}$ inside the configuration space, thus the whole
$(k-2)$-cycle is represented by an $S^{r-2} \times T^{k-r}$.

\begin{figure}[h]
  \begin{tikzpicture}[scale=0.6]
    \draw[blue, thick] (15,4) circle (0.5);
    \draw[->] (15.5,4) .. controls (16.25,4) and (16.25,4.25) .. (17,4.25) -- (18,4.25);
    \draw[blue, thick] (17,0) circle (0.5);
    \draw[blue, thick] (17,1) circle (0.5);
    \draw[blue, thick] (17,2) circle (0.5);
    \draw[blue, thick] (15,2.5) circle (0.5);
    \draw[->] (15.5,2.5) .. controls (16.25,2.5) and (16.25,2.75) .. (17,2.75) -- (18,2.75);
    \draw[blue, thick] (17,3.5) circle (0.5);
    \draw[->] (16.5,3.5) .. controls (15.75,3.5) and (15.75,3.25) .. (15,3.25) -- (14,3.25);
    \draw(17,1.5) circle (1);
    \draw(17,1) circle (1.5);
    \draw[->] (15.5,1) -- (14,1);
    \node[blue] at (15,4) {\bf 3};
    \node[blue] at (17,0) {\bf 2};
    \node[blue] at (17,1) {\bf 12};
    \node[blue] at (17,2) {\bf 21};
    \node[blue] at (15,2.5) {\bf 18};
    \node[blue] at (17,3.5) {\bf 15};
    \draw (13,-0.5)--(20,-0.5);
    \draw (13,4.5)--(20,4.5);
  \end{tikzpicture}
  \caption{The wheels in a filter always cross over and under each other in the
    same order.  This figure shows a filter with three wheels of size 1 and one
    wheel of size 3.  The resulting cycle is an $S^2 \times T^2$.}
  \label{fig:filter}
\end{figure}
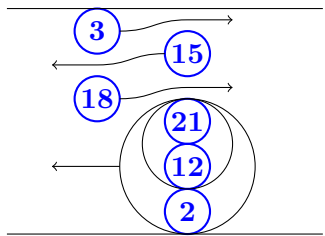

For some purposes, it makes sense to consider filters with $r=2$: such a filter
consists of two wheels $b_1$ and $b_2$ that don't commute, and can be written as
$b_2 \concat b_1-b_1 \concat b_2$.

Our next theorem gives a basis from among the cycles generated in this way.
\begin{thmA} \label{thm:basis}
  $H_*(\config(n,w); \mathbb{Z})$ is free abelian and has a basis consisting of
  concatenations of wheels and filters with $r \geq 2$.  We say one wheel
  \emph{ranks above} another if it has more disks, or has the same number of
  disks and its largest disk label is greater.  A cycle is in the basis if and
  only if:
  \begin{enumerate}[(i)]
  \item Each wheel is ordered so that the largest label comes first.
  \item The wheels inside each filter are in ascending order by largest label
    (regardless of the number of disks).
  \item Adjacent wheels not inside a filter are ordered from higher to lower
    rank.
  \item Every wheel immediately to the left of a filter ranks above the least
    wheel in the filter.
  \end{enumerate}
\end{thmA}

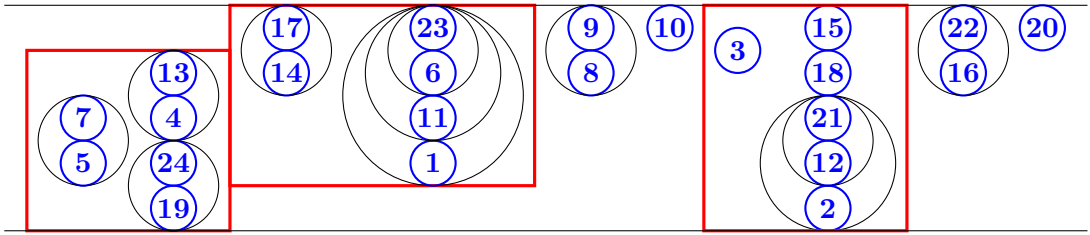
\begin{figure}
  \begin{tikzpicture}[scale=0.6]
    \draw[red, very thick] (-0.75,-0.5)--(3.75,-0.5)--(3.75,3.5)--(-0.75,3.5)--cycle;
    \draw[blue, thick](0.5,1) circle (0.5);
    \draw[blue, thick] (0.5,2) circle (0.5);
    \draw(0.5,1.5) circle (1);
    \draw[blue, thick](2.5,2) circle (0.5);
    \draw[blue, thick] (2.5,3) circle (0.5);
    \draw(2.5,2.5) circle (1);
    \draw[blue, thick](2.5,0) circle (0.5);
    \draw[blue, thick] (2.5,1) circle (0.5);
    \draw(2.5,0.5) circle (1);
    \draw[red, very thick] (3.75,0.5)--(10.5,0.5)--(10.5,4.5)--(3.75,4.5)--cycle;
    \draw[blue, thick] (5,3) circle (0.5);
    \draw[blue, thick] (5,4) circle (0.5);
    \draw(5,3.5) circle (1);
    \draw[blue, thick] (8.25,1) circle (0.5);
    \draw[blue, thick] (8.25,2) circle (0.5);
    \draw[blue, thick] (8.25,3) circle (0.5);
    \draw[blue, thick] (8.25,4) circle (0.5);
    \draw(8.25,3.5) circle (1);
    \draw(8.25,3) circle (1.5);
    \draw(8.25,2.5) circle (2);
    \draw[blue, thick] (11.75,3) circle (0.5);
    \draw[blue, thick] (11.75,4) circle (0.5);
    \draw(11.75,3.5) circle (1);
    \draw[blue, thick] (13.5,4) circle (0.5);
    \draw[blue, thick] (15,3.5) circle (0.5);
    \draw[blue, thick] (17,0) circle (0.5);
    \draw[blue, thick] (17,1) circle (0.5);
    \draw[blue, thick] (17,2) circle (0.5);
    \draw[blue, thick] (17,3) circle (0.5);
    \draw[blue, thick] (17,4) circle (0.5);
    \draw(17,1.5) circle (1);
    \draw(17,1) circle (1.5);
    \draw[red, very thick] (14.25,-0.5)--(18.75,-0.5)--(18.75,4.5)--(14.25,4.5)--cycle;
    \draw[blue, thick] (20,3) circle (0.5);
    \draw[blue, thick] (20,4) circle (0.5);
    \draw(20,3.5) circle (1);
    \draw[blue, thick] (21.75,4) circle (0.5);

    \node[blue] at (0.5,2)  {\bf 7};
    \node[blue] at (0.5,1)  {\bf 5};
    \node[blue] at (2.5,2)   {\bf 4};
    \node[blue] at (2.5,3) {\bf 13};
    \node[blue] at (2.5,1) {\bf 24};
    \node[blue] at (2.5,0) {\bf 19};
    \node[blue] at (5,4) {\bf 17};
    \node[blue] at (5,3) {\bf 14};
    \node[blue] at (8.25,1) {\bf 1};
    \node[blue] at (8.25,2) {\bf 11};
    \node[blue] at (8.25,3) {\bf 6};
    \node[blue] at (8.25,4) {\bf 23};
    \node[blue] at (11.75,3) {\bf 8};
    \node[blue] at (11.75,4) {\bf 9};
    \node[blue] at (13.5,4) {\bf 10};
    \node[blue] at (15,3.5) {\bf 3};
    \node[blue] at (17,0) {\bf 2};
    \node[blue] at (17,1) {\bf 12};
    \node[blue] at (17,2) {\bf 21};
    \node[blue] at (17,3) {\bf 18};
    \node[blue] at (17,4) {\bf 15};
    \node[blue] at (20,3) {\bf 16};
    \node[blue] at (20,4) {\bf 22};
    \node[blue] at (21.75,4) {\bf 20};
    \draw (-1.25,-0.5)--(22.75,-0.5);
    \draw (-1.25,4.5)--(22.75,4.5);
  \end{tikzpicture}
  \captionsetup{singlelinecheck=off}
  \caption[]{%
    A basic 14-cycle in $\config(24,5)$ represented by an
    $S^1 \times S^0 \times S^2 \times T^{12}$: the three red boxes are filters giving an $S^1$,
    an $S^0$ and an $S^2$ respectively, and there are 11 additional circular
    degrees of freedom from spinning the wheels (in black).  We also remark:
    \begin{enumerate}[(a)]
    \item If the disks $\circled{10}$ and $\circled{3}$ switched places, this
      would no longer represent a basic cycle, by property (iv) in Theorem
      \ref{thm:basis}.
    \item The single disk $\circled{10}$ can move freely past all the disks to
      its left.  So in a basic cycle, it has to appear all the way on the right.
    \end{enumerate}
  } \label{14-cycle}
\end{figure}

This combinatorial structure admits a natural interpretation via homotopical
algebra.  Notice that $\config(n,w)$ naturally embeds in $\config(n,w+1)$,
forming a filtration.  The union of this filtration is the classical
configuration space $\config(n)$ of $n$ points in the plane, which turns out to
be homotopy equivalent to $\config(n,n)$.  The algebraic structure of
$H_*(\config(-))$ was considered by Arnol'd~\cite{Arnold69} and
F.~Cohen~\cite[Ch.~III]{CLM}, who showed (in our terms) that it is a
\emph{commutative}, but infinitely generated twisted algebra whose generators are
wheels of all degrees.

For any two wheels in $\config(n)$, we have a choice of commuting them ``over''
or ``under''.  However, if we order all the wheels (for example, in ascending
order by largest label) and make them cross over each other in order, then they
commute in a \emph{homotopy coherent} way: informally, this means that
concatenated cycles can be permuted in any sequence, or all at the same time, and
that all such paths in the space of cycles in $\config(n)$ are
homotopic\footnote{Unfortunately, these choices cannot be made equivariantly with
  respect to relabeling, which underlies the seemingly unavoidable
  non-equivariance of many of our results.}.  Just as filters with two wheels are
commutators of wheels, other filters can then be thought of as nontrivial
``higher commutators'' that obstruct this homotopy coherence.

These higher commutators appear and disappear as we move up the filtration,
increasing $w$, while wheels of size $k$ are born when $w=k$ and stay forever.
This observation can be made precise by considering the filtration's persistent
homology.
\begin{thmA} \label{thm:PH}
  The basic cycles from Theorem \ref{thm:basis} form a $\mathbb{Z}[t]$-basis for
  $PH_*(\config(n,*); \mathbb{Z})$.  Every bar born at time $w$ is either infinite or dies
  by time $2w$.
\end{thmA}

\subsection{Proof ideas}

We analyze $\config(n,w)$ by relating it to a class of simpler spaces.  The
\strong{no-(w+1)-equal space} of $n$ points in $\mathbb{R}$, which we denote
$\no_{w+1}(n,\mathbb{R})$, is the subspace of $\mathbb{R}^n$ in which at most $w$
of the coordinates are the same.  The topology of this space is fairly easy to
understand, although it is related to more complicated questions about hyperplane
arrangements; see \cite{BW95}.  The space $\config(n,w)$ projects onto
$\no_{w+1}(n,\mathbb{R})$ by forgetting the $y$-coordinates of the disks.
Conversely, for every ordering on the numbered disks, there is an injective map
from the no-$(w+1)$-equal space to the configuration space of the strip in which
the disks, when they meet, go around each other ``in order'' from top to bottom.
Each subspace generated in this way is actually a retract of $\config(n,w)$.

However, in many places, two such injections coincide.  For example, if we
transpose two neighboring disks $\circled{a}$ and $\circled{b}$, then the two
injections coincide everywhere except for an open neighborhood of the
codimension-1 subspace of the no-$(w+1)$-equal space where $\circled{a}$ and
$\circled{b}$ coincide.  Abstractly, we can think of this subspace as a
``weighted'' no-$(w+1)$-equal space with $n-1$ symbols, of which one has weight
2.  Write $\no_{w+1}(n-1,\mathcal{W})$ for this space; $\mathcal{W}$ represents
the set of weights of different points.

In the example, the subspace has codimension 1, so its neighborhood looks like
$\no_{w+1}(n-1,\mathcal{W}) \times (0,1)$; we can decompose the union of the
images of the two injections into a copy of
$\no_{w+1}(n-1,\mathcal{W}) \times [0,1]$ glued onto $\no_{w+1}(n)$.  We will see
that we can model $\config(n,w)$ by breaking it up into layers that similarly
look like thickened weighted no-$(w+1)$-equal spaces.

To compute the homology of $\config(n,w)$, we write it as a direct sum of
homology groups of weighted no-$(w+1)$-equal spaces.  We use combinatorics
(specifically, discrete Morse theory) to compute the homology of these spaces.

\subsection{Additional results}

Besides Theorems \ref{thm:FG}, \ref{thm:basis}, and \ref{thm:PH}, the paper
includes a number of other results about $\config(n,w)$ and related spaces.

\subsubsection*{Counting the basis elements}
Theorem~\ref{thm:basis} gives us a way to compute formulas for the Betti numbers of $\config(n, w)$.  We describe a finite computation for each $j$ and $w$ that gives a formula for the rank of $H_j(\config(n, w); \mathbb{Z})$ as a function of $n$, and we show that this function is a sum of products of polynomial and exponential functions.

\subsubsection*{The cohomology of $\config(n,w)$}
We can represent cohomology classes in $H^j(\config(n,w); \mathbb{Z})$ via
Poincar\'e--Lefschetz duality as $(2n-j)$-dimensional compact submanifolds of the
configuration space.  We give a basis for the cohomology, showing that it is a
basis by exploiting the pairing with homology.  The main goal of the section is
to gain some understanding of the cup product in the cohomology ring, which is
fairly complicated and has many indecomposable elements---in contrast with
$H^*(\config(n); \mathbb{Z})$, which is generated as a ring by one-dimensional classes whose
pairing with homology measures the winding of two points around each other, see
\cite{Sinha}.  All higher-dimensional classes are linear combinations of cup
products of these.

On the other hand, in $\config(n,w)$, pairing with cup products often cannot
distinguish between $h \mid h'$ and $h' \mid h$, where $h$ and $h'$ are homology
classes.  In such cases, a cohomology class which pairs nontrivially with the
commutator is perforce indecomposable.  This observation allows us to prove
Theorem \ref{thm:indec}:
\begin{thm*}
  The ring $H^*(\config(n,w); \mathbb{Z})$ has indecomposables:
  \begin{enumerate}[(a)]
  \item Only in degree 1, when $w=2$.
  \item In every degree between $1$ and $\lfloor n/2 \rfloor$ and no others, when
    $w=3$.
  \item In degree 1 and in every degree between $w-1$ and
    $\left\lfloor\frac{n+w-3}{2}\right\rfloor$ and no degree greater than
    $n-\bigl\lceil \frac{n}{w-1} \bigr\rceil$, when $w \geq 4$.
  \end{enumerate}
\end{thm*}

\subsubsection*{Other algebraic structures}
One possible operation which takes $j$-cycles in $\config(n,w)$ to $j$-cycles in
$\config(n+1,w)$ is inserting a singleton disk.  If this operation were
well-defined, then $\config({-},w)$ would be an $\FI$-module, one of the central
objects of study in representation stability.  However, since some cycles do not
commute with singletons, there may be several nonequivalent ways of inserting a
singleton, separated by ``barriers''.  In some cases, if these several ways are
in turn well-defined, this can be formalized via an $\FI_d$-module
structure.  We show that $H_j(\no_{k+1}({-},\mathbb{R}); \mathbb{Z})$ and $H_j(\config({-},2); \mathbb{Z})$
have natural $\FI_d$-module structures for appropriate $d$.  On the other hand,
for $w=3$ we give an example which suggests that the notion of a ``barrier'' is
not well-defined and therefore there is no natural $\FI_d$-module structure on
$H_j(\config({-},w); \mathbb{Z})$ for $w \geq 3$.

Another strategy to pin down explicitly the algebraic structure of
$H_*(\config({-},w); \mathbb{Z})$ is to write down a presentation for the twisted algebra
using generators and relations.  Category theory dictates that generators and
relators in this situation are not just elements but $S_n$-representations for
various $n$.  We write down the generators and relations for
$H_*(\config({-},2); \mathbb{Z})$, but already there is some extra difficulty because of the
failure of Maschke's theorem integrally.  A potential avenue for future research
is to write down relations for $H_*(\config({-},w);\mathbb{Q})$ and explore the
implications for the multiplicity of various $S_n$-representations in
$H_j(\config(n,w);\mathbb{Q})$.

\subsubsection*{Unordered disks}
We also discuss the homology of the configuration space of unordered disks,
$\config(n,w)/S_n$.  Taking this quotient produces a large amount of torsion in
the homology, so instead of trying to write down all the torsion we compute
homology with coefficients in $\mathbb F_p$ and $\mathbb Q$.  (In the bulk of the paper, we use coefficients in $\mathbb{Z}$.)  The concatenation
product can still be defined in this case and gives $H_*(\config({-},w)/S_{-})$
the structure of a bigraded algebra over the base field.  We compute a basis for
the homology and give generators and relations for this algebra.

\subsection*{Structure of the paper}
Section~\ref{sec:prelim} contains preliminaries, including descriptions of the cell complexes we study and the algebraic framework we use in Theorem~\ref{thm:FG}.  Sections~\ref{S:nokequal} and~\ref{sec:layers} prove Theorem~\ref{thm:basis}, with Theorem~\ref{thm:FG} as a consequence; Section~\ref{S:nokequal} addresses the homology of weighted no-$(k+1)$-equal spaces, and Section~\ref{sec:layers} proves its relationship to homology of configuration spaces of disks in a strip.  This is the core technical content of the paper.

The rest of the sections are largely independent of each other and may appeal to
different audiences (e.g.\ Section~\ref{sec:cohom} to those interested in motion
planning, and Section~\ref{S:FId} to experts in representation stability).  Section~\ref{sec:formula} concerns finding a formula for the Betti numbers, by counting the basis elements from Theorem~\ref{thm:basis}.  Section~\ref{sec:cohom} describes aspects of the cup product structure of the cohomology of our configuration spaces.  Section~\ref{sec:PH} proves Theorem~\ref{thm:PH} about how the homology changes with $w$, the width of the strip.  Section~\ref{S:FId} concerns additional algebraic properties, in particular the question of whether our homology groups form $\FI_d$-modules.  Section~\ref{sec:unordered} characterizes the homology with field coefficients in the case of unordered disks.  Finally, Section~\ref{sec:open} lists some questions for further study.

\subsection*{Acknowledgements}  H.A.\ and F.M.\ were supported by the National Science Foundation under Awards No.~DMS-1802914 and DMS-2001042, respectively.  We thank Matt Kahle, Jenny Wilson, and Nick Wawrykow for many helpful conversations about this material, as well as the referee for comments that pushed us to make many sections easier to follow.  F.M.\ also thanks Shmuel Weinberger for the suggestion to look at cup products.

\section{Combinatorial and algebraic setup}\label{sec:prelim}

For the purpose of computation, the paper~\cite{AKM} replaces the configuration
space $\config(n, w)$ by a homotopy-equivalent cell complex $\cel(n, w)$.  We use
the same complex; we also use the same method to find a cell complex
$P(n,\mathcal{W}, k)$ that is homotopy equivalent to the weighted
no-$(k+1)$-equal space $\no_{k+1}(n, \mathcal{W})$.  In the case of ordinary,
unweighted no-$(k+1)$-equal spaces this recovers a result of Bj\"orner in~\cite{Bjo}.

In this section, we define the complexes $\cel(n, w)$ and $P(n,\mathcal{W}, k)$
and describe various algebraic structures on their cells which induce a similar
structure on their homology.  In particular, we introduce twisted algebras and
show that for each $w$, $H_*(\config({-},w))$ and $H_*(\no_{w+1}({-},\mathbb{R}))$
are examples.  In the remainder of the paper, unless otherwise specified, homology and cohomology are always computed with coefficients in $\mathbb{Z}$.

\subsection{The complex $\cel(n)$}
The cell complex $\cel(n, w)$ is defined as a subcomplex of a cell complex
$\cel(n)$ described by~\cite{BZ14}.  The complex $\cel(n)$ is defined in terms of the permutohedron $P(n)$, which is the $(n-1)$-dimensional polytope in $\mathbb{R}^n$ equal to the convex hull of the $n!$ points with coordinates $1, 2, \ldots, n$ in some order.  The faces of $P(n)$ can be labeled by partitions of $\{1, 2, \ldots, n\}$ where the sets of the partition are ordered but the elements of each set are unordered.  We refer to each set of the partition as a \strong{block}.  For instance, each vertex is labeled by a sequence of $n$ singleton blocks, and the top-dimensional face is labeled by the one block $\{1, 2, \ldots, n\}$.  Given any permutation $\sigma \in S_n$, thought of as an ordering on $\{1, 2, \ldots, n\}$, we can order the elements of each block according to $\sigma$.  Then, to write out the label of a given face, we can write out the elements of each block in order, with vertical bars between blocks.  We refer to a label of this form as a \strong{symbol}.  For example, the
following is a symbol with 4 blocks that could come from the ordering $4 \prec 5 \prec 6 \prec 7 \prec 8 \prec 1 \prec 2 \prec 3$:
\[(\,7\;2 \mid 6 \mid 4\;5\;1 \mid 8\;3\,).\]

To define $\cel(n)$, we start with $n!$ copies of $P(n)$, one for each $\sigma \in S_n$.  For the copy of $P(n)$ associated with $\sigma$, we use $\sigma$ to label each face of that copy of $P(n)$ by a symbol.  Then, whenever faces from multiple different copies of $P(n)$ have the same symbol, we identify those faces.  For instance, all copies of $P(n)$ have the same vertices, but each of them has its own distinct top-dimensional face.  In this way, $\cel(n)$ has exactly one cell for every possible symbol on $\{1, 2, \ldots, n\}$.

The incidence relation on cells of $\cel(n)$ can be deduced from the geometry, or can be described explicitly on symbols as follows.  A cell $f$ in $\cel(n)$ is a face of the boundary of a cell $g$ if $g$
can be obtained from $f$ by deleting a bar and \strong{shuffling} the entries of
the two neighboring blocks, preserving the ordering of the entries in each block.
For example, one shuffle of $4\;6\;1 \mid 7\;3\;2$ would be $7\;4\;6\;3\;1\;2$.
A $d$-dimensional cell consists of $n-d$ blocks.

Informally, we think of the elements of each block of a given symbol as the labels of disks in a vertical stack in $\config(n, w)$, as in Figure~\ref{fig-symbol}.  Accordingly, as a way to specify that no more than $w$ disks should be in each vertical stack, we define $\cel(n, w)$
to be the subcomplex of $\cel(n)$ consisting of all cells for which every block
has at most $w$ elements.  

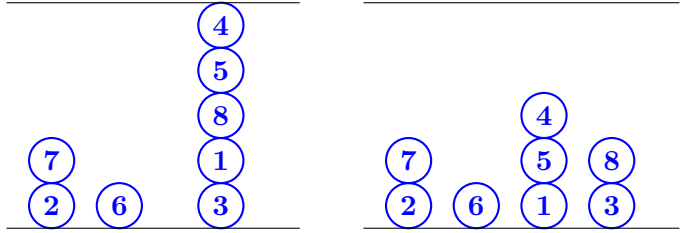
\begin{figure}
\begin{center}
\begin{tikzpicture}[scale=.6]
\node[blue] (a1) at (0, .5) {\bf 2};
\node[blue] (a2) at (0, 1.5) {\bf 7};
\node[blue] (a3) at (1.5, .5) {\bf 6};
\node[blue] (a4) at (3.75, 1.5) {\bf 1};
\node[blue] (a5) at (3.75, 3.5) {\bf 5};
\node[blue] (a6) at (3.75, 4.5) {\bf 4};
\node[blue] (a7) at (3.75, .5) {\bf 3};
\node[blue] (a8) at (3.75, 2.5) {\bf 8};
\draw[blue,thick] (a1) circle (.5);
\draw[blue,thick] (a2) circle (.5);
\draw[blue,thick] (a3) circle (.5);
\draw[blue,thick] (a4) circle (.5);
\draw[blue,thick] (a5) circle (.5);
\draw[blue,thick] (a6) circle (.5);
\draw[blue,thick] (a7) circle (.5);
\draw[blue,thick] (a8) circle (.5);
\draw (-1, 0)--(5.5, 0) (-1, 5)--(5.5, 5);
\end{tikzpicture}\hspace{20pt}
\begin{tikzpicture}[scale=.6, emp/.style={inner sep = 0pt, outer sep = 0pt}, >=stealth]
\node[blue] (a1) at (0, .5) {\bf 2};
\node[blue] (a2) at (0, 1.5) {\bf 7};
\node[blue] (a3) at (1.5, .5) {\bf 6};
\node[blue] (a4) at (3, .5) {\bf 1};
\node[blue] (a5) at (3, 1.5) {\bf 5};
\node[blue] (a6) at (3, 2.5) {\bf 4};
\node[blue] (a7) at (4.5, .5) {\bf 3};
\node[blue] (a8) at (4.5, 1.5) {\bf 8};
\draw[blue,thick] (a1) circle (.5);
\draw[blue,thick] (a2) circle (.5);
\draw[blue,thick] (a3) circle (.5);
\draw[blue,thick] (a4) circle (.5);
\draw[blue,thick] (a5) circle (.5);
\draw[blue,thick] (a6) circle (.5);
\draw[blue,thick] (a7) circle (.5);
\draw[blue,thick] (a8) circle (.5);
\draw (-1, 0)--(6, 0) (-1, 5)--(6, 5);
\end{tikzpicture}
\end{center}
\caption{We can imagine each symbol of $\cel(n, w)$ as a configuration in $\config(n, w)$ where the numbers in each block are the labels in a column of disks.  Pictured are configurations representing the symbol $(\,7\;2 \mid 6 \mid 4\;5\;8\;1\;3\,)$ and its face $(\,7\;2 \mid 6 \mid 4\;5\;1 \mid 8\;3\,)$.}\label{fig-symbol}
\end{figure}

To describe the relationship between $\cel(n, w)$ and $\config(n, w)$, we start by defining $\config(n, w)$ more precisely as the set of configurations of $n$ ordered open disks of diameter $\frac{1}{w}$ in the strip $\mathbb{R} \times (0,1)$.  This is a subspace of the set of configurations of $n$ ordered distinct points in $\mathbb{R} \times (0, 1)$ such that no more than $w$ points are on any vertical line, and Theorem~3.3 of~\cite{AKM} constructs a deformation retraction between the two spaces.  Thus, we abuse notation and use $\config(n, w)$ to mean configurations of points with no more than $w$ vertically aligned.  We use $\config(n)$ to mean configurations of points either in $\mathbb{R} \times (0, 1)$ or in $\mathbb{R}^2$, not distinguishing between these homotopy equivalent spaces.

\begin{thm} \label{thm:emb}
There is an affine embedding of the barycentric subdivision of $\cel(n)$ into $\config(n)$, such that for each $w$, the restriction to $\cel(n, w)$ maps into $\config(n, w)$ and is a homotopy equivalence.
\end{thm}

\begin{proof}
We use the following description of $\config(n)$ in coordinates:
\[\config(n) = \{(x_1, \ldots, x_n, y_1, \ldots, y_n) \in \mathbb{R}^{n} \times (0, 1)^n \ \vert\ (x_i, y_i) \neq (x_j, y_j) \mathrm{\ if\  } i \neq j\}.\]
To define the map on a point $p$ in $\cel(n)$, we need to specify $x$--coordinates and $y$--coordinates, as well as check the condition $(x_i, y_i) \neq (x_j, y_j)$.  

Let $p$ be an arbitrary point in $\cel(n)$.  It is in at least one of the permutohedra $P(n)$ that comprise $\cel(n)$, and $P(n)$ is embedded in $\mathbb{R}^n$, so in this sense $p$ has coordinates in $\mathbb{R}^n$.  We set the $x$--coordinates of the image of $p$ to be these coordinates of $p$ in $\mathbb{R}^n$.  

For the $y$--coordinates, we start with the case where $p$ is the barycenter of a cell in $\cel(n)$.  To set each $y_i$, we find the location of the number $i$ in the symbol of the cell of $p$.  If $i$ appears as the $\ell$th element of a block of size $k$, we set $y_i$ to be $\frac{k-\ell+1}{n+1}$.  In other words, for each block of size $k$ we set the $y$--coordinates of the points in the block, in order, to be $\frac{k}{n+1}, \frac{k-1}{n+1}, \ldots, \frac{1}{n+1}$.  To assign $y$--coordinates when $p$ is not a barycenter, we extend the map so that its restriction to each simplex of the barycentric subdivision of $\cel(n)$ is affine.

Note that in any given cell, away from the barycenter, the $y$--coordinates of the points in each block remain in order.  This is because in the closure of our cell, blocks may merge but may not separate, and when merging, the elements from each smaller block remain in order.  This implies that our map is injective, because different points with the same $x$--coordinates are distinguished by their $y$--coordinates, which in each block appear in the same order as in the corresponding symbol from $\cel(n)$.

To check that the resulting map lands in $\config(n)$, suppose that we have a point $p$ for which $x_i = x_j$.  The permutohedron coordinates imply that in the symbol of the cell of $p$, the numbers $i$ and $j$ are in the same block, and so the note in the previous paragraph implies that $y_i \neq y_j$.  Thus the image of $p$ is in $\config(n)$.

To show that the map is a homotopy equivalence, for each symbol $\alpha$ from $\cel(n)$, as in~\cite{AKM} we define the corresponding open set $U_\alpha \subseteq \config(n)$ to be the set of points $(x_1, \ldots, x_n, y_1, \ldots, y_n)$ in $\mathbb{R}^n \times (0, 1)^n$ such that
\begin{itemize}
\item Whenever $i$ appears before $j$ in the same block, we have $y_i > y_j$.
\item Whenever $i$ appears before $j$ in different blocks, we have $x_i < x_j$.
\item If $k$ and $\ell$ are in the same block, and $k'$ and $\ell'$ are in different blocks (with $k', \ell'$ not necessarily distinct from $k, \ell$), then we have
\[\abs{x_k - x_\ell} < \abs{x_{k'} - x_{\ell'}}.\]
\end{itemize}
We claim that the union of $U_\alpha$ where $\alpha$ ranges over the symbols from $\cel(n, w)$ is $\config(n, w)$.  If $\alpha$ is a symbol of $\cel(n, w)$, then $U_\alpha$ is contained in $\config(n, w)$ because points from different blocks have different $x$--coordinates, so no more than $w$ points can be vertically aligned.  For the reverse inclusion, every element of $\config(n, w)$ is in some $U_\alpha$, because we can construct $\alpha$ by taking the blocks to be the sets of points with the same $y$--coordinate, ordering the blocks from left to right, and ordering the elements of each block from top to bottom.

Each $U_\alpha$ is an open convex set, and Theorem~3.4 of~\cite{AKM} proves that the nerve of this open cover of $\config(n, w)$ is the barycentric subdivision of $\cel(n, w)$. Thus, because our map sends the barycenter of each cell $\alpha$ into a point of the corresponding open set $U_\alpha$, our map is a homotopy equivalence for each $w$.
\end{proof}

\subsection{Signs of the boundary operator} \label{S:boundary}
In order to study the integral cellular chain complex of $\cel(n,w)$, we need to
specify orientations on cells and signs for the boundary operator.  To describe
these, we first generalize slightly.  Observe that the entries of a symbol don't
have to be the numbers $1$ through $n$, but can be any $n$-element set.  So for
any finite set $A$, we have a complex $\cel(A)$.  (Similarly, we will sometimes
write $P(A)$ for the permutohedron whose coordinates are indexed by elements of
$A$.)  Moreover, there is a cellular map
\[{\concat}:\cel(A) \times \cel(B) \to \cel(A \sqcup B)\]
which takes a pair of symbols to the symbol obtained by putting them next to each
other with a bar in between.  We call this the \strong{concatenation product}.

Any cell in $\cel(n)$ is either top-dimensional or a concatenation product.  We
define the boundary operator on a top-dimensional cell $g$ by taking the
coefficient of a cell $f=a \mid b$ in $\partial g$ to be
\[(-1)^{\length(a)}\cdot\sgn(\text{permutation }g \mapsto ab).\]
On a cell $g_1 \concat g_2$, we define $\partial$ via a Leibniz rule:
\begin{equation} \label{eq:Leibniz}
  \partial(g_1 \concat g_2)=\partial g_1 \concat g_2+(-1)^{\dim(g_1)}g_1 \concat \partial g_2.
\end{equation}
This defines an injective chain complex homomorphism on cellular chains,
\[{\concat}:C_*(\cel(A)) \otimes C_*(\cel(B)) \to C_*(\cel(A \sqcup B)),\]
using the standard tensor product on chain complexes, where the differential is
defined by
\[\partial(a \otimes b)=\partial a \otimes b+(-1)^{\deg a}a \otimes \partial b.\]
\begin{prop}
  \begin{enumerate}[(a)]
  \item The boundary operator defined above satisfies $\partial^2=0$.
  \item The $S_n$-action on $\cel(n)$ induced by permutations of $[n]$ preserves
    orientations of cells.
  \end{enumerate}
\end{prop}
These two features will allow us to define a twisted algebra structure on
$H_*(\cel({-}))$.
\begin{proof}
  Let $g$ be a cell in $\cel(n)$.

  First, suppose that $g$ is top-dimensional, and let $e$ be a codimension-$2$
  face of $g$.  Then $e=e_1 \mid e_2 \mid e_3$, where $e_1$, $e_2$, and
  $e_3$ are blocks.  There are two intermediate faces between $e$ and $g$, which
  we denote by $f = b \mid e_3$ and $f' = e_1 \mid b'$.  We compare
  \[\sgn(e\text{ in }\partial f)\cdot\sgn(f\text{ in }\partial g)\quad\text{and}
  \quad\sgn(e\text{ in }\partial f')\cdot\sgn(f'\text{ in }\partial g),\]
  and we show that these two products are opposite signs.  This will show that
  $\del^2 g=0$.

  If we consider just the contribution from the signs of the permutations, both
  products give the sign of the permutation relating $e$ and $g$, so those
  contributions are equal.  For the contribution from the Leibniz rule, only the
  incidence between $e$ and $f'$ involves splitting a block that is not the
  first, so that incidence has a sign contribution of
  \[(-1)^{\dim(\,e_1\,)}=(-1)^{\length(e_1)-1}\]
  from the Leibniz rule, and all the other incidences have a sign contribution of
  $1$ from the Leibniz rule.  Finally, the contribution from the length of the
  first block gives
  \[(-1)^{\length(b)+\length(e_1)}=(-1)^{\length(e_2)}\]
  for the path through $f$ and $(-1)^{\length(e_1)+\length(e_2)}$ for the path through
  $f'$.  Taking the product of all of these, we see that the two paths give
  opposite signs.

  If $g$ is not top-dimensional, then $g$ is a concatenation product $g_1 \vert g_2$, in which case we use a standard argument.  The Leibniz rule gives
  \[\del^2(g_1 \concat g_2) = \del^2 g_1 \concat g_2 + \bigl[(-1)^{\dim(g_1)} + (-1)^{\dim(\del g_1)}\bigr]\cdot \del g_1 \concat \del g_2 + g_1 \concat \del^2 g_2,\]
  which is zero by induction on the number of blocks of $g$.  This proves (a).

  For (b), notice that the definitions of the signs of the boundary operator do
  not use any particular ordering on the numbers $1$ through $n$.  Therefore they
  are invariant with respect to permutations.  Since the $S_n$-action preserves
  signs of 0-cells, it preserves signs of all cells.
\end{proof}

\subsection{Twisted algebra structure} \label{S:twist}
Let $\FB$ be the category of finite sets and bijective maps.  Then $\cel({-})$ is
a functor from $\FB$ to the category of cell complexes and cellwise maps
(an \strong{$\FB$-complex}, for short).  (This is just a categorical way of
saying that there is a cellular $S_n$-action on $\cel(n)$.)  In particular, the
cellular chains $C_*(\cel({-}))$ form an $\FB$-chain complex.  Moreover, we can
define a tensor product (the \strong{Day convolution}) on $\FB$-objects in a
monoidal category by
\[(F \otimes G)(S)=\bigoplus_{A \sqcup B=S} F(A) \otimes F(B).\]
A unital monoid object with respect to this tensor product is called a
\strong{twisted algebra} (see \cite{SamSno} for a detailed discussion of twisted
\emph{commutative} algebras).  In plain English, a twisted algebra is a family of
objects $A_n, n=0,1,\ldots$ (e.g.~abelian groups) equipped with the
following structure:
\begin{itemize}
\item Each $A_n$ is equipped with an $S_n$-action.
\item For every partition of $\{1,\ldots,n\}$ into subsets of size $i$ and $j$,
  there is a ``multiplication'' $A_i \otimes A_j \to A_n$.  For different
  partitions, these multiplications commute with the $S_n$-action on $A_n$.
\item There is a unit in $A_0$ such that multiplying by it induces the identity
  map on $A_n$.
\end{itemize}

The observations about the differential in \S\ref{S:boundary} show that the
concatenation product
\[{\concat}:C_i(\cel(A)) \otimes C_j(\cel(B)) \to C_{i+j}(\cel(A \sqcup B))\]
makes $C_*(\cel({-}))$ into a (noncommutative) differential graded twisted
algebra or \strong{dgta}, whose unit is the unique 0-cell $(\,)$ in
$C_*(\cel(\emptyset))$.  (By definition, a dgta is simply a twisted algebra in
the category of chain complexes.  The fact that the multiplication maps
$A_i \otimes A_j \to A_n$ are chain maps forces the differential to satisfy a
Leibniz rule such as \eqref{eq:Leibniz}.) The homology of a dgta naturally forms
a graded twisted algebra.  The graded twisted algebra $H_*(\cel({-}))$ is
well-understood since the work of F.~Cohen in the 1970s and is in fact
commutative; see e.g.\ \cite[Theorem 3.4]{MilWil}.

We are most interested in the subcomplex $\cel(n,w)$ consisting of cells whose
every block has size at most $w$.  Since the concatenation product of two such
cells again has the same property, $C_*(\cel({-},w))$ is a sub-dgta of
$C_*(\cel({-}))$.  Our goal in this paper is to understand the graded twisted
algebra $H_*(\cel({-},w))$, and in particular to show that it is finitely
generated.

\subsection{The permutohedron and no-$(k+1)$-equal spaces}
The difference between $\cel(n)$ and $P(n)$ is that in $\cel(n)$ the numbers
within a block are ordered and in $P(n)$ they are not.  This gives a natural
projection $\cel(n) \to P(n)$ (forget the ordering of entries inside each block)
and, for every global ordering of $1,\ldots,n$, an injective map
$P(n) \to \cel(n)$ (arrange the entries in each block in the given order).

$P(n)$ is a polytope, and is in particular contractible.  As with $\cel(n)$, we
can filter $P(n)$ by the largest size of a block, producing a sequence of
complexes $P(n,k)$.  These are homotopy equivalent to the no-$(k+1)$-equal space
of $n$ points in $\mathbb{R}$, as pointed out by Bj\"orner in
\cite[Theorem 2.4]{Bjo}; their homology was computed first by Bj\"orner and
Welker in \cite{BW95}.  As with $\cel(n,w)$, $C_*(P({-},k))$ naturally has a dgta
structure which induces a graded twisted algebra structure on $H_*(P({-},k))$.
From the results of \cite{BW95}, one sees that this is finitely generated, in
fact just has two generators: one in degree $0$ (a point) and one in degree $k-1$
(the boundary of a $P(k+1)$).  We recover this, together with a set of relations,
in \S\ref{S:FId}.

However, we are interested in a somewhat more complicated structure.  Let $\FBW$
be the category of \emph{weighted} finite sets and weight-preserving bijections.
That is, every element is associated with a natural number which is its weight.
Then given a weighted set $(A,\mathcal{W} \in \mathbb{N}^A)$, there is a complex
$P(A,\mathcal{W},k)$ which consists of all the cells for which the sum total
weight of every block is at most $k$.  The following generalization of
\cite[Theorem 2.4]{Bjo} follows by the same argument:
\begin{thm}
  The complex $P(A,\mathcal{W},k)$ is homotopy equivalent to the
  \strong{weighted no-$(k+1)$-equal space} $\no_{k+1}(A,\mathcal{W})$ of $|A|$
  points in $\mathbb{R}$ with weights $\mathcal{W}$, that is, the space of
  configurations of $|A|$ points in $\mathbb{R}$ such that no set of coincident
  points has total weight greater than $k$.
\end{thm}

\begin{rmk}
The functor $C_*(P({-},k))$ is an $\FBW$-chain complex, and in fact an
$\FBW$-dga.  That is, we can define a tensor product on $\FBW$-objects in a
monoidal category by
\[(F \otimes G)(S,\mathcal{W})=
\bigoplus_{(A,\mathcal{W}_A) \sqcup (B,\mathcal{W}_B)=(S,\mathcal{W})}
F(A,\mathcal{W}_A) \otimes F(B,\mathcal{W}_B),\]
and the concatenation product on $C_*(P({-},k))$ then makes it into a monoid
object.  This in turn makes $H_*(P({-},k))$ into a graded $\FBW$-algebra.  Once
one makes all this precise, Theorem \ref{thm-weighted-basis} can be interpreted
as showing that this algebra is finitely generated for every $k$, analogously to
our results about $\cel({-},w)$.
\end{rmk}

\subsection{Generators and relations for twisted algebras}

The above discussion deduces the following facts:
\begin{thm}
  The sequences of graded abelian groups $H_*(\cel({-}))$, $H_*(\cel({-},w))$,
  and $H_*(P({-},k))$ admit the structure of graded twisted algebras.
\end{thm}
To demonstrate Theorem \ref{thm:FG}, we will show that wheels and filters form a
finite generating set for $H_*(\cel({-},w))$.  Later we will also give
presentations of $H_*(\cel({-},2))$ and $H_*(P({-},k))$ by generators and
relations.  To make this precise, we define a \strong{free twisted algebra}
functor $F_\tau$ from FB-modules to twisted algebras as the left adjoint to the
forgetful functor $U_\tau$ from twisted algebras to FB-modules; such a functor
always exists for monoids in a reasonable monoidal category
\cite[VII.3 Thm.~2]{MacLane}.  Informally, a basis element for a free twisted
algebra on a set $\{V_i\}$ of representations of various $S_{n_i}$ is specified by
a list of basis vectors of the various $V_i$, each labeled by a set of $n_i$
labels; one takes products by concatenating these lists and retaining the labels.

If $A$ is a twisted algebra, we say an FB-submodule $G \subset A$
\strong{generates} $A$ if the induced morphism $F_\tau G \to A$ is surjective;
$A$ is \strong{finitely generated} if it has a finite-dimensional generating
module $G$.  A \strong{presentation} of a twisted algebra $A$ by generators and
relations formally consists of a coequalizer diagram
\[F_\tau R \overset{r}{\underset{0}{\rightrightarrows}} F_\tau G \to A,\]
where $G$ and $R$ are FB-modules of generators and relations, respectively, and
$r:F_\tau R \to F_\tau G$ is an FB-module homomorphism \cite[\S5.4]{Riehl}.  By
the adjunction, it suffices to provide a homomorphism $R \to U_\tau F_\tau G$,
i.e.~to describe the relators of $A$ as linear combinations of words in $G$.

Informally, to prove finite generation, it's enough to provide a finite number of
elements whose closure under multiplication and $S_n$-action is all of $A$.
However, to provide a full description of a presentation, one must also
understand the $S_n$-action on the generators.  To show that $A$ is presented by
a generating set $G$ with relations $R$, it is enough to show that $A$ is
generated by $G$ and that every product of elements of $G$ can be reduced to a
basis element of $A$ via the relators.

\section{Homology of weighted no-$(k+1)$-equal spaces} \label{S:nokequal}

In the previous section, we showed that the space $\no_{k+1}(n,\mathcal{W})$
retracts to a subcomplex $P(n,\mathcal{W},k)$ of the permutohedron $P(n)$.  To
compute the homology of $P(n,\mathcal{W},k)$, we use a discrete gradient vector field on
$P(n)$.

\subsection{Discrete Morse theory}
In any cell complex, the cellular homology comes from a chain complex generated
by the cells; very broadly, discrete Morse theory gives a way to decompose the
chain complex as a direct sum of a chain complex that has no homology (which we
discard) and a chain complex generated by a smaller subset of cells, the critical
cells.  To compute the homology exactly, we need to:
\begin{enumerate}
\item Reduce to the smaller chain complex.
\item Show that the differentials in the smaller chain complex are all zero, so
  that the homology has a $\mathbb{Z}$--basis in bijection with the set of
  critical cells.
\end{enumerate}

The basic definitions in discrete Morse theory are as follows.  In any polyhedral cell complex, we say that cell $f$ is a \strong{face} of cell $g$ if $f$ is in the boundary of $g$ and $\dim f = \dim g - 1$, and we say that $g$ is a \strong{coface} of $f$ if $f$ is a face of $g$.  A \strong{discrete vector field} on a polyhedral cell complex is a set $V$ of pairs of cells $[f, g]$ such that $f$ is a face of $g$ and each cell can be in at most one pair; an example is shown in Figure~\ref{fig-discrete-morse}.  A discrete vector field $V$ is \strong{gradient} if there are no closed $V$--walks.  A $V$--walk is a sequence of pairs $[f_1, g_1], \ldots, [f_r, g_r]$ with $[f_i, g_i] \in V$, such that each $f_{i+1}$ is a face of $g_i$ other than $f_i$.  The $V$--walk is closed if $f_r = f_1$.

\begin{figure}
\begin{center}
\begin{tikzpicture}[scale = 2, emp/.style={inner sep = 0pt, outer sep = 0pt}, vert/.style={circle, draw=black, fill=black, inner sep = 0pt, minimum size = .5mm}]
\draw[draw=gray!50, fill=gray!50] (1.06495000000000, 0.0375000000000000) -- (0.877450000000000, 0.362250000000000) -- (0.942400000000000, 0.399750000000000) -- (0.750000000000000, 0.433000000000000) -- (0.682600000000000, 0.249750000000000) -- (0.747550000000000, 0.287250000000000) -- (0.935050000000000, -0.0375000000000000) --cycle;
\draw[draw=gray!50, fill=gray!50] (0.500000000000000, 0.941000000000000) -- (0.125000000000000, 0.941000000000000) -- (0.125000000000000, 1.01600000000000) -- (0.000000000000000, 0.866000000000000) -- (0.125000000000000, 0.716000000000000) -- (0.125000000000000, 0.791000000000000) -- (0.500000000000000, 0.791000000000000) --cycle;
\draw[draw=gray!50, fill=gray!50] (-0.564950000000000, 0.903500000000000) -- (-0.752450000000000, 0.578750000000000) -- (-0.817400000000000, 0.616250000000000) -- (-0.750000000000000, 0.433000000000000) -- (-0.557600000000000, 0.466250000000000) -- (-0.622550000000000, 0.503750000000000) -- (-0.435050000000000, 0.828500000000000) --cycle;
\draw[draw=gray!50, fill=gray!50] (0.500000000000000, -0.791000000000000) -- (0.125000000000000, -0.791000000000000) -- (0.125000000000000, -0.716000000000000) -- (0.000000000000000, -0.866000000000000) -- (0.125000000000000, -1.01600000000000) -- (0.125000000000000, -0.941000000000000) -- (0.500000000000000, -0.941000000000000) --cycle;
\draw[draw=gray!50, fill=gray!50] (-0.435050000000000, -0.828500000000000) -- (-0.622550000000000, -0.503750000000000) -- (-0.557600000000000, -0.466250000000000) -- (-0.750000000000000, -0.433000000000000) -- (-0.817400000000000, -0.616250000000000) -- (-0.752450000000000, -0.578750000000000) -- (-0.564950000000000, -0.903500000000000) --cycle;
\draw (1, 0)--(.5, .866)--(-.5, .866)--(-1, 0)--(-.5, -.866)--(.5, -.866)--(1, 0);
\node[vert] at (1, 0) [label=right:{$1\vert 2 \vert 3$}] {};
\node[vert] at (.5, .866) [label=above:{$2\vert 1 \vert 3$}] {};
\node[vert] at (-.5, .866) [label=above:{$2\vert 3 \vert 1$}] {};
\node[vert] at (-1, 0) [label=left:{$3\vert 2 \vert 1$}] {};
\node[vert] at (-.5, -.866) [label=below:{$3\vert 1 \vert 2$}] {};
\node[vert] at (.5, -.866) [label=below:{$1\vert 3 \vert 2$}] {};
\node[emp] at (.75, .433) [label=left: {$21\vert 3$}] {};
\node[emp] at (0, .866) [label=below: {$2\vert 31$}] {};
\node[emp] at (-.75, .433) [label=right: {$32\vert 1$}] {};
\node[emp] at (-.75, -.433) [label=right: {$3\vert 21$}] {};
\node[emp] at (0, -.866) [label=above: {$31\vert 2$}] {};
\node[emp] at (.75, -.433) [label=left: {$1\vert 32$}] {};
\node at (2, 0) {$\simeq$};
\node[vert] at (3, 0) [label=left: {$3\vert 2\vert 1$}] {};
\draw (3, 0) to [out=45, in=90] (3.5, 0);
\draw (3, 0) to [out=-45, in=-90] (3.5, 0);
\node[emp] at (3.5, 0) [label=right: {$1\vert 32$}] {};
\end{tikzpicture}
\end{center}
\caption{A \strong{discrete gradient vector field} consists of a set of disjoint pairs of cells, each pair incident and of consecutive dimensions.  The complex is homotopy equivalent to one in which the paired cells are collapsed, and only the \strong{critical} (unpaired) cells remain.}\label{fig-discrete-morse}
\end{figure}
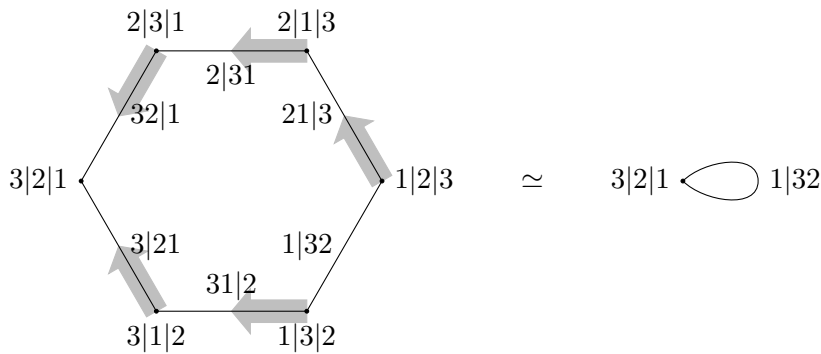

A cell is \strong{critical} with respect to a discrete gradient vector field $V$ if the cell is not in any pair in $V$.  The fundamental theorem of discrete Morse theory~\cite{Forman02} states that there is a cell complex that is a strong deformation retraction of the original cell complex, in which there is one cell per critical cell of $V$.  Thus, we can compute the homology groups $H_j(P(n, \mathcal{W}, k))$ by defining discrete gradient vector fields and computing the homology of the collapsed chain complexes generated by the critical cells.

Our cell complexes $\cel(n, w)$ are not polyhedral cell complexes because, for instance, there are distinct cells with the same boundary.  However, discrete Morse theory still gives an isomorphism on homology between the original chain complex and the collapsed chain complex, as long as the following property holds: for each pair $[f, g]$ of the discrete gradient vector field, the coefficient of $f$ in the boundary of $g$ is a unit in our coefficient ring, in this case $\pm 1$ because we are using coefficients in $\mathbb{Z}$.  In $\cel(n, w)$, every coefficient in every boundary map is $\pm 1$, so this property holds automatically.  Alternatively, the barycentric subdivision of $\cel(n, w)$ is a polyhedral cell complex, so our arguments could be adapted to work with the barycentric subdivision instead.  Thus, we proceed with the discrete Morse theory as if $\cel(n, w)$ were a polyhedral cell complex.

One way to define a discrete gradient vector field on a cell complex is by defining a total ordering on all the cells.  Given a total ordering, the resulting vector field contains a pair $[f, g]$ if and only if both $f$ is the greatest face of $g$ and $g$ is the least coface of $f$ (and for non-polyhedral complexes, we also require the coefficient of $f$ in the boundary of $g$ to be a unit).  One can prove that this vector field is gradient (see Remark~3.7 of~\cite{Bauer19}).

One advantage of doing the construction in this way is that there is a
simple criterion guaranteeing that the discrete gradient vector field forms a perfect Morse
function.  The term ``perfect Morse function'' refers to the case where, on the collapsed chain complex generated by the critical cells, the differential is zero, so the critical cells form a basis for the homology of the
complex.  The following general lemma shows that it suffices to construct, for
each critical cell $e$, a cycle $z(e)$ such that $e$ is its maximum cell and has
coefficient equal to a unit in the coefficient ring.  It turns out that such cycles automatically represent
linearly independent homology classes.  In our main theorems we use the lemma with $\mathbb{Z}$ coefficients, and in Section~\ref{sec:unordered} we use it with field coefficients $\mathbb{Q}$ and $\mathbb{F}_p$.

\begin{lem}\label{lem-max-basis}
  Let $X$ be any finite cell complex with a total ordering on the
  cells, giving a discrete gradient vector field.  Let $R$ be a ring of coefficients.  Suppose there is a collection of cellular cycles with the following properties:
  \begin{itemize}
  \item The cycles in our collection are in bijection with the critical cells of the discrete gradient vector field.  For each critical cell $e$, we denote the corresponding cycle by $z(e)$.
  \item Under the total ordering, the greatest of all the cells appearing with nonzero coefficient in the cellular chain $z(e)$ is the cell $e$.
  \item The coefficient of $e$ in the chain $z(e)$ is a unit in $R$.
  \end{itemize}
Then every homology class in $H_*(X; R)$ can be written uniquely as an
  $R$-linear combination of the homology classes of the cycles $z(e)$.
\end{lem}

\begin{proof}
  For any pair $[f, g]$ in the discrete vector field, we refer to $f$ as a
  ``match-up cell'' and refer to $g$ as a ``match-down cell''.  We also define
  $z'(f)$ to be the boundary of $g$; we know that $f$ is the greatest cell
  appearing in $z'(f)$, and that it has unit coefficient because of how we have defined the discrete vector field from the ordering.

  First, we show that every $j$-cycle $z$ is an $R$-linear combination
  of cycles $z(e)$ and $z'(f)$, where $e$ ranges over the critical $j$-cells and
  $f$ ranges over the match-up $j$-cells.  This follows from the following
  observation: if a match-down cell $g$ is the greatest cell in a $j$-chain,
  then in the boundary of that chain, the corresponding match-up cell $f$ appears
  with nonzero coefficient, because $g$ is the least coface of $f$, so no other
  cell in the chain has $f$ as a face.  Thus, for any $j$-cycle $z$, the
  greatest cell of $z$ cannot be a match-down cell.  It is either a critical cell
  $e$ or a match-up cell $f$, so we subtract the appropriate multiple of $z(e)$
  or $z'(f)$ to get a new cycle with lesser maximum.  Repeating this process
  gives us $z$ as a linear combination of cycles $z(e)$ and $z'(f)$, so because
  each $z'(f)$ is a boundary, this implies that $z$ is homologous to a linear
  combination of the cycles $z(e)$ only.

  To show the uniqueness, we need to show that no nontrivial linear combination
  of cycles $z(e)$ is null-homologous.  Because the cycles $z(e)$ and $z'(f)$
  have distinct maxima, they are linearly independent.  Thus, it suffices to show
  that if a $j$-cycle $z$ is a boundary, it is a linear combination of the
  boundaries $z'(f)$.  To see this, we look at the set of all $j+1$-chains.  The
  chains $z(e)$, $z'(f)$, and $g$ (as $e$ ranges over all critical
  $(j+1)$-cells, $f$ ranges over all match-up $(j+1)$-cells, and $g$ ranges
  over all match-up $(j+1)$-cells) form an $R$-basis for the set of all
  $(j+1)$-chains, because they have distinct maxima equal to the set of all
  $j$-cells.  When we apply the boundary map to this basis, the cycles $z(e)$
  and $z'(f)$ map to zero, and the match-down cells $g$ map to the
  $j$-dimensional boundaries $z'(f)$.  Thus indeed every $j$-dimensional
  boundary is a linear combination of these boundaries $z'(f)$.

  Thus, every homology class in $H_*(X)$ can be written as an $R$-linear
  combination of the homology classes of the cycles $z(e)$, and the combination
  is unique.
\end{proof}

\subsection{Discrete gradient vector fields on permutohedra}
In what follows, we define a total ordering on all of $P(n)$, the polyhedral complex that contains $P(n, \mathcal{W}, k)$ as a subcomplex.  We use the resulting discrete gradient vector fields to compute the homology, by analyzing the critical cells and constructing cycles dual to each one.

Recall that the cells of $P(n)$ are in bijection with ordered partitions of $[n]$
into blocks.  We say the \strong{weight} of a block is the sum of the weights of
its elements.  A block is a \strong{singleton} if it only has one element.
We assign some blocks to \strong{leader--follower pairs} by walking left to right.
A block is a \strong{follower} if it follows a singleton (its \strong{leader}) whose
element is smaller than any of the follower's elements and which is not itself a
follower.  Then a total ordering $\prec$ on cells with symbols $f$ and $g$ is
given by looking at the first block at which they differ.  Let's call the two
blocks $f_i$ and $g_i$.  Then the ordering is given as follows:

\begin{enumerate}[(i)]
\item If $f_i$ is a follower and $g_i$ is not, then $f \prec g$.
\item If $f_i$ and $g_i$ are both followers, then they are ordered by number of
  elements, and then arbitrarily (e.g.~in lexicographic order) if they have the
  same number of elements.
\item If $f_i$ and $g_i$ are not followers, then we order their elements from
  largest to smallest; they are then ordered lexicographically, but according to
  a backwards order on the ``alphabet''.  That is,
  \[3 \prec 3\;2 \prec 3\;2\;1 \prec 3\;1 \prec 2 \prec 2\;1 \prec 1.\]
\end{enumerate}

\begin{lem} \label{lem:crit}
  A cell in $P(n,\mathcal{W},k)$ is critical if and only if every block is either
  \begin{enumerate}[(a)]
  \item A singleton which is not a follower; or,
  \item A follower such that its weight and that of its leader add up to at
    least $k+1$.
  \end{enumerate}
\end{lem}

\begin{figure}
  \centering
  \begin{tikzpicture}[scale=0.6]
    \draw[red, very thick] (-0.75,-0.5)--(3.75,-0.5)--(3.75,3.5)--(-0.75,3.5)--cycle;
    \draw[blue, thick] (0.5,1.5) circle (1);
    \draw[blue, thick] (2.5,2.5) circle (1);
    \draw[blue, thick] (2.5,0.5) circle (1);
    \draw[red, very thick] (3.75,0.5)--(10.5,0.5)--(10.5,4.5)--(3.75,4.5)--cycle;
    \draw[blue, thick] (5,3.5) circle (1);
    \draw[blue, thick] (8.25,2.5) circle (2);
    \draw[blue, thick] (11.75,3.5) circle (1);
    \draw[blue, thick] (13.5,4) circle (0.5);
    \draw[blue, thick] (15,3.5) circle (0.5);
    \draw[blue, thick] (17,3) circle (0.5);
    \draw[blue, thick] (17,4) circle (0.5);
    \draw[blue, thick] (17,1) circle (1.5);
    \draw[red, very thick] (14.25,-0.5)--(18.75,-0.5)--(18.75,4.5)--(14.25,4.5)--cycle;
    \draw[blue, thick] (20,3.5) circle (1);
    \draw[blue, thick] (21.75,4) circle (0.5);

    \node[blue] at (0.5,1.85)  {\bf 6};
    \node at (0.5,1.15)  {(2)};
    \node at (2.5,2.15)   {(2)};
    \node[blue] at (2.5,2.85) {\bf 8};
    \node[blue] at (2.5,0.85) {\bf 11};
    \node at (2.5,0.15) {(2)};
    \node[blue] at (5,3.85) {\bf 9};
    \node at (5,3.15) {(2)};
    \node at (8.25,2.15) {(4)};
    \node[blue] at (8.25,2.85) {\bf 13};
    \node at (11.75,3.15) {(2)};
    \node[blue] at (11.75,3.85) {\bf 7};
    \node[blue] at (13.5,4) {\bf 5};
    \node[blue] at (15,3.5) {\bf 1};
    \node at (17,0.65) {(3)};
    \node[blue] at (17,1.35) {\bf 12};
    \node[blue] at (17,3) {\bf 4};
    \node[blue] at (17,4) {\bf 3};
    \node at (20,3.15) {(2)};
    \node[blue] at (20,3.85) {\bf 10};
    \node[blue] at (21.75,4) {\bf 2};
    \draw (-1.25,-0.5)--(22.75,-0.5);
    \draw (-1.25,4.5)--(22.75,4.5);
  \end{tikzpicture}
  \captionsetup{singlelinecheck=off}
  \caption[]{%
    The 3-cycle in $\no_6(13,\mathcal{W})$, with weights indicated by circle size
    and in parentheses, corresponding to the critical cell
    \[(6 \mid 8\;11 \mid 9 \mid 13 \mid 7 \mid 5
    \mid 1 \mid 3\;4\;12 \mid 10 \mid 2).\]
    The red boxes enclose leader--follower pairs and indicate the boundaries of a
    $P(3)$, a $P(2)$, and a $P(4)$ respectively; thus the cycle as a whole is
    represented by an $S^1 \times S^0 \times S^2$.
  } \label{nokcycle}
\end{figure}
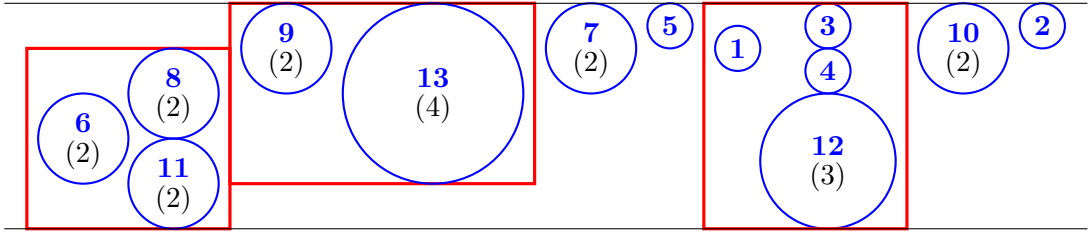
\begin{proof}
 We claim that the matching of the other cells is as follows.  We look at the first block
  where a given cell does not match the characterization of critical given above.  There are two
  possibilities:
  \begin{itemize}
  \item The block is a follower and, together with its leader, has total weight
    at most $k$.  In that case, the cell matches \emph{up} to the cell in which
    the follower and its leader are combined.
  \item The block is not a follower or a singleton; if preceded by a non-follower
    singleton, it has at least one element smaller than that singleton.  Then the
    cell matches \emph{down} to the cell where the least element of the block is
    split off into its own block which comes before the rest.  This is now a
    leader--follower pair, and its leader is not also a follower.
  \end{itemize}
  If $f$ matches up to $g$, then $g$ matches down to $f$, because the new combined block cannot be a follower (otherwise the leader in $f$ would also be a follower, contradicting the definition of leader).  Similarly, if $g$ matches down to $f$, then $f$ matches up to $g$.  Thus, the match-up and match-down cells pair to form a discrete vector field.

  It remains to show that this discrete vector field is induced by our
  ordering.  Let $f$ be a match-down cell and $g$ be the corresponding match-up
  cell.  We need to show $f$ is the greatest face of $g$ and $g$ is the least
  coface of $f$.

  To show that $f$ is the greatest face of $g$, consider the result of splitting
  any earlier block of $g$.  Because $g$ looks like a critical cell at that
  stage, that block is a follower; after splitting, it is shorter and is still a
  follower, so it is smaller by property (ii) of the ordering.  In contrast,
  among ways to split the $k$th block another way, or to split a later block, $f$
  is the greatest because it is the only one for which the $k$th block begins
  with the least element of the $k$th block of $g$ (here we apply properties (i)
  and (iii)).

  To show that $g$ is the least coface of $f$, consider the result of combining
  any earlier blocks of $f$.  Because $f$ looks like a critical cell at that
  stage, the two blocks would be non-follower singletons in decreasing order, so
  the combined block would be larger, not be a follower, and have the same first
  element as the first of the two singletons; therefore the new block is larger
  by property (iii).  In contrast, among ways to combine later blocks of $f$, $g$
  is the least because it is the only one that increases the first element of the
  $k$th block (again applying property (iii)).
\end{proof}

\subsection{Basis for homology}
We now give a basis for the homology of $P(A,\mathcal{W},k)$ for an $n$-element
set $A$, thereby demonstrating that it is free abelian.  By Lemma
\ref{lem-max-basis}, it is sufficient to exhibit a cycle for every critical cell
such that the critical cell is the largest cell in the cycle.  The construction
implicitly relies on an ordering on $A$; for it to work, we must fix an ordering
so that $\mathcal{W}=(w_1,\ldots,w_n)$ satisfies
\begin{equation} \label{sizes}
  w_1 \leq w_2 \leq \cdots \leq w_n.
\end{equation}

Every critical cell $e$ is a concatenation product of non-follower singletons and
leader-follower pairs $(\,b_i \mid b_{i+1}\,)$.  We now build a corresponding
cycle $z(e) \in P(A,\mathcal{W},k)$.  For a cell consisting of one singleton, the
cycle will simply be the corresponding 0-cell.  Now suppose $e$ consists of a
single leader-follower pair $(\,b_1 \mid b_2\,)$.  Then our cycle $z(e)$ will be
the boundary of the top cell $(\,b_1b_2\,)$ of $P(A)$.  Since by \eqref{sizes}
each element of $b_2$ has greater or equal weight to the singleton $b_1$, every
cell in the boundary is a cell of $P(A,\mathcal{W},k)$.  Moreover,
$(\,b_1 \mid b_2\,)$ is the highest cell of $z(e)$ according to our total
ordering, since $b_1$ is the highest-ranked block.

For a general critical cell, we define $z$ by requiring that
\[z(e)=z(e_1 \concat e_2)=z(e_1) \concat z(e_2)\]
for any splitting $e=e_1 \concat e_2$ which does not split a leader--follower
pair.  Here $z(e_i)$ is defined based on the ordering on $A_i \subset A$
inherited from $A$.  To see that $e$ is the highest cell in the resulting cycle,
note that the ordering on cells depends on the first block in which two cells
differ.  For any two cells in $z(e)$, this block will be a leader and the above
argument will apply.

It is clear also that all the nonzero coefficients of $z(e)$ are $\pm 1$.  We
have now proved:
\begin{thm}\label{thm-weighted-basis}
  Suppose that $\mathcal{W}=(w_1,\ldots,w_n)$ is a nondecreasing sequence of
  weights.  Then $H_*(P(n, \mathcal{W}, k))$ is free abelian, with a basis given
  by the classes of the cycles $z(e)$, where $e$ ranges over all cells whose
  blocks are of the following two types:
  \begin{enumerate}[(a)]
  \item A singleton which is not a follower.
  \item A follower such that its weight and that of its leader add up to at
    least $k+1$.
  \end{enumerate}
\end{thm}
Translating these cellular cycles into cycles in $\no_{k+1}(n,\mathcal{W})$, we
get the following picture:
\begin{itemize}
\item Singletons and leader--follower pairs correspond to points and groups of
  points arranged in order along the line.
\item Every leader--follower pair corresponds to a set of $r$ points moving back
  and forth of which any $r-1$ can coincide, but all $r$ cannot.
\end{itemize}

\section{Decomposing cell$(n,w)$ into layers}\label{sec:layers}

In this section we will prove Theorems \ref{thm:FG} and \ref{thm:basis} by
expressing $H_*(\cel(n,w))$ in terms of the homology of weighted no-$(w+1)$-equal
spaces.  To this end, we assign an ordering to the top-dimensional cells of
$\cel(n)$.  For a lower-dimensional cell, its \strong{layer} will be indexed by the
first top-dimensional cell containing it: this is when that cell appears in the
complex when we think of the complex as glued, in order, out of its top-dimensional cells.
Recall that these top cells are identified with permutations in $S_n$.

The intersections of layers of $\cel(n)$ with $\cel(n,w)$ form the layers of
$\cel(n,w)$.  We will show that $H_*(\cel(n,w))$ is a direct sum of pieces which
appear once each subsequent layer is glued on.  Topologically, the layer
associated to a permutation $\sigma \in S_n$ is a copy of
$P(n-r,\mathcal{W},w) \times [0,1]^r$, where $r$ and $\mathcal{W}$ depend on
$\sigma$, which is glued on along $P(n-r,\mathcal{W},w) \times \partial[0,1]^r$;
the combinatorial structure is rather more complicated.  It follows (once
homological triviality of the gluing is established) that the added summand of
$H_j(\cel(n,w))$ is in bijection with elements of $H_{j-r}(P(n-r,\mathcal{W},w))$.  The main
technical result of this section states:
\begin{thm} \label{thm:split}
  The homology of $\cel(n,w)$ decomposes as
  \[H_*(\cel(n,w))=
  \bigoplus_{\sigma \in S_n} H_{*-\#\sigma}(P(n-\#\sigma,\mathcal{W}(\sigma),w)).\]
\end{thm}
We will define $\#\sigma$ and $\mathcal{W}(\sigma)$ combinatorially.

From the configuration space point of view, the new cycles in layer $\sigma$ are
those basic cycles from Theorem \ref{thm:basis} that have a particular collection
of wheels (irrespective of how those wheels are grouped into filters).  The
bijection above is given by replacing wheels of $k$ disks by points of weight
$k$.  Thus, for example, this bijection in an appropriate layer takes the
$14$-cycle depicted in Figure \ref{14-cycle} to the $3$-cycle depicted in Figure
\ref{nokcycle}.

The decomposition depends in a crucial way on the ordering of labels of disks;
it is not at all equivariant with respect to the $S_n$-action on $\cel(n,w)$.
Therefore, the methods of this section will tell us little about the $S_n$-module
structure on the homology.

\subsection{Combinatorial description of layers}
Given a cell in $\cel(n)$ broken up into blocks, we further (deterministically)
break up each block into \strong{wheels}: each entry of a block is the \strong{axle}
of a wheel if it is the largest entry of the block up to that point, and the
wheel consists of the axle and all the following smaller entries before the next
axle.

\begin{prop}\label{prop-layers}
  Given a symbol $f$, the following represent the same permutation $\sigma(f)$ of
  $[n]$:
  \begin{enumerate}[(i)]
  \item The lexicographically least \strong{shuffle} of $f$.  A shuffle is a
    permutation in which the order of every block is preserved.
  \item The permutation obtained by arranging all the wheels, regardless of
    block, in ascending order by axle.
  \end{enumerate}
\end{prop}

\begin{proof}
The first number in the lexicographically least shuffle is an axle, because the first element of every block is an axle.  Thus, it must be the least axle.  The remaining elements of the wheel of that axle are less than that axle and thus are also less than all the other axles.  Thus, the next numbers in the lexicographically least shuffle are the remaining elements of the first wheel.  Repeating the same argument for each wheel, we deduce inductively from left to right that the two permutations are identical.
\end{proof}

For example, the symbol
\[f=(\,7\;2 \mid 6 \mid 4\;5\;8\;1\;3\,)\]
has five wheels: $7\,2$, $6$, $4$, $5$, and $8\,1\,3$; and therefore
$\sigma(f)=4\,5\,6\,7\,2\,8\,1\,3$.
We say that $f$ is in the \strong{layer} $L(\sigma(f))$ of $\cel(n)$, or of
$\cel(n,w)$.  By Proposition~\ref{prop-layers}, being in a given layer is equivalent to having a given set of wheels.  We also write
\[\#\sigma=n-\text{the number of wheels of }\sigma.\]
Notice that $\#\sigma$ is the dimension of the lowest-dimensional cells of
$L(\sigma)$, in which every wheel is its own block.

Let $g$ be a boundary cell of $f$.  Then one of the following holds:
\begin{enumerate}
\item The splitting of a block to make $g$ respects the wheels of $f$: each wheel
  goes completely into one of the new blocks.  Then $g \in L(\sigma(f))$.
\item At least one wheel of $f$ is decomposed into two or more wheels of $g$.
  Then $\sigma(g) \prec \sigma(f)$ lexicographically.
\end{enumerate}
Thus we can build up $\cel(n)$ or $\cel(n,w)$ by gluing each subsequent layer, in
lexicographical order, onto the union of the previous ones.  In other words, the
layers define a filtration by subcomplexes
\[L(\leq \sigma)=\bigcup_{\tau \leq \sigma} L(\tau).\]

Now for each cell of $L(\sigma)$, we can obtain a cell of
$P(\{\text{wheels of }\sigma\})$ with the same block structure, replacing each
wheel by a single label.
\begin{prop} \label{bij}
  The $k$-cells of the layer $L(\sigma)$ of $\cel(n,w)$ are in
  incidence-preserving bijection with the $(k-\#\sigma)$-cells of the complex
  $P(n-\#\sigma,\mathcal{W}(\sigma),w)$, where $\mathcal{W}(\sigma)$ consists of
  the cardinalities of the wheels of $\sigma$.  In particular, the cells of the
  layer $L(\sigma)$ of $\cel(n)$ are in incidence-preserving bijection with those
  of $P(n-\#\sigma)$.
\end{prop}

\begin{proof}
All cells of $L(\sigma)$ have exactly the same wheels, and the elements of each wheel always appear consecutively.  Thus, given a symbol in $L(\sigma)$, we can view it as a sequence of these wheels, separated by vertical bars between blocks.  The resulting new symbol is a symbol of $P(n-\#\sigma,\mathcal{W}(\sigma),w)$, because instead of $n$ numbers in each symbol we have $n-\#\sigma$ wheels.  If the cell of the original symbol has dimension $k$, it has $n-1-k$ bars, so the cell of the new symbol also has $n-1-k$ bars, and thus has dimension $n-\#\sigma-1-(n-1-k) = k-\#\sigma$.

As in option (1) above, cell $g$ in $L(\sigma)$ is a boundary cell of $f$ in $L(\sigma)$ if and only if a block in $f$ is split to form $g$, such that each wheel in the block is assigned entirely to the left or entirely to the right.  Reinterpreting the symbols to be sequences of wheels, this is the same as the criterion for incidence in $P(n-\#\sigma, \mathcal{W}(\sigma), w)$.
\end{proof}

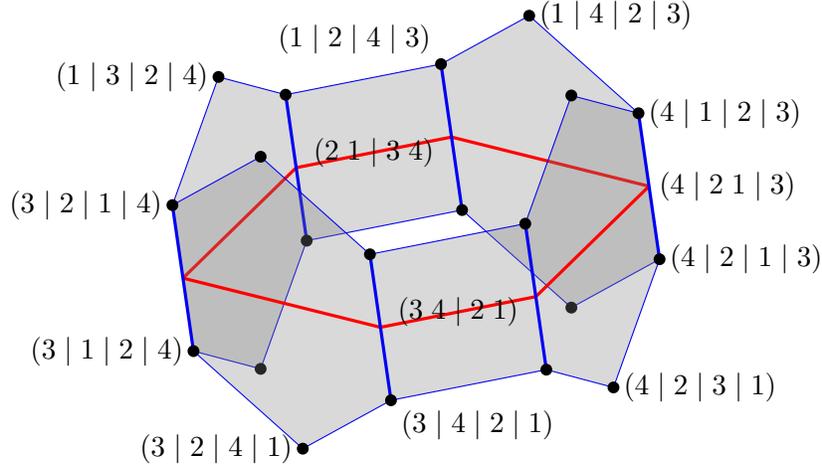
\begin{figure}
  \begin{tikzpicture}%
	[x={(0.684801cm, -0.611902cm)},
	y={(0.728730cm, 0.575018cm)},
	z={(-0.000002cm, 0.543074cm)},
	scale=1.5,
	back/.style={},
	edge/.style={color=blue!95!black},
	facet/.style={fill=gray,fill opacity=0.300000},
	vertex/.style={inner sep=1.3pt,circle,draw=black!25!black,fill=black!75!black,thick}]
%
%

\coordinate (-A) at (-0.70711, -1.22474, -1.73205);
\coordinate (-B) at (-0.70711, -2.04124, -0.57735);
\coordinate (-C) at (-1.41421, 0.00000, -1.73205);
\coordinate (--+D) at (-1.41421, -1.63299, 0.57735);
\coordinate (-+-E) at (-2.12132, 0.40825, -0.57735);
\coordinate (--+E) at (-2.12132, -0.40825, 0.57735);
\coordinate (+--A) at (0.70711, -1.22474, -1.73205);
\coordinate (+--B) at (0.70711, -2.04124, -0.57735);
\coordinate (-+-A) at (-0.70711, 1.22474, -1.73205);
\coordinate (--+A) at (-0.70711, -1.22474, 1.73205);
\coordinate (-+-D) at (-1.41421, 1.63299, -0.57735);
\coordinate (-+C) at (-1.41421, 0.00000, 1.73205);
\coordinate (+-C) at (1.41421, 0.00000, -1.73205);
\coordinate (+-+D) at (1.41421, -1.63299, 0.57735);
\coordinate (++-A) at (0.70711, 1.22474, -1.73205);
\coordinate (+-+A) at (0.70711, -1.22474, 1.73205);
\coordinate (-++B) at (-0.70711, 2.04124, 0.57735);
\coordinate (-++A) at (-0.70711, 1.22474, 1.73205);
\coordinate (++-E) at (2.12132, 0.40825, -0.57735);
\coordinate (+-+E) at (2.12132, -0.40825, 0.57735);
\coordinate (++-D) at (1.41421, 1.63299, -0.57735);
\coordinate (C) at (1.41421, 0.00000, 1.73205);
\coordinate (B) at (0.70711, 2.04124, 0.57735);
\coordinate (A) at (0.70711, 1.22474, 1.73205);
\fill[facet] (-C) -- (-+-A) -- (-+-D) -- (-+-E) -- cycle {};
\fill[facet] (-+-A) -- (++-A) -- (++-D) -- (B) -- (-++B) -- (-+-D) -- cycle {};
\fill[facet] (-A) -- (-C) -- (-+-E) -- (--+E) -- (--+D) -- (-B) -- cycle {};
\draw[red,very thick] (barycentric cs:++-D=0.5,B=0.5) -- (barycentric cs:-+-A=0.5,-+-D=0.5) -- (barycentric cs:-C=0.5,-+-E=0.5) -- (barycentric cs:--+D=0.5,-B=0.5);

\node[left] at (--+D) {$(3 \mid 2 \mid 1 \mid 4)$};
\node[left] at (-B) {$(3 \mid 1 \mid 2 \mid 4)$};
\node[right] at (++-D) {$(4 \mid 2 \mid 1 \mid 3)$};
\node[right] at (barycentric cs:++-D=0.5,B=0.5) {$(4 \mid 2\;1 \mid 3)$};
\node[right] at (B) {$(4 \mid 1 \mid 2 \mid 3)$};
\node[right] at (-++B) {$(1 \mid 4 \mid 2 \mid 3)$};
\node[above left] at (-+-D) {$(1 \mid 2 \mid 4 \mid 3)$};
\node[right] at (++-E) {$(4 \mid 2 \mid 3 \mid 1)$};
\node[left] at (+--B) {$(3 \mid 2 \mid 4 \mid 1)$};
\node[below right] at (+-+D) {$(3 \mid 4 \mid 2 \mid 1)$};
\node[left] at (--+E) {$(1 \mid 3 \mid 2 \mid 4)$};
\draw[edge,back] (-A) -- (-B);
\draw[edge,back] (-A) -- (-C);
\draw[edge,back,very thick] (-C) -- (-+-E);
\draw[edge,back] (-C) -- (-+-A);
\draw[edge,back] (-+-E) -- (--+E);
\draw[edge,back] (-+-E) -- (-+-D);
\draw[edge,back,very thick] (-+-A) -- (-+-D);
\draw[edge,back] (-+-A) -- (++-A);
\draw[edge,back] (-+-D) -- (-++B);
\draw[edge,back] (++-A) -- (++-D);
\node[vertex] at (-A)     {};
\node[vertex] at (-C)     {};
\node[vertex] at (-+-E)     {};
\node[vertex] at (-+-A)     {};
\node[vertex] at (-+-D)     {};
\node[vertex] at (++-A)     {};
\fill[facet] (C) -- (+-+A) -- (+-+D) -- (+-+E) -- cycle {};
\fill[facet] (+-+A) -- (--+A) -- (--+D) -- (-B) -- (+--B) -- (+-+D) -- cycle {};
\fill[facet] (A) -- (C) -- (+-+E) -- (++-E) -- (++-D) -- (B) -- cycle {};

\draw[red,very thick] (barycentric cs:--+D=0.5,-B=0.5) -- (barycentric cs:+-+A=0.5,+-+D=0.5) -- (barycentric cs:C=0.5,+-+E=0.5) -- (barycentric cs:++-D=0.5,B=0.5);

\node at (barycentric cs:+-+D=0.5,C=0.5) {$(3\;4 \mid 2\;1)$};
\node at (barycentric cs:-+-D=0.5,-C=0.5) {$(2\;1 \mid 3\;4)$};
\draw[edge,very thick] (-B) -- (--+D);
\draw[edge] (-B) -- (+--B);
\draw[edge] (--+D) -- (--+E);
\draw[edge] (--+D) -- (--+A);
\draw[edge] (+--B) -- (+-+D);
\draw[edge] (--+A) -- (+-+A);
\draw[edge,very thick] (+-+D) -- (+-+A);
\draw[edge] (+-+D) -- (+-+E);
\draw[edge] (+-+A) -- (C);
\draw[edge] (-++B) -- (B);
\draw[edge] (++-E) -- (+-+E);
\draw[edge] (++-E) -- (++-D);
\draw[edge,very thick] (+-+E) -- (C);
\draw[edge,very thick] (++-D) -- (B);
\draw[edge] (C) -- (A);
\draw[edge] (B) -- (A);
\node[vertex] at (-B)     {};
\node[vertex] at (--+D)     {};
\node[vertex] at (--+E)     {};
\node[vertex] at (+--B)     {};
\node[vertex] at (--+A)     {};
\node[vertex] at (+-+D)     {};
\node[vertex] at (+-+A)     {};
\node[vertex] at (-++B)     {};
\node[vertex] at (++-E)     {};
\node[vertex] at (+-+E)     {};
\node[vertex] at (++-D)     {};
\node[vertex] at (C)     {};
\node[vertex] at (B)     {};
\node[vertex] at (A)     {};
\end{tikzpicture}
  \caption{The layer of $\cel(4,3)$ corresponding to the permutation
    $\sigma=2\,1\,3\,4$.  Selected cells are labeled, and the core (a copy of
    $P(3,\mathcal{W}(\sigma),3)$, which is the boundary of a $P(3)$) is
    highlighted in red.  The image of $\spin_{\sigma,3}$ is combinatorially the
    double of the layer along the top and bottom: in the additional cells, $2\,1$
    is replaced by $1\,2$.} \label{fig:inj}
\end{figure}
In fact, this bijection is not just combinatorial, but can be understood from
several points of view: via cellular maps, cellular chains, or configurations.
In the next subsection, we will construct an injective cellular map
\[[0,1]^{\#\sigma} \times P(n-\#\sigma,\mathcal{W}(\sigma),w) \to \cel(n,w).\]
which matches each product of $[0,1]^{\#\sigma}$ with an $(k-\#\sigma)$-cell to a
$k$-cell in $L(\sigma)$.  In particular, the image of
$\{\frac{1}{2}\}^{\#\sigma} \times P(n-\#\sigma,\mathcal{W}(\sigma),w)$ forms a
``core'' or ``spine'' inside the layer, as illustrated in Figure \ref{fig:inj}.
In fact, this will be the restriction of a (likewise injective and cellular) map
\[\spin_{\sigma,w}:T^{\#\sigma} \times P(n-\#\sigma,\mathcal{W}(\sigma),w) \to
L(\leq \sigma),\]
to one top cell of the torus.  Here $T^{\#\sigma}$ is given a product cell
structure induced by a cell structure on $S^1$ with two vertices and two edges.
In turn, $\spin_{\sigma,w}$ is the restriction of a map
\[\spin_\sigma:T^{\#\sigma} \times P(n-\#\sigma) \to \cel(n).\]

From the point of view of cellular chains, $\spin_{\sigma,w}$ induces a chain map
\[i_\sigma:C_{*-\#\sigma}(P(n-\#\sigma,\mathcal{W}(\sigma),w)) \to
C_*(L(\leq \sigma))\]
via
\[i_\sigma(z)=\spin_{\sigma,w}([T^{\#\sigma}] \times z).\]
In the other direction, we get a surjective chain map
\[p_\sigma:C_*(L(\leq \sigma)) \to C_{*-\#\sigma}(P(n-\#\sigma,\mathcal{W}(\sigma),w))\]
in which cells of $L(<\sigma)$ (defined as
$L(<\sigma)=\bigcup_{\tau \prec \sigma} L(\tau)$) are sent to $0$ and cells of
$L(\sigma)$ are sent to cells of $P(n-\#\sigma,\mathcal{W}(\sigma),w)$, up to
sign.  From the above discussion, one sees that $p_\sigma \circ i_\sigma=\id$.
Together $p_\sigma$ and $i_\sigma$ give a splitting
\begin{equation} \label{split}
  H_*(L(\leq \sigma)) \cong H_*(L(<\sigma)) \oplus
  H_{*-\#\sigma}(P(n-\#\sigma,\mathcal{W}(\sigma),w)).
\end{equation}
This implies Theorem \ref{thm:split}, once we produce 
the map
$\spin_\sigma$.

Finally, from the configuration point of view, the map $\spin_{\sigma,w}$
associates a configuration in the weighted no-$(w+1)$-equal space to a torus of
configurations obtained by replacing points of weight $r$ by wheels of size $r$
and spinning those wheels.  This also describes the action of $i_\sigma$ on
cycles.  For example, when
\[\sigma=3\,7\,5\,9\,8\,10\,13\,4\,15\,17\,14\,18\,20\,21\,12\,2\,22\,16\,23\,6\,11\,1\,24\,19,\]
$i_\sigma$ sends the $3$-cycle in Figure \ref{nokcycle} to the $14$-cycle in
Figure \ref{14-cycle}.  For points in the core, the disks in the wheel are in
the ``standard'' vertically ordered position.

\begin{rmk} \label{rmk:crit-cells}
  Instead of using the splitting, one can prove Theorem \ref{thm:split} directly
  by constructing a discrete gradient vector field and applying Lemma
  \ref{lem-max-basis}.  Put a total ordering $\prec$ on the cells of $\cel(n,w)$
  as follows:
  \begin{enumerate}[(i)]
  \item If $\sigma(g) \prec \sigma(f)$, then $g \prec f$.
  \item If $\sigma(g)=\sigma(f)$, then use the previously defined ordering on
    cells of $P(n-\#\sigma,\mathcal{W}(\sigma),w)$.  This ordering is based on
    an ordering on the wheels of $\sigma$; for this we order first by number of
    elements, then by axle.
  \end{enumerate}
  This ordering induces a discrete gradient vector field on $\cel(n,w)$ which restricts
  to that on $P(n-\#\sigma,\mathcal{W}(\sigma),w)$ on each layer.  The images of
  the bases for $H_*(P(n-\#\sigma,\mathcal{W}(\sigma),w)$ under $i_\sigma$, for
  all $\sigma$, give us a set of cycles to which we can apply Lemma
  \ref{lem-max-basis}.

  For later reference, we describe the set of critical cells of this discrete
  gradient vector field in a self-contained way.  Order wheels in a layer first
  according to their number of elements and then according to their largest
  element (\strong{axle}).  A block is a \strong{unicycle} if it consists of a single
  wheel, that is, if its largest element comes first.  We assign some blocks to
  leader--follower pairs by walking left to right: a block is a \strong{follower}
  if it follows a unicycle (its \strong{leader}) which is not itself a follower and
  whose wheel is smaller than any of the follower's wheels.  A cell of
  $\cel(n,w)$ is critical if every block is either
  \begin{enumerate}[(i)]
  \item A unicycle which is not a follower.
  \item A follower such that it and its leader have at least $w+1$ elements in
    total.
  \end{enumerate}
\end{rmk}

\subsection{Maps between permutohedra}
To complete the proof of Theorem \ref{thm:split}, we must still describe the map
$\spin_\sigma:T^{\#\sigma} \times P(n-\#\sigma) \to \cel(n)$.  Informally, points in
$P(n-\#\sigma)$ encode positions of the wheels of $\sigma$ relative to each
other, whereas points in the torus encode positions of disks inside the wheels,
which can vary as in Figure \ref{fig:wheel}.  In particular, the image
of $\spin_\sigma$ will be the union of the $2^{\#\sigma}$ top-dimensional cells
obtained from $\sigma$ by ``spinning its wheels'', that is, by applying
permutations such as those depicted in Figure \ref{fig:wheel2}.
\begin{figure}[h]
  \begin{tikzpicture}[scale=0.6]
    \draw[blue, thick] (6,0) circle (0.5);
    \draw[blue, thick] (6,1) circle (0.5);
    \draw[blue, thick] (6,2) circle (0.5);
    \draw[blue, thick] (6,3) circle (0.5);
    \draw[blue, thick] (6,4) circle (0.5);
    \draw(6,3.5) circle (1);
    \draw(6,3) circle (1.5);
    \draw(6,2.5) circle (2);
    \draw(6,2) circle (2.5);
    \node[blue] at (6,0) {\bf 7};
    \node[blue] at (6,1) {\bf 1};
    \node[blue] at (6,2) {\bf 11};
    \node[blue] at (6,3) {\bf 6};
    \node[blue] at (6,4) {\bf 23};

    \draw[blue, thick] (12,0) circle (0.5);
    \draw[blue, thick] (12,1) circle (0.5);
    \draw[blue, thick] (12,2) circle (0.5);
    \draw[blue, thick] (12,3) circle (0.5);
    \draw[blue, thick] (12,4) circle (0.5);
    \draw(12,1.5) circle (1);
    \draw(12,2) circle (1.5);
    \draw(12,1.5) circle (2);
    \draw(12,2) circle (2.5);
    \node[blue] at (12,0) {\bf 1};
    \node[blue] at (12,1) {\bf 23};
    \node[blue] at (12,2) {\bf 6};
    \node[blue] at (12,3) {\bf 11};
    \node[blue] at (12,4) {\bf 7};

    \draw[blue, thick] (18,0) circle (0.5);
    \draw[blue, thick] (18,1) circle (0.5);
    \draw[blue, thick] (18,2) circle (0.5);
    \draw[blue, thick] (18,3) circle (0.5);
    \draw[blue, thick] (18,4) circle (0.5);
    \draw(18,1.5) circle (1);
    \draw(18,2) circle (1.5);
    \draw(18,2.5) circle (2);
    \draw(18,2) circle (2.5);
    \node[blue] at (18,0) {\bf 7};
    \node[blue] at (18,1) {\bf 6};
    \node[blue] at (18,2) {\bf 23};
    \node[blue] at (18,3) {\bf 11};
    \node[blue] at (18,4) {\bf 1};
  \end{tikzpicture}
  \caption{Some configurations of a wheel with $5$ disks which give different
    orderings, the first of which is the ``name'' of the wheel.}
  \label{fig:wheel2}
\end{figure}
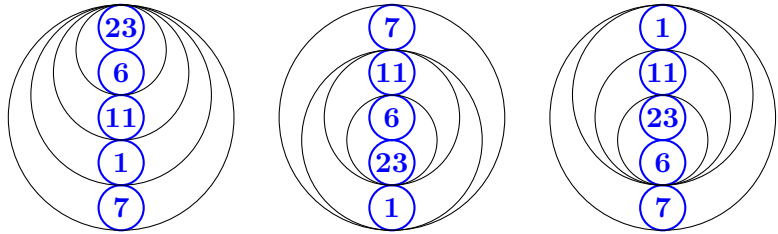

To make this precise, we start with the following lemma:
\begin{lem} \label{lem:hyperplane}
  Let $A$ be a finite set and $\mathbf{a}$, $\mathbf{b}$, $\mathbf{c}$ additional
  symbols not in $A$.  Then there is a map
  \[\iota:[0,1] \times P(A \cup \{\mathbf{a}\}) \to P(A\cup\{\mathbf{b,c}\})\]
  with the following properties:
  \begin{enumerate}[(i)]
  \item It is a homeomorphism and a cellular map.  That is, it sends every
    $k$-face to a disk which is a union of $k$-faces.
  \item It is equivariant with respect to 
    the $\mathbb{Z}/2\mathbb{Z}$-actions given by $t \mapsto -t$ on the domain
    and $\mathbf{b} \leftrightarrow \mathbf{c}$ on the codomain.
  \item For each cell $\sigma$ of $P(A \cup \{\mathbf{a}\})$,
    $\iota([0,1] \times \sigma)$ is the cell of
    $P(A \cup \{\mathbf{b},\mathbf{c}\})$ with the same blocks except that
    $\mathbf{a}$ is replaced in its block by $\mathbf{b}$ and $\mathbf{c}$.
  \item It takes $\{0\} \times P(A \cup \{\mathbf{a}\})$ to the union of
    those cells in which $\mathbf{b}$ and $\mathbf{c}$ are contained in separate
    blocks with $\mathbf{b}$ preceding $\mathbf{c}$.  Similarly, it takes
    $\{1\} \times P(A \cup \{\mathbf{a}\})$ to the union of those cells in which
    $\mathbf{c}$ precedes $\mathbf{b}$.
  \end{enumerate}
\end{lem}

\begin{figure}
  \begin{tikzpicture}%
	[x={(-0.413993cm, -0.273860cm)},
	y={(0.910280cm, -0.124463cm)},
	z={(-0.000084cm, 0.953682cm)},
	scale=1,
	top/.style={very thick,draw=blue!95!black,fill=blue,fill opacity=0.3},
	back/.style={dotted},
	edge/.style={color=blue!95!black},
	facet/.style={fill=gray,fill opacity=0.300000},
	back vertex/.style={inner sep=1.3pt,circle,draw=gray,fill=gray,thick},
	vertex/.style={inner sep=1.3pt,circle,draw=black!25!black,fill=black!75!black,thick}]
%
%

\coordinate (-A) at (-0.70711, -1.22474, -1.73205);
\coordinate (-B) at (-0.70711, -2.04124, -0.57735);
\coordinate (-C) at (-1.41421, 0.00000, -1.73205);
\coordinate (--+D) at (-1.41421, -1.63299, 0.57735);
\coordinate (-+-E) at (-2.12132, 0.40825, -0.57735);
\coordinate (--+E) at (-2.12132, -0.40825, 0.57735);
\coordinate (+--A) at (0.70711, -1.22474, -1.73205);
\coordinate (+--B) at (0.70711, -2.04124, -0.57735);
\coordinate (-+-A) at (-0.70711, 1.22474, -1.73205);
\coordinate (--+A) at (-0.70711, -1.22474, 1.73205);
\coordinate (-+-D) at (-1.41421, 1.63299, -0.57735);
\coordinate (-+C) at (-1.41421, 0.00000, 1.73205);
\coordinate (+-C) at (1.41421, 0.00000, -1.73205);
\coordinate (+-+D) at (1.41421, -1.63299, 0.57735);
\coordinate (++-A) at (0.70711, 1.22474, -1.73205);
\coordinate (+-+A) at (0.70711, -1.22474, 1.73205);
\coordinate (-++B) at (-0.70711, 2.04124, 0.57735);
\coordinate (-++A) at (-0.70711, 1.22474, 1.73205);
\coordinate (++-E) at (2.12132, 0.40825, -0.57735);
\coordinate (+-+E) at (2.12132, -0.40825, 0.57735);
\coordinate (++-D) at (1.41421, 1.63299, -0.57735);
\coordinate (C) at (1.41421, 0.00000, 1.73205);
\coordinate (B) at (0.70711, 2.04124, 0.57735);
\coordinate (A) at (0.70711, 1.22474, 1.73205);

\draw[edge,back] (-A) -- (-B);
\draw[edge,back] (-A) -- (-C);
\draw[edge,back] (-A) -- (+--A);
\draw[edge,back] (-C) -- (-+-E);
\draw[edge,back] (-C) -- (-+-A);
\draw[edge,back] (-+-E) -- (--+E);
\draw[edge,back] (-+-E) -- (-+-D);
\draw[edge,back] (-B) -- (--+D);
\draw[edge,back] (-B) -- (+--B);
\draw[edge,back] (--+D) -- (--+E);
\draw[edge,back] (--+D) -- (--+A);
\draw[edge,back] (--+E) -- (-+C);
\node[back vertex] at (-A)     {};
\node[back vertex] at (-C)     {};
\node[back vertex] at (-+-E)     {};
\node[back vertex] at (--+E)     {};
\node[back vertex] at (--+D)     {};
\node[back vertex] at (-B)     {};
\draw[red,dotted]  (barycentric cs:-++B=0.5,-+-D=0.5) -- (barycentric cs:--+E=0.5,-+C=0.5) -- (barycentric cs:--+D=0.5,--+A=0.5) -- (barycentric cs:+--B=0.5,+-+D=0.5);
\fill[red,opacity=0.1] (barycentric cs:+--B=0.5,+-+D=0.5) -- (barycentric cs:++-E=0.5,+-C=0.5) -- (barycentric cs:++-D=0.5,++-A=0.5) -- (barycentric cs:-++B=0.5,-+-D=0.5) -- (barycentric cs:--+E=0.5,-+C=0.5) -- (barycentric cs:--+D=0.5,--+A=0.5) -- cycle;
\filldraw[top] (C) -- (+-+A) -- (+-+D) -- (+-+E) -- cycle {};
\filldraw[top] (A) -- (-++A) -- (-+C) -- (--+A) -- (+-+A) -- (C) -- cycle {};
\fill[facet] (-+-A) -- (++-A) -- (++-D) -- (B) -- (-++B) -- (-+-D) -- cycle {};
\fill[facet] (+-C) -- (++-E) -- (++-D) -- (++-A) -- cycle {};
\filldraw[top] (A) -- (-++A) -- (-++B) -- (B) -- cycle {};
\fill[facet] (+-+E) -- (+-+D) -- (+--B) -- (+--A) -- (+-C) -- (++-E) -- cycle {};
\filldraw[top] (A) -- (C) -- (+-+E) -- (++-E) -- (++-D) -- (B) -- cycle {};

\draw[red,very thick] (barycentric cs:+--B=0.5,+-+D=0.5) -- (barycentric cs:++-E=0.5,+-C=0.5) -- (barycentric cs:++-D=0.5,++-A=0.5) -- (barycentric cs:-++B=0.5,-+-D=0.5);

\draw[edge] (+--A) -- (+--B);
\draw[edge] (+--A) -- (+-C);
\draw[edge] (+--B) -- (+-+D);
\draw[edge] (--+A) -- (-+C);
\draw[edge] (--+A) -- (+-+A);
\draw[edge] (-+C) -- (-++A);
\draw[edge] (+-C) -- (++-E);
\draw[edge] (+-+D) -- (+-+A);
\draw[edge] (+-+D) -- (+-+E);
\draw[edge] (+-+A) -- (C);
\draw[edge] (-++B) -- (-++A);
\draw[edge] (-++B) -- (B);
\draw[edge] (-++A) -- (A);
\draw[edge] (++-E) -- (+-+E);
\draw[edge] (++-E) -- (1.41421, 1.63299, -0.57735);
\draw[edge] (+-+E) -- (C);
\draw[edge] (++-D) -- (B);
\draw[edge] (C) -- (A);
\draw[edge] (B) -- (A);
\draw[edge] (++-A) -- (++-D);
\draw[edge] (-+-A) -- (-+-D);
\draw[edge] (-+-A) -- (++-A);
\draw[edge] (-+-D) -- (-++B);
\draw[edge] (+-C) -- (++-A);
\node[vertex] at (+--A)     {};
\node[vertex] at (-++A)     {};
\node[vertex] at (+--B)     {};
\node[vertex] at (-++B)     {};
\node[vertex] at (--+A)     {};
\node[vertex] at (++-A)     {};
\node[vertex] at (-+C)     {};
\node[vertex] at (+-C)     {};
\node[vertex] at (+-+D)     {};
\node[vertex] at (-+-D)     {};
\node[vertex] at (+-+A)     {};
\node[vertex] at (-+-A)     {};
\node[vertex] at (++-E)     {};
\node[vertex] at (+-+E)     {};
\node[vertex] at (++-D)     {};
\node[vertex] at (C)     {};
\node[vertex] at (B)     {};
\node[vertex] at (A)     {};
\end{tikzpicture}
  \caption{The permutohedron $P(4)$ and the homeomorphism
    $\iota:[0,1] \times P(3) \to P(4)$. The images of
    $\{\nicefrac{1}{2}\} \times P(3)$ (red) and $\{1\} \times P(3)$ (blue) are
    highlighted.}
\end{figure}
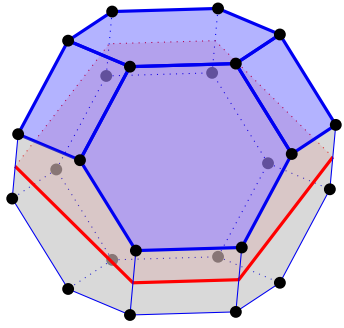

Before proving the lemma, we use it to construct $\spin_\sigma$.  First, notice
that for any ordering of $A \cup \{\mathbf{b},\mathbf{c}\}$, the lemma gives us a
well-defined map to the corresponding top cell of
$\cel(A\cup\{\mathbf{b},\mathbf{c}\})$.  In particular, if we are instead given
an ordering of $A \cup \{\mathbf{a}\}$ we have two choices: we can replace
$\mathbf{a}$ by $\mathbf{b}\,\mathbf{c}$ or by $\mathbf{c}\,\mathbf{b}$.
Moreover, these choices coincide on $\{0,1\} \times P(A \cup \{\mathbf{a}\})$.
Identifying the two subspaces gives a map
\[S^1 \times P(A \cup \{\mathbf{a}\}) \to \cel(A\cup\{\mathbf{b},\mathbf{c}\}).\]
Since top cells of $\cel(A \cup \{\mathbf{a}\})$ correspond to orderings of
$A \cup \{\mathbf{a}\}$, this gives us a map
\[\spin:S^1 \times \cel(A \cup \{\mathbf{a}\}) \to \cel(A\cup\{\mathbf{b},\mathbf{c}\}).\]
The equivariance of $\iota$ means that $\spin$ is well-defined.

We build $\spin_\sigma$ by iterating the map $\spin$.  Start by thinking of
$P(n-\#\sigma)$ as a top cell of $\cel(\{\text{wheels of }\sigma\})$.  At each
step, we replace a wheel with $k$ elements by a singleton and a wheel with $k-1$
elements, applying $\spin$ to this splitting.


From this description, it is evident that
\begin{align*}
  \spin_\sigma(P(n-\#\sigma,\mathcal{W}(\sigma),w)) &= \spin_\sigma(P(n)) \cap \cel(n,w).
\end{align*}

\begin{proof}[Proof of Lemma \ref{lem:hyperplane}]
  In the proof, we will use a slightly different notation than in the statement
  of the lemma: we will replace the finite set $A$ by $\{1,\ldots,n-2\}$ and the
  symbols $\mathbf a$, $\mathbf b$, and $\mathbf c$ by $n-1$, $n-1$, and $n$,
  respectively.  This lets us write coordinates in $\mathbb{R}^n$ more easily.

  The permutohedron $P(n)$ is a \strong{zonotope}, that is, the Minkowski sum of a
  set of line segments.  We refer to \cite[\S7.3]{Ziegler} for known facts about
  zonotopes.

  The standard permutohedron $P(n)$ is the zonotope in $\mathbb{R}^n$ which is
  the Minkowski sum of the ${n \choose 2}$ segments connecting every pair of
  standard unit basis vectors.  That is,
  \[P(n)=\Bigl\{\sum_{1 \leq i<j \leq n} a_{ij}\mathbf{e}_i+(1-a_{ij})\mathbf{e}_j :
  0 \leq a_{ij} \leq 1\Bigr\}.\]
  Evidently, choosing a different set of $n$ linearly independent vectors in any
  Euclidean space gives a linearly equivalent polytope.  Less obviously, the
  combinatorial structure of a zonotope only depends on the oriented matroid
  associated to the set of line segments.  
  In particular, changing the lengths of the line segments without changing their
  direction does not change the combinatorial structure.

  We use this fact to prove the following:
  \begin{lem} \label{comb-equiv}
    The polytope $Z(n)=P(n) \cap \{x_{n-1}=x_n\}$ is combinatorially equivalent to
    $P(n-1)$.
  \end{lem}
  \begin{proof}
    We first show that $Z(n)$ can also be written as
    \begin{equation} \label{eqn:avg}
      \Bigl\{\sum_{1 \leq i<j \leq n} a_{ij}\mathbf{e}_i+(1-a_{ij})\mathbf{e}_j :
      0 \leq a_{ij} \leq 1, a_{i(n-1)}=a_{in}, a_{(n-1)n}=\textstyle{\frac{1}{2}}
      \Bigr\} \subset P(n).
    \end{equation}
    Indeed, suppose that $\mathbf{z} \in Z(n)$, so that
    \[\mathbf{z}=\sum_{1 \leq i<j \leq n} b_{ij}\mathbf{e}_i+(1-b_{ij})\mathbf{e}_j
    \in P(n)\]
    and $z_{n-1}=z_n$.  If we switch $\mathbf{e}_{n-1}$ and $\mathbf{e}_n$ in this
    formula, we get $\mathbf{z}$ back.  We can get an expression as in
    \eqref{eqn:avg} by averaging these two expressions for $\mathbf{z}$, getting
    \[a_{i(n-1)}=a_{in}=\frac{b_{i(n-1)}+b_{in}}{2},\qquad a_{(n-1)n}=\frac{1}{2}.\]

    From \eqref{eqn:avg} we see that $Z(n)$ is the zonotope generated by the
    segments
    \[\mathbf{e}_i \leftrightarrow \mathbf{e}_j, 1 \leq i<j \leq n-2 \quad
    \text{and}\quad\mathbf{e}_{n-1}+\mathbf{e}_n \leftrightarrow 2\mathbf{e}_i,
    1 \leq i \leq n-2.\]
    This zonotope is combinatorially equivalent to one in which each segment
    $\mathbf{e}_{n-1}+\mathbf{e}_n \leftrightarrow 2\mathbf{e}_i$ is replaced by
    $\frac{\mathbf{e}_{n-1}+\mathbf{e}_n}{2} \leftrightarrow \mathbf{e}_i$.  This
    in turn is linearly equivalent to $P(n-1)$.
  \end{proof}

  The lemma gives us an injective map $\iota_{1/2}:P(n-1) \to P(n)$ which is a
  homeomorphism onto its image; it remains to extend it to
  $\iota:[0,1] \times P(n-1) \to P(n)$.  We do this by letting
  \[\iota(t,x)=
  \iota_{1/2}(x)+R(x)(2t-1)(\mathbf{e}_n-\mathbf{e}_{n-1}),\]
  where $R(x)$ is the maximum value for which the image still lies in $P(n)$.
  Property (ii) of the lemma follows immediately from this definition.

  To show that $\iota$ is the desired map, we must show it is a homeomorphism.
  To prove injectivity, it suffices to show that $R(x)$ is always positive.  But
  in fact $R(x) \geq 1/2$, since we can always vary $a_{(n-1)n}$.  Surjectivity
  follows from the fact that $P(n)$ is convex and symmetric about the hyperplane
  $Z(n)$.  Finally, $\iota$ can only be discontinuous if
  $R \circ \iota_{1/2}^{-1}:Z(n) \to \mathbb{R}$ is discontinuous, but convexity
  of $P(n)$ implies that this is a convex function, and therefore continuous.
  Continuity of the inverse likewise follows from the fact that the reciprocal of
  this function is continuous.

  Next we show that $\iota$ is cellular.  Given a face $\sigma$ of $P(n-1)$, we
  know that $\iota_{1/2}(\sigma)$ is a face of $Z(n)$ cut out by a half-space of
  the hyperplane $\{x_{n-1}=x_n\}$.  Extending this half-space orthogonally to
  $\mathbb{R}^n$, we obtain the half-space that cuts out the face
  $\iota([0,1] \times \sigma)$ of $P(n)$.  It follows that
  $\iota(\{0\} \times \sigma)$ and $\iota(\{1\} \times \sigma)$ are also unions
  of faces of $P(n)$ of the appropriate dimension: specifically, each facet of
  $\iota([0,1] \times \sigma)$ which is not the image of $[0,1] \times \tau$ for
  some facet $\tau$ of $\sigma$ is contained in one of those two, depending on
  whether $x_{n-1}-x_n$ is positive or negative for points in that facet.

  To show properties (iii) and (iv), we recall the relationship between symbols
  of cells and the zonotope structure.  Namely, points in a face corresponding to
  a given symbol are those which can be expressed as
  \[\sum_{1 \leq i<j \leq n} a_{ij}\mathbf{e}_i+(1-a_{ij})\mathbf{e}_j\]
  where $a_{ij}$ is $0$ if $j$ is in a later block than $i$ and $1$ if $i$ is in
  a later block than $j$.  When $i$ and $j$ are in the same block, $a_{ij}$ can
  be anything in $[0,1]$.

  From the proof of Lemma \ref{comb-equiv}, it follows that $\iota_{1/2}$ maps a
  face with a given symbol into a face with the same symbol with $n$ added to the
  same block as $n-1$.  Since $\iota$ is cellular, this shows (iii).  Property
  (iv) then also follows from our argument that $\iota$ is cellular.
\end{proof}

\subsection{Proofs of main theorems}
\begin{proof}[Proof of Theorem \ref{thm:basis}]
  Combining the arguments in \S\ref{S:nokequal} and the proof of Theorem
  \ref{thm:split}, we find a basis for $H_*(\cel(n,w))$: for each permutation
  $\sigma \in S_n$ and each basic cycle $z$ of
  $H_j(P(n-\#\sigma,\mathcal{W}(\sigma),w))$, it contains the cycle $i_\sigma(z)$.
  The exact set of cycles we get depends on the correspondence between wheels of
  $\sigma$ and elements of $[n-\#\sigma]$: to make this concrete, we order the
  wheels first by size and then by axle.

  Geometrically, the correspondence $z \mapsto i_\sigma(z)$ matches:
  \begin{itemize}
  \item Points moving in $\mathbb{R}$ to wheels moving in
    $\mathbb{R} \times [0,w]$.
  \item Points moving through each other to wheels moving over and above each
    other, with the wheel with the smaller axle on top.
  \end{itemize}
  We now verify the numbered conditions of Theorem \ref{thm:basis}:
  \begin{itemize}
  \item Condition (i) follows from the way we split $\sigma$ into wheels.
  \item Condition (ii) follows from the definition of the layer $L(\sigma)$.
  \item Conditions (iii) and (iv) follow from Lemma \ref{lem:crit} and the
    ordering of the wheels. \qedhere
  \end{itemize}
\end{proof}
\begin{proof}[Proof of Theorem \ref{thm:FG}]
  Every basic cycle is a concatenation product of wheels and filters.  It is
  enough to show that there are finitely many (unlabeled) types of wheels and
  filters in $\config({-},w)$, and that these have at most $3w/2$ disks and
  generate cycles of dimension at most $3w/2-2$.

  A wheel consists of at most $w$ disks, and a wheel with $k$ disks generates a
  $(k-1)$-cycle.  Filters with at least 3 wheels have at most $3w/2$ disks, and a
  filter with $k$ disks generates a $(k-2)$-cycle.  Filters with 2 wheels $b_1$
  and $b_2$ can be written as $b_2 \concat b_1-b_1 \concat b_2$, and do not need
  to be counted separately.  Therefore $H_*(\config({-},w))$ is spanned by
  concatenation products of cycles in
  $H_{\leq \frac{3}{2}w-2}(\config(\leq 3w/2,w))$.

  Every wheel with the same number of disks has the same shape, and there are
  finitely many shapes of filters of width $w$.  Therefore, $H_*(\config({-},w))$
  is finitely generated as a twisted algebra.
\end{proof}

\section{Betti number growth function}\label{sec:formula}

The paper~\cite{AKM} examines the growth of the dimension of $H_j(\config(n, w))$ for fixed $j$ and $w$ and varying $n$, and shows that it is asymptotically polynomial times exponential.  Having computed a basis for homology, we can now answer the question, is the dimension exactly equal to a polynomial times an exponential?  The answer is that it is a sum of such functions, each with a different base of the exponential.  The polynomials are integer-valued, and the theory of finite difference calculus states that every integer-valued polynomial in $n$ is an integer linear combination of binomial coefficient functions $\binom{n}{a}$ for various $a$.  Thus, in this section we prove the following theorem.

\begin{thm}\label{thm:count-growth}
For fixed $j$ and $w$, as a function of $n$ the dimension of $H_j(\config(n, w))$ is an integer linear combination of terms of the form $\binom{n}{a}b^{n-a}$, where $a$ and $b$ are nonnegative integers.
\end{thm}

Note that the case $b = 0$ is permitted, and represents the function that is equal to $1$ when $n = a$, and $0$ otherwise.

Our proof counts the critical cells from Remark~\ref{rmk:crit-cells}, and decomposes them into concatenation products of factors that are easier to count.  To describe how the counts transform under concatenation product, we introduce the following terminology.  A \emph{\textbf{cell family}} $F$ consists of, for each $n$, a set $F_n$ of cells from $\cel(n)$.  Its \emph{\textbf{counting function}} $f(n)$ is the number of cells in $F_n$.  The concatenation product $F \mid G$ of cell families $F$ and $G$ is the cell family consisting of all concatenation products of a cell in $F$ with a cell in $G$, taken with all possible disk labelings that preserve the order within each factor\footnote{This is closely related to the Day convolution defined in \S\ref{S:twist}, except that $F$ and $G$ need not be FB-subsets of the FB-set of cells of $\cel(-)$, and in our application will not be $S_n$-equivariant.}.  The following lemma shows that to prove our theorem, it suffices to consider separately the various independent factors of a concatenation product.

\begin{lem}\label{lem:egf}
Let $F$ and $G$ be two cell families, such that their counting functions $f(n)$ and $g(n)$ are integer linear combinations of terms of the form $\binom{n}{a}b^{n-a}$, where $a$ and $b$ are nonnegative integers.  Suppose that for every element of their concatenation product $F \ \vert\ G$, there is only one way to write it as the concatenation of an element of $F$ and an element of $G$.  Then the counting function of $F \ \vert\ G$ is also an integer linear combination of terms of the form $\binom{n}{a}b^{n-a}$.
\end{lem}

\begin{proof}
Let $f\vert g$ denote the counting function of $F\ \vert\ G$.  Then we have
\[(f\vert g)(n) = \sum_{i+j = n} \binom{i+j}{i} f(i)g(j),\]
which we refer to as the \emph{\textbf{labeled convolution}} of $f$ and $g$.  Exponential generating functions are convenient for dealing with these labeled convolutions; given the exponential generating functions $\sum_{i} f(i)\frac{x^i}{i!}$ for $f$ and $\sum_j g(j) \frac{x^j}{j!}$ for $g$, their product gives the exponential generating function $\sum_n (f\vert g)(n) \frac{x^n}{n!}$ for $f\vert g$.

We can compute the labeled convolution of two terms $\binom{n}{a_1}b_1^{n-a_1}$ and $\binom{n}{a_2}b_2^{n-a_2}$ by converting them to exponential generating functions, taking the product, and then converting the product back.  The exponential generating function of $\binom{n}{a}b^{n-a}$ is
\[\sum_n \binom{n}{a}b^{n-a}\frac{x^n}{n!} = \sum_n \frac{x^a}{a!} \frac{(bx)^{n-a}}{(n-a)!} = \frac{x^a}{a!}e^{bx},\]
so the product of exponential generating functions of $\binom{n}{a_1}b_1^{n-a_1}$ and $\binom{n}{a_2}b_2^{n-a_2}$ is 
\[\frac{x^{a_1 + a_2}}{a_1!a_2!}e^{(b_1 + b_2)x} = \binom{a_1+a_2}{a_1}\frac{x^{a_1+a_2}}{(a_1 + a_2)!}e^{(b_1 + b_2)x},\]
and so the labeled convolution of $\binom{n}{a_1}b_1^{n-a_1}$ and $\binom{n}{a_2}b_2^{n-a_2}$ is $\binom{a_1 + a_2}{a_1}\binom{n}{a_1+a_2}(b_1+b_2)^{n - (a_1 + a_2)}$.  (This identity can also be verified by constructing sets counted by each of the counting functions.)
\end{proof}

Using this lemma, we are ready to prove Theorem~\ref{thm:count-growth}.

\begin{proof}[Proof of Theorem~\ref{thm:count-growth}]
Given $j$ and $w$, the critical cells from Remark~\ref{rmk:crit-cells} form a cell family, and we want to find the counting function of this cell family.  To do this, it helps to count the ways to distribute the singletons separately from counting the ways to arrange the larger wheels.  We define a \emph{\textbf{skyline}} to be a critical cell for which the only singleton blocks are leaders of a filter.  Given any critical cell, we can find its associated skyline by deleting all the non-leader singletons and then shifting the disk labels down so they are consecutive.  For fixed $j$ and $w$, there are only finitely many possible skylines among the critical cells.  Thus it suffices to count the cell family consisting of all critical cells with a given skyline.

For each skyline, its cell family is the concatenation product of other cell families.  We define a \emph{\textbf{tadpole}} to be a critical cell with exactly one filter, which is at the far right, and we define a \emph{\textbf{tail}} to be a critical cell with no filters.  Every critical cell can be written uniquely as a concatenation of some number of tadpoles and possibly one tail on the right.  Thus, it suffices to count the cell family consisting of all tadpoles with a given skyline, and the cell family consisting of all tails with a given skyline.

A tadpole with a singleton as the leader of its filter may have other singletons to the left, but a tadpole with a larger wheel as the leader of its filter may not have any singletons.  Thus, given a tadpole skyline, if the leader of its filter is not a singleton, the counting function of its cell family returns $1$ when the input is the number of disks in the skyline, and $0$ otherwise.  If the leader of its filter is a singleton, then for any tadpole with that skyline, the leader of the filter must be disk $1$.  If $k$ is the total number of disks in the skyline, then the counting function of the tadpole skyline family is $\binom{n-1}{k-1}$, representing the number of ways to choose which disk labels appear in the skyline.  Given a tail skyline, if $k$ is the total number of disks in the skyline (possibly $0$), the counting function of its cell family is $\binom{n}{k}$.

Thus, for each skyline with a given $j$ and $w$, we can compute the counting function of its cell family by taking the labeled convolution of the counting functions of its tadpole skyline and tail skyline factors.  Because each factor has the desired form, so does the labeled convolution, by Lemma~\ref{lem:egf}.  The counting function of the family of all critical cells is the sum of the counting functions of all the skylines, and so it also has the desired form.
\end{proof}

\section{Cohomology ring}\label{sec:cohom}

We now give a basis for $H^j(\config(n,w))$ and describe the cup product
structure in terms of that basis.  We use a similar strategy to that used in
\cite{AKM} to find a lower bound for the dimensions of homology groups.  The main
tool is Poincar\'e--Lefschetz duality: for a compact $2n$-manifold with boundary
$(M,\partial M)$,
\[H^j(M) \cong H_{2n-j}(M,\partial M).\]
Moreover, when homology classes are realized by submanifolds (as they will be in
our case), then:
\begin{enumerate}[(i)]
\item The pairing between classes in $H^j(M)$ and $H_j(M)$ is given by the
  transverse signed intersection number between classes in
  $H_{2n-j}(M,\partial M)$ and $H_j(M)$.
\item The cup product between classes in $H^i(M)$ and $H^j(M)$ is given by the
  transverse intersection map
  \[\pitchfork: H_{2n-i}(M,\partial M) \otimes H_{2n-j}(M,\partial M)
  \to H_{2n-i-j}(M,\partial M).\]
\end{enumerate}
While $\config(n,w)$ as previously described is not a compact $2n$-manifold with
boundary, we can define a homotopy equivalent compact manifold with boundary:
\begin{defn}
  Let $M(n,w) \subset \mathbb{R}^{2n}$ be the configuration space of open disks
  of radius 1 in a strip of any finite length $N>n$ and width $w+\epsi$, for any
  $0<\epsi<1$.
\end{defn}
The boundary consists of those configurations in which a disk
touches either another disk or the boundary of the strip.  It has corners, but every point of the boundary has a neighborhood homeomorphic to a half-space.  (Without the addition of $\epsi$, a boundary configuration with $w$ vertically aligned disks would not have a neighborhood homeomorphic to a half-space.)

Alternatively, in an open manifold without boundary (such as the space of
configurations of $n$ points in $\mathbb{R} \times (0,1)$ of which no more than
$w$ are vertically aligned, as used in Theorem \ref{thm:emb}) Poincar\'e duality
gives an isomorphism $H^j(M) \cong H_{2n-j}^{BM}(M)$.  Here $H_*^{BM}$ indicates
\emph{Borel--Moore homology}, the homology of the complex of locally finite
chains; this is isomorphic to the homology of a compactification relative to the
added points.

\subsection{Basis for cohomology} \label{S:coho-basis}
We will describe a basis for $H_{2n-j}(M(n,w),\partial M(n,w))$ by associating
basis elements to critical $j$-cells of $\cel(n,w)$, as described in Remark
\ref{rmk:crit-cells}.  Given a critical cell with symbol $f$, we define the
submanifold $V(f) \subset M(n,w)$ as the set of configurations such that
\begin{enumerate}[(i)]
\item The disks in each block of $\sigma$ are lined up vertically in order.
\item If a block $b_1$ comes before $b_2$ in $f$, they have at least $w+1$
  elements combined, and one of them is a follower, then the column of disks
  labeled by elements of $b_1$ is to the left of that labeled by elements of
  $b_2$.
\end{enumerate}
To see that this is a submanifold, and that moreover
$\partial V(f)=V(f) \cap \partial M(n,w)$, notice that two columns of disks which
have at least $w+1$ elements combined cannot move past each other while still
satisfying condition (i).

To show that this is a basis, we describe the intersection pairing between these
submanifolds and our generators of $H_j(\config(n,w))$.
\begin{lem}\label{lem:int-pair}
  Let $f$ and $g$ be symbols of critical cells of $\cel(n,w)$.  Write $Z(g)$ for
  the basic cycle corresponding to $g$.  Then:
  \begin{enumerate}[(a)]
  \item $V(g) \pitchfork Z(g)=\pm 1$.
  \item If $V(g) \pitchfork Z(f) \neq 0$, then $g \preceq f$ according to the
    ordering in Remark \ref{rmk:crit-cells}.
  \end{enumerate}
\end{lem}
From the lemma, we see that under the ordering of the critical cells by $\prec$,
the intersection pairing is described by a triangular matrix.  Therefore, the
$V(g)$ form a basis for $H_{2n-j}(M(n,w),\partial M(n,w))$.
\begin{proof}
  We use the embedding of Theorem \ref{thm:emb} to associate $V(g)$ to a cellular
  $j$-cocycle $\nu(g)$ in $\cel(n,w)$.  The value of $\nu(g)$ on a $j$-cell is
  given by the signed intersection number of the embedded cell with $V(g)$.

  Suppose that a symbol $h$ is in the support of $\nu(g)$.  By comparing
  condition (i) of the definition of $V(g)$ to the construction of the embedding
  in Theorem \ref{thm:emb}, we see that $f$ must have the same set of blocks as
  $g$, although possibly in a different order; in particular, $h$ and $g$ are in
  the same layer.  Moreover, the barycenter of $h$ is the unique point of
  intersection.  Condition (ii) restricts the possible orderings.

  Now suppose that $h \neq g$.  We look at the first block, say $h_i$ and $g_i$,
  where they differ.  This means that $h_i$ is a block which occurs later in
  $g_i$, say $g_j$.  The structure of critical cells gives us the following
  possibilities:
  \begin{itemize}
  \item The block $g_i$ is a follower in $g$, and $g_j$ is not.  Since $g_i$ is a
    follower, it and $g_{i-1}$ have at least $w+1$ elements in total.  On the
    other hand, condition (ii) implies that $g_i$ and $g_j$ have at most $w$
    elements in total, so $g_{i-1}$ has more elements than $g_j$.  Since $g_j$ is
    not a follower, it must be a unicycle.  Then $g_{i-1} \prec g_j = h_i$ in the
    ordering on wheels, implying that $h_i$ is not a follower in $h$.  Therefore
    $g \prec h$.
  \item There are no followers between $g_i$ and $g_j$, inclusive.  Then $g_j$
    and $g_i$ are both unicycles and $g_j \prec g_i$ in the ordering on wheels.
    Therefore $h_i$ is not a follower and $h_i \prec g_i$.  Because the ordering
    on cells uses the reverse ordering on wheels, $g \prec h$.
  \item There is at least one follower between $g_i$ and $g_j$, and it is not
    $g_i$.  But $g_j$ can't be a follower, because if it were, $V(g)$ could not
    contain configurations with $g_j$ to the left of its leader.  Then $g_j$ and
    the first follower after $g_i$ appear in opposite orders in $g$ and $h$, so
    $g_j$ must have fewer elements than the first leader after $g_i$.  Since
    $g_j$ is not a follower, it is a unicycle, and the ordering implies that it
    must also have fewer elements than $g_i$.  Thus $h_i \prec g_i$, and
    therefore $g \prec h$.
  \end{itemize}
  In other words, we always have that $g \preceq h$.  On the other hand, we know
  that if a cell $h$ is in the support of $Z(f)$, then $h \preceq f$.  Therefore,
  $g \preceq f$.  This proves (b).

  From this we also know that $g$ is the unique cell which is in the support of
  both $\nu(g)$ and $Z(g)$.  Since the coefficient in both cases is $\pm 1$, this
  proves (a).
\end{proof}

\subsection{Cup product structure}
The cup product structure of $H^*(\config(n,w))$ is complicated, with many
indecomposables as well as many nontrivial products.  We start with some simple
observations.  First, the cohomology algebra of $H^*(\config(n))$ is
well-understood: it is generated by 1-dimensional classes.  A good description is
given by Sinha \cite{Sinha}.  Secondly, the pullback of $H^*(\config(n))$ to
$H^*(\config(n,w))$ along the inclusion map is a subalgebra and contains all
classes of degree less than $w-1$.  That means that all classes in
$H^{w-1}(\config(n,w))$ which are not pullbacks from $H^{w-1}(\config(n))$ are
indecomposable.  These include the basis elements corresponding to critical
cells which have one leader--follower pair with $w+1$ total elements and in
which all other blocks are singletons.

In higher degrees, the story is more complicated.  Recall that our basic cocycles
are carved out using three kinds of relations: coincidence between the
$x$-coordinates of two disks, vertical ordering of elements within a block and
horizontal ordering between blocks.  When we take a cup product, the coincidences
from both factors accumulate; see Table \ref{table:products} for some examples.
Likewise, when two blocks are ordered in a product, the ordering must be
``inherited'' from one of the factors.  If many pairs of blocks are ordered, this
may force the cohomology class to be indecomposable.

\begin{table}
  \begin{empheq}[box=\fbox]{align*}
  \nu(2\,1 \mid 6 \mid 5 \mid 4 \mid 3) \cup \nu(3\,2\mid 6 \mid 5 \mid 4 \mid 1)
  &= \nu(3\,2\,1 \mid 6 \mid 5 \mid 4) \\
  \nu(3\,2 \mid 6 \mid 5 \mid 4 \mid 1) \cup \nu(3\,1\mid 6 \mid 5 \mid 4 \mid 2)
  &= \nu(3\,2\,1 \mid 6 \mid 5 \mid 4)+\nu(3\,1\,2 \mid 6 \mid 5 \mid 4) \\
  \nu(3\,2\,1 \mid 6 \mid 5 \mid 4) \cup \nu(5\,4 \mid 6 \mid 3 \mid 2 \mid 1)
  &= \nu(3\,2\,1 \mid 5\,4 \mid 6) \\
  \nu(5\,4 \mid 6\,2\,3 \mid 1) \cup \nu(4\,1 \mid 6 \mid 5 \mid 3 \mid 2)
  &= \nu(5\,4\,1 \mid 6\,2\,3).
  \end{empheq}
\caption{Some examples of nontrivial cup products in $\config(6,4)$.  These are
easy to deduce using intersections of dual cycles.} \label{table:products}
\end{table}

The last example in Table \ref{table:products} illustrates an important class of
decomposable cohomology classes: those associated to critical cells consisting of only two blocks, where the total number of elements is at least $w+2$.  In other
words, a filter with more than $w+1$ elements always pairs with a decomposable
cohomology class.

We also describe an important set of cases in which cup products are zero.
\begin{prop}
  If there is a pair of labels $i$ and $j$ which are contained in the same block
  in both $f$ and $g$, then $V(f) \pitchfork V(g)=\emptyset$.
\end{prop}
\begin{proof}
  If the two labels are in opposite orders in the two blocks, then the
  intersection of the two cycles is empty.  If they are in the same order, then
  the intersection may be nonempty, but we can move it off itself by moving every
  point of $V(f)$ slightly in the $x_i$-direction (say, move the point $p$ by a
  distance of $\epsi d(p,\partial M(n,w))$ for some $\epsi>0$ small enough).
  After this operation, $V(f) \cap V(g)$ lies in the boundary of $M(n,w)$.
\end{proof}
In all other cases, we can compute cup products by looking at the intersections
of the associated submanifolds.
\begin{prop}
  If no two blocks in $f$ and $g$ have two labels in common, then $V(f)$ and
  $V(g)$ intersect transversely.
\end{prop}

\begin{proof}
Locally, $V(f)$ and $V(g)$ are linear subspaces of $\config(n)$, cut out by the linear equations constraining the $x$--coordinates in each block to coincide.  Thus, the dimension of $V(f)$ is $n$ (for the $y$--coordinates) plus the number of blocks in $f$ (for the $x$--coordinates), and similarly for $V(g)$.  In the intersection, we imagine starting with the constraints for $V(f)$ and including the constraints for $V(g)$ one block at a time.  Each block of size $k$ in $g$ merges $k$ different blocks in $f$ into one, so the net change in the number of blocks is $1 - k$.  In total, the number of blocks in $V(f) \cap V(g)$ is $\#\mathrm{blocks}(f) + \#\mathrm{blocks}(g) - n$, so the local codimension of $V(f) \cap V(g)$ is $2n - \#\mathrm{blocks}(f) - \#\mathrm{blocks}(g)$, which equals $\codim V(f) + \codim V(g)$.
\end{proof}

Finally, we show that there are many indecomposable elements in
$H^*(\config(n,w))$, but that they do not occur in the very highest degrees.
\begin{thm} \label{thm:indec}
  The ring $H^*(\config(n,w))$ has indecomposables:
  \begin{enumerate}[(a)]
  \item Only in degree 1, when $w=2$.
  \item In every degree between $1$ and $\lfloor n/2 \rfloor$ and no others, when
    $w=3$.
  \item In degree 1 and in every degree between $w-1$ and
    $\left\lfloor\frac{n+w-3}{2}\right\rfloor$ and no degree greater than
    $n-\bigl\lceil \frac{n}{w-1} \bigr\rceil$, when $w \geq 4$.
  \end{enumerate}
\end{thm}
When $w \geq 4$, indecomposables also seem to occur in most degrees below
$n-\bigl\lceil \frac{n}{w-1} \bigr\rceil$, but perhaps not all.
\begin{proof}
  We first show that every class of degree greater than
  $n-\bigl\lceil \frac{n}{w-1} \bigr\rceil$ is decomposable.  The proof does not
  depend on $w$.  Every cell with fewer than
  $\bigl\lceil \frac{n}{w-1} \bigr\rceil$ blocks has a block with $w$ elements,
  and therefore $V(f) \subset M(n,w)$ satisfies an equation of the form
  $x_{i_1}=\cdots=x_{i_w}$.  Therefore it is enough to show:
  \begin{lem}
    Suppose that $V \subset M(n,w)$ is a connected compact submanifold of
    dimension at most $2n-w$, satisfying $\partial V \subset \partial M(n,w)$,
    which is cut out by relations of the form $x_i=x_j$, $y_i<y_j$, and $x_i<x_j$.
    Suppose furthermore that $V$ satisfies $x_{i_1}=\cdots=x_{i_w}$ for some set of
    indices $i_1,\ldots,i_w$.  Then $V$ is the transverse intersection of two
    proper compact submanifolds satisfying $\partial V \subset \partial M(n,w)$.
  \end{lem}
  \begin{proof}
    Define $W_1$ to be the connected component of $\{x_{i_1}=\cdots=x_{i_w}\}$
    which contains $V$; this exists since $V$ is connected.  The defining
    relations of $W_1$ are:
    \begin{itemize}
    \item $x_{i_k}=x_{i_\ell}$, for each $k \neq \ell$;
    \item $y_{i_k}<y_{i_\ell}$ or $y_{i_k}>y_{i_\ell}$, for each $k \neq \ell$;
    \item $x_j<x_{i_k}$ or $x_j>x_{i_k}$, whenever $j \neq i_k$ for any $k$.
    \end{itemize}
    Let $W_2$ be the submanifold cut out by all defining relations of $V$ which
    involve pairs of points constrained to be to the same side of
    $x_{i_1},\ldots,x_{i_w}$.  Then $V=W_1 \cap W_2$, since every necessary
    relation defining $V$ is a defining relation of either $W_1$ or $W_2$.
    Moreover, by counting the number of defining equalities we immediately see
    that the intersection is transverse.
  \end{proof}
  Applying this to each connected component of $V(f)$, we get a decomposition of
  the corresponding cohomology class as a sum of cup products.

  We now build indecomposable cocycles when $w \geq 4$; the proof for $w=3$ will
  be similar, but not identical.  We will show that basic cocycles corresponding
  to critical cells of certain shapes are indecomposable.  Specifically, we
  consider a critical cell $f$ which starts with some number $r$ of blocks with
  $2$ elements, followed by one block (a follower) with $w-1$ elements, and where
  the remaining blocks are singletons.  The degree of such an $f$ is $r+w-2$, and
  $r$ can be any number between $1$ and $\frac{n-w+1}{2}$, giving us all degrees
  between $w-1$ and $\bigl\lfloor\frac{n+w-3}{2}\bigr\rfloor$.  It is clear that
  there are critical cells of any such shape; in particular, we select $f$ to be the cell with all entries in order from greatest to least.

  We will show that the corresponding cohomology class $\nu(f)$ is indecomposable
  by induction on $r$.  We will do this by constructing a cycle with which
  $\nu(f)$ pairs nontrivially, but any decomposable class pairs trivially.

  To construct this cycle, first let $R$ be the set of blocks of $f$ which are
  left of the follower.  For any subset $S \subseteq R$, there is a critical cell
  $f_S$ in which the blocks in $S$ are moved to the right of the follower, and
  otherwise the ordering on blocks is the same.  Let $Z_S$ be the concatenation product of the toroidal cycles obtained by interpreting each block as a wheel and spinning it.  This cycle is represented by a map $g_S:T^{r+w-2} \to \config(n,w)$.  The $Z_S$ represent linearly independent homology classes, since they generate a $2^r$-dimensional subspace of $H_{r+w-2}(\config(n,w))$: for every $S \subseteq R$, the basic cycle $Z(f_S)$ is a linear combination of them, given by $Z_S$ itself if $S = R$, or otherwise the difference between $Z_S$ and $Z_{S'}$, where $S'$ is the union of $S$ with the greatest block not in $S$.  We will show that for every decomposable class $\nu$,
  \begin{equation} \label{eqn:sum}
    \sum_{S \subseteq R} (-1)^{|S|}\langle \nu, Z_S \rangle=0.
  \end{equation}
  In particular, $\nu$ cannot pair nontrivially with exactly one of the cycles
  $Z_S$.  On the other hand, $\nu(f)$ pairs nontrivially with $Z_S$ if and only
  if $S=\emptyset$.

  It is enough to show this equation holds for every pairwise cup product,
  $\nu=\alpha \cup \beta$.  We study the pullbacks of such a pairwise cup product
  along each $g_S$.  Suppose that $\alpha$ is of degree $p$, and write
  \[T^{r+w-2}=\prod_{i \in R \sqcup R'} S^1_i,\]
  where $R'$ is the set of degrees of freedom of the $(w-1)$-element wheel.  To
  understand $g_S^*\alpha$, it is enough to understand how it pairs with
  $T^P=\prod_{i \in P} S^1_i$, for each $p$-element set $P \subseteq R \sqcup R'$.
  Then
  \begin{equation} \label{eqn:decomp}
    \langle \nu, Z_S \rangle=\sum_{P \sqcup Q=R \sqcup R'}
    \langle g_S^*\alpha, [T^P] \rangle \cdot \langle g_S^*\beta, [T^Q] \rangle.
  \end{equation}

  We now show that these decompositions are not independent for different $S$.
  \begin{lem}
    If $P \cap R' \neq R'$, then $\langle g_S^*\alpha, [T^P]\rangle$ is the same
    for every $S$.
  \end{lem}
  \begin{proof}
    It suffices to show that the pushforward cycles $(g_S)_*[T^P]$ are homologous
    regardless of $S$.  In fact, the corresponding maps $T^P \to \config(n,w)$
    are homotopic.  We can compress the allowed movements of the $(w-1)$-element
    wheel into a subset of the strip of width $w-2$, letting wheels of width $2$
    pass by.  Clearly wheels of width 2 can also pass by each other since we are
    assuming $w \geq 4$.  Using this set of motions, we can construct a homotopy
    between any two such maps.
  \end{proof}
  \begin{lem}
    If $P \cap R'=R'$, then
    $\langle g_S^*\alpha, [T^P]\rangle=\langle g_{S'}^*\alpha, [T^P]\rangle$
    whenever $(S \bigtriangleup S') \cap P=\emptyset$ (where $\bigtriangleup$
    indicates symmetric difference).
  \end{lem}
  \begin{proof}
    Again, it suffices to show that $(g_S)_*[T^P]$ and $(g_{S'})_*[T^P]$ are
    homologous, and again the corresponding maps $T^P \to \config(n,w)$ are
    homotopic.  We can build the homotopy by moving the individual disks
    comprising blocks not in $P$ around the $(w-1)$-element set.
  \end{proof}
  Together, the lemmas imply that for any pair of nonempty complementary sets
  $P,Q \subset R \sqcup R'$, the quantity
  \[\langle g_S^*\alpha, [T^P] \rangle \cdot \langle g_S^*\beta, [T^Q] \rangle\]
  is independent of whether $i \in S$ for at least one $i \in R$.  In particular,
  \[\sum_{S \subseteq R} (-1)^{|S|}\langle g_S^*\alpha, [T^P] \rangle \cdot
  \langle g_S^*\beta, [T^Q] \rangle=0.\]
  From here we get \eqref{eqn:sum} by summing over all pairs $P,Q$ and using
  \eqref{eqn:decomp}.

  Finally we deal with the case $w=3$.  In this case, we consider critical cells
  composed of $r$ two-element blocks (with the bigger element first) followed by
  $n-2r$ singletons, for some $1 \leq r \leq \lfloor n/2 \rfloor$.  Such a cell
  is critical if and only if the singletons are in reverse order.  In particular,
  every permutation $\sigma$ of the set of two-element blocks gives a critical
  cell $f_\sigma$.  For each $\sigma \in S_r$, let $Z_\sigma$ be an $r$-cycle,
  represented by a map $g_\sigma:T^r \to \config(n,w)$, obtained by arranging the
  blocks in the order dictated by $\sigma$ and spinning each of them.  These
  cycles are linearly independent in homology since the basic cycles $Z(f_\sigma)$
  are all linear combinations of them.

  We define a function $\mu:S_r \to \mathbb{Z}$ as follows.  Consider the
  expression $[\cdots[[1,2],3], \cdots r]$.  If $\sigma$ cannot be obtained from
  this by commuting some of the brackets (equivalently, by spinning the wheel
  $1\,2\,\cdots\,r$, as in Figure \ref{fig:wheel2}), then $\mu(\sigma)=0$.  If it
  can, then $\mu(\sigma)=(-1)^c$ where $c$ is the number of commutations
  required.  We will show that for every decomposable class $\nu$,
  \begin{equation} \label{eqn:sum2}
    \sum_{\sigma \in S_r} \mu(\sigma)\langle \nu, Z_\sigma \rangle=0.
  \end{equation}
  In particular, $\nu$ cannot pair nontrivially with exactly one of the
  $Z_\sigma$, and therefore there is an indecomposable class of degree $r$.

  It is enough to show the equation for every pairwise cup product,
  $\nu=\alpha \cup \beta$.  Write $T^r=\prod_{i=1}^r S^1_i$; for every
  $P \subset \{1,\ldots,r\}$, write $T^P=\prod_{i \in P} S^1_i$.  Then
  \begin{equation} \label{eqn:decomp2}
    \langle \nu, Z_\sigma \rangle=\sum_{P \sqcup Q=\{1,\ldots,r\}}
    \langle g_\sigma^*\alpha, [T^P] \rangle
    \cdot \langle g_\sigma^*\beta, [T^Q] \rangle.
  \end{equation}
  Once again, these decompositions are not independent for different $\sigma$:
  \begin{lem}
    Whenever $\sigma$ and $\tau$ impose the same ordering on elements of $P$,
    \[\langle g_\sigma^*\alpha,[T^P]\rangle=\langle g_\tau^*\alpha,[T^P]\rangle.\]
  \end{lem}
  \begin{proof}
    To show that $(g_\sigma)_*[T^P]$ is homologous to $(g_\tau)_*[T^P]$, we
    construct a homotopy between the corresponding maps $T^P \to \config(n,w)$.
    This involves moving around the individual disks comprising the blocks not in
    $P$ to make sure they are in the right order.
  \end{proof}
  The lemma implies that for any pair of nonempty complementary sets
  $P,Q \subset \{1,\ldots,r\}$,
  \[\langle g_\sigma^*\alpha, [T^P] \rangle \cdot \langle g_\sigma^*\beta, [T^Q] \rangle=\langle g_\tau^*\alpha, [T^P] \rangle \cdot \langle g_\tau^*\beta, [T^Q] \rangle\]
  if $\tau$ and $\sigma$ differ by commuting one of the brackets of
  $[\cdots[[1,2],3], \cdots r]$: specifically, the innermost bracket in which the
  right side is in $Q$ if $1 \in P$, or vice versa.  In particular,
  \[\sum_{\sigma \in S_r} \mu(\sigma)\langle g_\sigma^*\alpha, [T^P] \rangle \cdot
  \langle g_\sigma^*\beta, [T^Q] \rangle=0.\]
  From here we get \eqref{eqn:sum2} by summing over all pairs $P,Q$ and using
  \eqref{eqn:decomp2}.
\end{proof}

\section{Persistent homology}\label{sec:PH}

The majority of this paper has investigated the properties of $\config(n,w)$ as
we increase $n$ and keep $w$ fixed.  But we can also look at what happens when
$w$ grows.  Since $\config(n,w)$ naturally injects into $\config(n,w+1)$, forming
a filtration, the right framework for understanding this is
\strong{persistent homology}, which considers homology for all $w$ at once, together with the maps induced by the inclusions.  We give a short introduction to the
machinery; for more details, see \cite{EdH,ZC}.

Specifically, we can regard $\bigoplus_{w} H_*(\config(n, w))$ as a
$\mathbb{Z}[t]$-module in which multiplication by $t$ corresponds to applying
the maps $H_*(\config(n, w)) \rightarrow H_*(\config(n, w+1))$ induced by the
inclusions $\config(n, w) \hookrightarrow \config(n, w+1)$.  We denote this $\mathbb{Z}[t]$-module by $PH_*(\config(n, *))$.  Similarly, $\bigoplus_w H^*(\config(n, w))$ is a $\mathbb{Z}[t]$-module in which multiplication by $t$ corresponds to applying the pullback maps $H^*(\config(n, w)) \rightarrow H^*(\config(n, w-1))$, and we denote this $\mathbb{Z}[t]$-module by $PH^*(\config(n, *))$.

A \strong{cyclic} summand of $PH_*(\config(n, *))$ or $PH^*(\config(n, *))$ is generated by a single element that is nonzero for some interval of values of $w$.  It is standard to refer to a cyclic summand as a \strong{bar}, and to the endpoints of the corresponding interval as the values $w$ of its \strong{birth} and \strong{death}.  A decomposition of $PH_*(\config(n, *))$ or $PH^*(\config(n, *))$ into cyclic summands means selecting $\mathbb{Z}$--bases for the various values of $w$ in a way that agrees with the maps between the values of $w$.

The fundamental theorem for modules over a PID guarantees that a persistence
module with coefficients in a field will always decompose into cyclic summands.
No such guarantee exists for integral persistence modules: for example, one could
have a single class born at one time and later become divisible by $2$, yielding
a module isomorphic to the ideal $(2,t) \subset \mathbb{Z}[t]$.

Theorems~\ref{thm:PH:basis} and~\ref{thm:PH:barlength}, which we state and prove below, together prove Theorem \ref{thm:PH}.

\begin{thm}\label{thm:PH:basis}
  The homology basis elements given in Theorem~\ref{thm:basis} induce a
  decomposition of $PH_*(\config(n, *))$ as the direct sum of cyclic
  $\mathbb{Z}[t]$--modules.  The cohomology basis elements from
  \S\ref{S:coho-basis} give a decomposition of $PH^*(\config(n, *))$ as the
  direct sum of cyclic $\mathbb{Z}[t]$--modules.
\end{thm}

To prove the theorem, it suffices to show that when a given basis element stops being in the basis, it also becomes zero in homology or cohomology.  To verify this for homology, we show the corresponding statement for weighted no--$(w+1)$--equal spaces.

\begin{thm} \label{thm:PH:weighted}
For the weighted no--$(w+1)$--equal spaces, the homology basis elements from Theorem~\ref{thm:basis} give a $\mathbb{Z}[t]$--basis for $PH_*(\no_*(n, \mathcal{W}))$.
\end{thm}
\begin{proof}
  Recall that a basic cycle
  consists of a product of single vertices corresponding to singletons and
  boundaries of permutohedral cells corresponding to leader--follower pairs.

  The key fact is that as we increase $w$, a particular cell stays critical as
  long as for every leader--follower pair, the weights add up to at least $w+1$.
  But once they add up to only $w$ for some leader--follower pair, that means
  that the corresponding permutohedron boundary is filled in, and so the cycle
  becomes a boundary.  Therefore every cell of $P(n)$ which is critical in
  $P(n,\mathcal{W},w)$ for some $w$ corresponds to a direct summand of the
  persistence module.
\end{proof}

\begin{proof}[Proof of Theorem~\ref{thm:PH:basis}]
  Theorem \ref{thm:PH:basis} follows easily from Theorem \ref{thm:PH:weighted}.
  This is because the splitting
  \[H_*(\cel(n,w))=
  \bigoplus_{\sigma \in S_n} H_{*-\#\sigma}(P(n-\#\sigma,\mathcal{W}(\sigma),w))\]
  of Theorem \ref{thm:split} is natural with respect to increasing $w$ (even on
  the chain level, as one readily sees).  Therefore
  \[PH_*(\cel(n,*))=
  \bigoplus_{\sigma \in S_n} PH_{*-\#\sigma}(P(n-\#\sigma,\mathcal{W}(\sigma),*)).\]
  
  In particular, an element leaves the basis exactly when one of its filters has the wheel sizes adding up to at most $w$, and the correspondence with the no--$(w+1)$--equal homology proves that the element is null-homologous in this case.
  
For cohomology, the restriction map $H^*(\config(n, w)) \rightarrow H^*(\config(n, w-1))$ corresponds to intersecting each basis element $V(g)$ of $H_{2n-*}(M(n, w), \del M(n, w))$ with the smaller space $M(n, w-1)$.  When $w$ gets too small for $V(g)$ to be in the basis, it is because some block of $g$ has more than $w$ elements; thus, the intersection of $V(g)$ with this $M(n, w)$ is empty, and the restricted cohomology class is zero.  Thus, it is also true for cohomology that when a given element stops being in the basis, it becomes zero.
\end{proof}
\begin{rmk}
  We can also explore how the persistence module structure on homology and
  cohomology interacts with other structures we have discussed.
  \begin{enumerate}
  \item Cohomological persistence does not play nicely with the cup product
    structure: frequently a class is born indecomposable at time $w$ and becomes
    decomposable at time $w-1$.  For example, in the notation of the previous
    section, we have the relation
    \[\nu(5\,4 \mid 6\,2\,3 \mid 1) \cup \nu(4\,1 \mid 6 \mid 5 \mid 3 \mid 2)
    = \nu(5\,4\,1 \mid 6\,2\,3),\]
    in $\config(6,3)$ and $\config(6,4)$, but $\nu(5\,4\,1 \mid 6\,2\,3)$ is
    indecomposable in $\config(6,5)$.
  \item The concatenation product is perfectly well-defined on persistence
    modules.  So we can think of $PH_*(\config({-},*))$ as a graded twisted
    algebra in $\mathbb{Z}[t]$-modules!  The $\mathbb{Z}[t]$-module structure of
    this algebra is relatively easy to understand as we have seen above, but it
    does not interact in a nice way with the symmetric group action.  Moreover,
    unlike the algebras $H_*(\config({-},w))$ for fixed $w$, it is not finitely
    generated as an algebra.  For these reasons, we do not study this structure
    further.
  \end{enumerate}
\end{rmk}

\subsection{Asymptotics of the persistence module}
A main goal of \cite{AKM} was to understand the growth of Betti numbers of
$\config(n,w)$ as $n$ increases.  Now that we have described the persistence
module of $\config(n,*)$, we can refine this: as the number of disks increases,
we keep track of the number of bars of different lengths.  It turns out that the
number of short bars grows faster and eventually dominates the number of longer
bars.  We make this precise in the following theorem.

\begin{thm}\label{thm:PH:barlength}
  Each $\mathbb{Z}[t]$--basis element of $PH_*(\config(n, *))$ born at $w=w_0$
  either persists for all $w > w_0$ or dies by $w=2w_0$.  For each $j$ and $w_0$,
  either the maps
  \[H_j(\config(n, w_0)) \rightarrow H_j(\config(n, w))\]
  are isomorphisms for all $w > w_0$ and all $n$, or the fraction of basis
  elements of $H_j(\config(n, w_0))$ that persist to $H_j(\config(n, w_0+1))$
  approaches $0$ as $n$ approaches $\infty$.
\end{thm}

We note that because each homology basis element is matched to a cohomology basis element in the same degree with the same birth and death, the theorem for homology immediately implies a corresponding statement for cohomology, which we do not include here.

For the second statement of Theorem~\ref{thm:PH:barlength}, it helps to have an asymptotic estimate of the dimension of $H_j(\config(n, w))$ as $n$ approaches $\infty$.  Examining the basis for homology and estimating the number of elements recovers the following theorem.

\begin{thm}[\cite{AKM}]\label{thm:AKM:asymptotic}
If $w \geq 2$ and $0 \leq j \leq w-2$, then the inclusion of $\config(n, w)$ into the configuration space of points in the plane induces an isomorphism on $H_j$.  If $w \geq 2$ and $j \geq w-1$, then there are positive constants $c_1$ and $c_2$, depending on $w$ and $j$, such that the following is true.  Write $j = q(w-1) + r$ with $q \geq 1$ and $0 \leq r < w-1$.  Then 
\[c_1\cdot n^{qw+2r}(q+1)^n \leq \dim H_j(\config(n, w)) \leq c_2\cdot n^{qw+2r}(q+1)^n.\]
\end{thm}

\begin{proof}[Proof of Theorem~\ref{thm:PH:barlength}]
The basis elements that persist indefinitely are those with no filters.  Each filter can be written as a leader--follower pair, and any leader--follower pair that appears at time $w_0$ has total weight at most $2w_0$, because the leader block has at most the weight of the follower block.  Thus it only remains a filter until at most $w = 2w_0$.  This proves the first statement of the theorem.
  
For the second statement, by Theorem~\ref{thm:AKM:asymptotic} it suffices to show that for $j \geq w-1$, the number of basis elements of $H_j(\config(n, w))$ that persist to $H_j(\config(n, w+1))$ grows polynomially in $n$.  We observe that if a filter for $w$ contains any wheel of size $1$, then its total size is $w+1$, so it does not remain a filter for $w+1$.  Thus, in any basis element of $H_j(\config(n, w))$ that persists, every filter must contain only wheels of size at least $2$.  One rule for being a basis element is, ``Every wheel immediately to the left of a filter is greater than the least wheel in the filter,'' so this then implies that in any basis element of $H_j(\config(n, w))$ that persists, all of the wheels of size $1$ appear to the right of all other wheels and filters.

To estimate the number of such basis elements, we simply count the cell symbols that end with (at least) $n-2j$ singleton blocks in descending order.  There are $\binom{n}{2j}\cdot (2j)!\cdot 2^{2j-1}$ such symbols, and that function grows polynomially in $n$.  Because the total number of basis elements of $H_j(\config(n, w))$ grows exponentially in $n$, asymptotically almost all of the basis elements do not persist to $w+1$.
\end{proof}

\section{Relations in the twisted algebra and $\FI_d$-modules} \label{S:FId}

The twisted algebra structure of $H_*(\cel(-, w))$ is unusual because many pairs of elements do not commute.  In particular, there are some elements that do not commute with the $0$--cycles coming from a single disk: we think of these as \strong{barriers} that prevent singleton disks from passing back and forth.  This non-commuting with singleton $0$--cycles is the main reason for Theorem~\ref{thm:AKM:asymptotic}: a given homology element in degree $j$ can be written as the concatenation product of up to $\bigl\lfloor\frac{j}{w-1}\bigr\rfloor$ barriers, giving up to $1+\bigl\lfloor\frac{j}{w-1}\bigr\rfloor$ non-equivalent ways to insert a new disk as a singleton.

One algebraic object that exhibits this kind of exponential growth is an $\FI_d$--module.  The best example for understanding the idea of an $\FI_d$--module is the $j$th homology of the configuration space of $n$ disks on the disjoint union of $d$ planes.  Each additional disk can be added to any of the $d$ planes.  Sam and Snowden define $\FI_d$--modules for the first time in~\cite{Sam17}.  In~\cite{Ramos17}, Ramos shows that finitely generated $\FI_d$--modules satisfy a notion of generalized representation stability, and in~\cite{Ramos19} he shows that the homology groups of a certain kind of graph configuration space are finitely generated $\FI_d$--modules.

We show in this section that the homology of unweighted no--$(w+1)$--equal spaces and of $\config(n, 2)$ are both $\FI_d$--modules.  On the other hand, we give an example to show that the homology of $\config(n, w)$ for $w > 2$ is probably not well-described via $\FI_d$--modules, since there seems to be no consistent way of decomposing homology classes as products of barriers.

Formally, the category $\FI_d$ has one object $[n] = \{1, \ldots, n\}$ for each natural number $n$.  The morphisms are pairs $(\varphi, c)$, where $\varphi$ is an injection, say, from $[n]$ to $[m]$, and $c$ is a $d$--coloring on the complement of the image of $\varphi$; that is, $c$ is a map from $[m] \setminus \varphi([n])$ to a set of size $d$, which in this paper we choose to be $\{0, 1, \ldots, d-1\}$.  The morphisms compose as illustrated in Figure~\ref{fig-compose}: for each element colored by the first morphism, in the composition, the image of that element under the second morphism is the one that gets that color.  (In the picture, the color of a given element is shown in a diamond just above the element.)  More formally, if $(\varphi, c) \co [n_1] \rightarrow [n_2]$ and $(\varphi', c') \co [n_2] \rightarrow [n_3]$ are two morphisms, then we have
\[(\varphi', c') \circ (\varphi, c) = (\varphi' \circ \varphi, c''),\]
where $c''(i)$ is equal to $c'(i)$ if $i \not\in \varphi'([n_2])$, and is equal to $c(\varphi'^{-1}(i))$ if $i \in \varphi'([n_2])$. 

\begin{figure}
\begin{center}
\begin{tikzpicture} [scale=.5, emp/.style={inner sep = 0pt, outer sep = 0pt}, >=stealth]
\node[emp] at (1, 7) {$1$};
\node[emp] at (2, 7) {$2$};
\node[emp] at (3, 7) {$3$};
\draw (.5, 6.5)--(3.5, 6.5)--(3.5, 7.5)--(.5, 7.5)--cycle;
\draw (1.5, 6.5)--(1.5, 7.5) (2.5, 6.5)--(2.5, 7.5);

\draw[->] (1, 6.5)--(3, 4.5);
\draw[->] (2, 6.5)--(2, 4.5);
\draw[->] (3, 6.5)--(4, 4.5);
\draw (1, 4.5)--(1.5, 5)--(1, 5.5)--(.5, 5)--cycle;
\node[emp] at (1, 5) {$2$};

\node[emp] at (1, 4) {$1$};
\node[emp] at (2, 4) {$2$};
\node[emp] at (3, 4) {$3$};
\node[emp] at (4, 4) {$4$};
\draw (.5, 3.5)--(4.5, 3.5)--(4.5, 4.5)--(.5, 4.5)--cycle;
\draw (1.5, 3.5)--(1.5, 4.5) (2.5, 3.5)--(2.5, 4.5) (3.5, 3.5)--(3.5, 4.5);

\draw[->] (1, 3.5)--(2, 1.5);
\draw[->] (2, 3.5)--(1, 1.5);
\draw[->] (3, 3.5)--(4, 1.5);
\draw[->] (4, 3.5)--(5, 1.5);
\draw (3, 1.5)--(3.5, 2)--(3, 2.5)--(2.5, 2)--cycle;
\node[emp] at (3, 2) {$1$};

\node[emp] at (1, 1) {$1$};
\node[emp] at (2, 1) {$2$};
\node[emp] at (3, 1) {$3$};
\node[emp] at (4, 1) {$4$};
\node[emp] at (5, 1) {$5$};
\draw (.5, .5)--(5.5, .5)--(5.5, 1.5)--(.5, 1.5)--cycle;
\draw (1.5, .5)--(1.5, 1.5) (2.5, .5)--(2.5, 1.5) (3.5, .5)--(3.5, 1.5) (4.5, .5)--(4.5, 1.5);


\node[emp] at (6, 4) {$=$};

\node[emp] at (7+1, 5.5) {$1$};
\node[emp] at (7+2, 5.5) {$2$};
\node[emp] at (7+3, 5.5) {$3$};
\draw (7+.5, 5)--(7+3.5, 5)--(7+3.5, 6)--(7+.5, 6)--cycle;
\draw (7+1.5, 5)--(7+1.5, 6) (7+2.5, 5)--(7+2.5, 6);

\draw[->] (7+1, 5)--(7+4, 3);
\draw[->] (7+2, 5)--(7+1, 3);
\draw[->] (7+3, 5)--(7+5, 3);
\draw (7+2, 3)--(7+2.5, 3.5)--(7+2, 4)--(7+1.5, 3.5)--cycle;
\node[emp] at (7+2, 3.5) {$2$};
\draw[fill=white] (7+3, 3)--(7+3.5, 3.5)--(7+3, 4)--(7+2.5, 3.5)--cycle;
\node[emp] at (7+3, 3.5) {$1$};

\node[emp] at (7+1, 2.5) {$1$};
\node[emp] at (7+2, 2.5) {$2$};
\node[emp] at (7+3, 2.5) {$3$};
\node[emp] at (7+4, 2.5) {$4$};
\node[emp] at (7+5, 2.5) {$5$};
\draw (7+.5, 2)--(7+5.5, 2)--(7+5.5, 3)--(7+.5, 3)--cycle;
\draw (7+1.5, 2)--(7+1.5, 3) (7+2.5, 2)--(7+2.5, 3) (7+3.5, 2)--(7+3.5, 3) (7+4.5, 2)--(7+4.5, 3);
\end{tikzpicture}
\end{center}
\caption{To compose two morphisms in $\FI_d$, we have $(\varphi', c') \circ (\varphi, c) = (\varphi' \circ \varphi, c'')$, where $c''(i)$ is equal to $c'(i)$ if $i$ is not in the image of $\varphi'$ (for instance, $i = 3$ has color $1$ in the example shown) and is equal to $c(\varphi'^{-1}(i))$ if $i$ is in the image of $\varphi'$ (for instance, $i = 2$ has color $2$ in the composition because $c(1) = 2$ and $\varphi'(1) = 2$).}\label{fig-compose}
\end{figure}
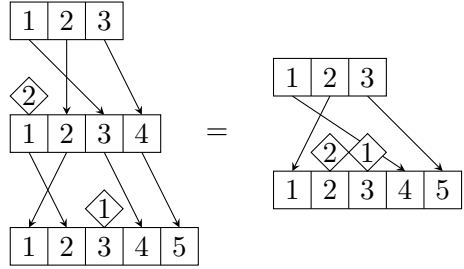

An \strong{$\FI_d$--module} $M$ over a commutative ring $k$ is defined to be a functor from $\FI_d$ to $k$--modules; that is, we have a $k$--module $M_n$ for each $n$, and for each $(\varphi, c) \co [n] \rightarrow [m]$, we have a corresponding $k$--module map $(\varphi, c)_* \co M_n \rightarrow M_m$.  In the present paper we use $k = \mathbb{Z}$.  An $\FI_d$--module is \strong{finitely generated} if there exists a finite set of elements $x_1, \ldots, x_r \in \bigsqcup_{n = 1}^\infty M_n$ such that the only $\FI_d$--submodule of $M$ containing $x_1, \ldots, x_r$ is $M$ itself.  Figure~\ref{fig-strip-morph} sketches the $\FI_{j+1}$--module structure for $H_j(\config(n, 2))$: the colors of the disks, shown in the picture as the numbers in the diamonds, indicate where to insert the disks between barriers.

\begin{figure}
\begin{center}
\begin{tikzpicture}[scale=.5, emp/.style={inner sep = 0pt, outer sep = 0pt}, >=stealth]
\node[emp] at (1, 1.5) {$1$};
\node[emp] at (2, 1.5) {$2$};
\node[emp] at (3, 1.5) {$3$};
\node[emp] at (4, 1.5) {$4$};
\node[emp] at (5, 1.5) {$5$};
\draw (.5, 1)--(5.5, 1)--(5.5, 2)--(.5, 2)--cycle (1.5, 1)--(1.5, 2) (2.5, 1)--(2.5, 2) (3.5, 1)--(3.5, 2) (4.5, 1)--(4.5, 2);

\node[emp] at (1, -1.5) {$1$};
\node[emp] at (2, -1.5) {$2$};
\node[emp] at (3, -1.5) {$3$};
\node[emp] at (4, -1.5) {$4$};
\node[emp] at (5, -1.5) {$5$};
\node[emp] at (6, -1.5) {$6$};
\node[emp] at (7, -1.5) {$7$};
\node[emp] at (8, -1.5) {$8$};
\draw (.5, -2)--(8.5, -2)--(8.5, -1)--(.5, -1)--cycle (1.5, -2)--(1.5, -1) (2.5, -2)--(2.5, -1) (3.5, -2)--(3.5, -1) (4.5, -2)--(4.5, -1) (5.5, -2)--(5.5, -1) (6.5, -2)--(6.5, -1) (7.5, -2)--(7.5, -1);

\draw[->] (1, 1)--(3, -1);
\draw[->] (2, 1)--(2, -1);
\draw[->] (3, 1)--(5, -1);
\draw[->] (4, 1)--(4, -1);
\draw[->] (5, 1)--(8, -1);

\node[emp] at (1, -.5) {$1$};
\draw (1, -1)--(1.5, -.5)--(1, 0)--(.5, -.5)--cycle;
\draw[fill=white] (6, -1)--(6.5, -.5)--(6, 0)--(5.5, -.5)--cycle;
\node[emp] at (6, -.5) {$0$};
\draw[fill=white] (7, -1)--(7.5, -.5)--(7, 0)--(6.5, -.5)--cycle;
\node[emp] at (7, -.5) {$1$};

\draw[->] (0, 2.5)--(0, -2.5);


\draw (-5, -3)--(8, -3);
\draw (-5, -5)--(8, -5);
\node[emp] (d6) at (-3.5, -4.5) {$6$};
\node[emp] (d5) at (-1.5, -4.5) {$5$};
\node[emp] (d2) at (.5, -4.5) {$2$};
\node[emp] (d3) at (.5, -3.5) {$3$};
\node[emp] (d1) at (2.5, -4.5) {$1$};
\node[emp] (d7) at (4.5, -4.5) {$7$};
\node[emp] (d8) at (6.5, -4.5) {$8$};
\node[emp] (d4) at (6.5, -3.5) {$4$};

\draw (d1) circle (.5);
\draw (d2) circle (.5);
\draw (d3) circle (.5);
\draw (d4) circle (.5);
\draw (d5) circle (.5);
\draw (d6) circle (.5);
\draw (d7) circle (.5);
\draw (d8) circle (.5);

\draw[->] (.5+.52, -5+.7) arc (-30:30:.6);
\draw[->] (.5-.52, -5+1.3) arc (150:210:.6);
\draw[->] (6.5+.52, -5+.7) arc (-30:30:.6);
\draw[->] (6.5-.52, -5+1.3) arc (150:210:.6);


\draw (-5, 3)--(8, 3);
\draw (-5, 5)--(8, 5);
\node[emp] (d5) at (-1.5, 3.5) {$3$};
\node[emp] (d2) at (1.5, 3.5) {$2$};
\node[emp] (d3) at (1.5, 4.5) {$1$};
\node[emp] (d8) at (4.5, 3.5) {$5$};
\node[emp] (d4) at (4.5, 4.5) {$4$};

\draw (d2) circle (.5);
\draw (d3) circle (.5);
\draw (d4) circle (.5);
\draw (d5) circle (.5);
\draw (d8) circle (.5);

\draw[->] (1.5+.52, 3+.7) arc (-30:30:.6);
\draw[->] (1.5-.52, 3+1.3) arc (150:210:.6);
\draw[->] (4.5+.52, 3+.7) arc (-30:30:.6);
\draw[->] (4.5-.52, 3+1.3) arc (150:210:.6);

\end{tikzpicture}
\end{center}
\caption{When applying this $\FI_3$--morphism to a class in $H_2(\config(5, 2))$, we insert the disks with color--$k$ labels immediately after the $k$th circling pair.}\label{fig-strip-morph}
\end{figure}
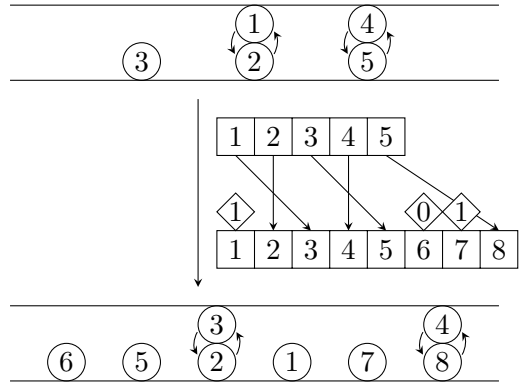

\subsection{No--$(w+1)$--equal spaces}

\begin{thm}\label{thm-unwt-gen-rel}
  The graded twisted algebra $H_*(\no_{w+1}(-))$ has a presentation with two
  generators and two relations.  The generators are the representations spanned
  by a singleton $0$-cycle and the $(w-1)$-cycle $\partial(1\,2\,\cdots\,w+1)$:
  these are trivial representations of $S_1$ and $S_{w+1}$, respectively.  The
  relations are that two singletons commute, and for each $(w+2)$--block, taking
  the signed sum of the boundaries of those facets that have a $(w+1)$--block in
  them gives zero.
\end{thm}

\begin{proof}
The fact that our proposed generators do generate comes automatically from the description of the basis.  To show that the relations are true, we know that $\del^2 = 0$, in particular when applied to a $(w+2)$--block, so the signed sum of the boundaries of \emph{all} facets of the $(w+2)$--block gives zero.  The facets that have no $(w+1)$--block in them are cells in $P(n, w)$, so their boundaries are null-homologous; thus, the sum of the boundaries of the remaining facets gives zero in homology.

To show that the specified relations are sufficient, we take every product of generators that is not in the basis and use the specified relations to write it in terms of the basis.  First we may assume that consecutive singletons are always in descending order, using the relation that commutes singletons.  Then, if a generator is not in the basis, some boundary of a $(w+1)$--block must be immediately preceded by a singleton that is less than every element of the $(w+1)$--block.  If we combine those elements to form a $(w+2)$--block, this substring of our generator is the boundary of one facet of the $(w+2)$--block, and the boundaries of all the other facets are in the basis.  Thus, applying the relation replaces a non-basis substring of our generator by a sum of basis substrings.  Applying the relations repeatedly from left to right rewrites our original non-basis generator in terms of the basis.
\end{proof}

\begin{thm}\label{thm-unwt-fid}
For each $j$ a multiple of $w-1$, the homology groups $H_j(\no_{w+1}(-))$ form a finitely generated $\FI_d$--module for $d = 1 + \frac{j}{w-1}$.
\end{thm}

\begin{proof}
The $\FI_d$--module structure is as follows.  Suppose we have an $\FI_d$--morphism from $[m]$ to $[n]$ with $n \geq m$.  For any element of $H_j(\no_{w+1}(m))$, we write it in terms of the generators from Theorem~\ref{thm-unwt-gen-rel}.  We apply the relabeling injection, and for each additional number with color $i \in \{0, 1, \ldots, \frac{j}{w-1}\}$, we insert an element of that color into each summand between the $i$th and the $(i+1)$st factors of the generator that look like the boundary of a $(w+1)$--block (rather than a singleton).  The result is an element of $H_j(\no_{w+1}(n))$.

To show that this map on homology is well-defined, we need to show that if we write an element of $H_j(\no_{w+1}(m))$ in terms of the generators in two different ways, the resulting elements of $H_j(\no_{w+1}(n))$ are the same.  To see this, suppose that we apply a relation to an original generator.  The same relation can just as easily be applied \emph{after} the relabeling and insertion, just by moving the new singletons past old singletons so that they are out of the way.  Thus, applying the relation does not change which homology class we get.

The fact that the maps respect composition of $\FI_d$--morphisms is automatic once we use the fact that consecutive singletons commute.  Thus we have an $\FI_d$--module.  It is finitely generated by the basis of $H_j(\no_{w+1}(j(w+1)))$.
\end{proof}

A proof along similar lines shows that the homology of weighted no--$(w+1)$--equal spaces can be written as a direct sum, with each summand equal to the span of the generators with a particular number of filters.  This allows us to insert weight-$1$ points between the filters in a well-defined way.  However, the result is not formally an $\FI_d$-module (or a direct sum of them) in a reasonable way, because relabeling cannot permute points of different weights without leaving the no--$(w+1)$--equal space.  Because the statement of the theorem would be cumbersome, we do not include the details.

\subsection{Disks in a strip of width $w=2$} Essentially the same proofs as for the no--$(w+1)$--equal spaces give generators and relations for the configuration spaces for $w=2$, which then give its homology an $\FI_d$--module structure.

\begin{thm}\label{thm-w2-gen-rel}
  The graded twisted algebra $H_*(\config(-,2))$ has a presentation with three
  generators and three relations.  The generators are:
  \begin{enumerate}
  \item $H_0(\config(1,2)) \cong \mathbb{Z}$, i.e.~a singleton $0$-cycle on which
    $S_1$ acts trivially.
  \item $H_1(\config(2,2)) \cong \mathbb{Z}$, i.e.~a 2-wheel on which $S_2$ acts
    trivially.  Write this as $w(1,2)=2\,1+1\,2$.
  \item The two-dimensional representation of $S_3$ spanned by the cycles
    $z(1,2,3)=\begin{tikzpicture}[baseline=-15pt, scale=.4, emp/.style={inner sep = 0pt, outer sep = 0pt}]
      \draw (-.25, 0)--(2.25, 0);
      \draw (-.25, -2)--(2.25, -2);

      \node (d3) at (1, -.55) {$\scriptstyle 2$};
      \node (d5) at (1+.5, -.55-.866) {$\scriptstyle 3$};
      \node (d2) at (1-.5, -.55-.866) {$\scriptstyle 1$};

      \draw (d3) circle (.5);
      \draw (d5) circle (.5);
      \draw (d2) circle (.5);

      \draw[->] (1.53033008588991, -0.596669914110089) arc (45:15:.75);
      \draw[->] (1.19411428382689, -1.85144436971680) arc (-75:-105:.75);
      \draw[->] (0.275555630283199, -0.932885716173110) arc (-195:-225:.75);
    \end{tikzpicture}$
    and
    $z(1,3,2)=\begin{tikzpicture}[baseline=7.75pt, scale=.4, emp/.style={inner sep = 0pt, outer sep = 0pt}]
      \draw (-.25, 0)--(2.25, 0);
      \draw (-.25, 2)--(2.25, 2);

      \node (d3) at (1, .55) {$\scriptstyle 2$};
      \node (d5) at (1+.5, .55+.866) {$\scriptstyle 3$};
      \node (d2) at (1-.5, .55+.866) {$\scriptstyle 1$};

      \draw (d3) circle (.5);
      \draw (d5) circle (.5);
      \draw (d2) circle (.5);

      \draw[<-] (1.53033008588991, 0.596669914110089) arc (-45:-15:.75);
      \draw[<-] (1.19411428382689, 1.85144436971680) arc (75:105:.75);
      \draw[<-] (0.275555630283199, 0.932885716173110) arc (195:225:.75);
  \end{tikzpicture}$
    in $H_1(\config(3,2))$, where transpositions in $S_3$ act by switching the
    two basis vectors.  This representation is irreducible over $\mathbb{Z}$, but
    over $\mathbb{Q}$ it splits into the direct sum of a trivial representation
    and a sign representation.
  \end{enumerate}
  The relations are:
  \begin{enumerate}
  \item The singletons commute: $a \mid b=b \mid a$.
  \item The relation induced by boundaries of 4-cells in $\cel(4)$.
  \item The relation in $H_1(\cel(3,2))$:
    \begin{align*}
      z(1,2,3) + z(1, 3, 2) &=1 \mid w(2,3)+w(2,3) \mid 1 \\
      &\qquad{}+2 \mid w(3,1)+w(3,1) \mid 2+3 \mid w(1,2)+w(1,2) \mid 3.
    \end{align*}
  \end{enumerate}
\end{thm}
Rationally, one can use the sign representation in $H_1(\config(3,2))$ as a
generator, removing the need for the third relation.

\begin{proof}
  We first establish a basis for $H_*(\config(n,2))$ consisting of concatenation
  products of (1), (2), and (3).  We use the same critical cells as in Remark
  \ref{rmk:crit-cells}, but interpret them differently: leader--follower pairs of
  the form $a \mid b\,c$ where $a<b<c$ become $z(a,b,c)$ factors, while the rest
  of the blocks individually yield singletons and $2$-wheels.

  To show that the generators and relations are valid, it suffices to show:
  \begin{enumerate}[(i)]
  \item Any concatenation product of generators can be written as a sum of basis
    elements by rewriting via the relations.
  \item Every element of the basis from Theorem \ref{thm:basis} can be written as
    a sum of concatenation products of the new generators.
  \end{enumerate}
  Together, the two steps imply that the new basis spans $H_*(\config(n,w))$;
  since it has the same cardinality as the old basis, this shows that it is
  indeed a basis.  The first step then implies that the provided relations are
  sufficient.

  To see (ii), we note that
  \begin{equation} \label{abc}
    \partial(a\,b\,c) + z(a, b, c) = w(b,c) \mid a + b \mid w(c,a) + w(a,b) \mid c.
  \end{equation}
  In this fashion we obtain filters of singletons.  We already know that every
  filter of a singleton and a $2$-wheel is given by an expression of the form
  \[a \mid w(b,c)-w(b,c) \mid a.\]
  Every element of the old basis is a concatenation product of these two types of
  cycles as well as singletons and $2$-wheels.

  To see (i), note first that in the description of the basis, the $2$-wheels
  separate the strip into intervals that do not interact with each other or with
  the $2$-wheels; that is, the requirements for being a basis element are the
  same as the requirements for each of these intervals individually to give a
  basis element.  Then to rewrite a product of generators in terms of basis
  elements, we only need to rewrite products of singletons and $3$-disk
  generators.

  Relation (3) allows us to eliminate $3$-disk generators that are in the
  ``wrong'' order.  Equation \eqref{abc} lets us turn relation (2) into a way to
  eliminate subwords of the form $a \mid z(b,c,d)$ where $a<b<c<d$.  Thus we can
  rewrite every product of singletons and $3$-disk generators in normal form
  using the same method as in Theorem~\ref{thm-unwt-gen-rel}.
\end{proof}

\begin{thm}\label{thm-w2-fid}
For each $j$, the homology groups $H_j(\config(-, 2))$ form a finitely generated $\FI_{j+1}$--module.
\end{thm}

\begin{proof}
Every generator of $H_j(\config(n, 2))$ has exactly $j$ factors of types (2) and (3).  We define the $\FI_{j+1}$--module structure as follows: to insert a new disk of color $i \in \{0, 1, \ldots, j\}$ into a generating cycle, we insert it as a singleton $0$--cycle between the $i$th and $(i+1)$st factors of type (2) or (3) of the generator.

As in Theorem~\ref{thm-unwt-fid}, the relations can be applied either before or after the insertion, so the $\FI_{j+1}$--module structure is well-defined.  The basis of $H_j(\config(3j, 1))$ gives a finite generating set for the $\FI_{j+1}$--module.
\end{proof}

\subsection{Disks in a strip of width $w > 2$}

We say that an element of $H_*(\config(n, w))$ is a \textbf{\textit{barrier}} if the two ways to concatenate it with a one-disk $0$--cycle represent different homology classes.  The wheels of size $w$ are barriers, as are any filters that contain a wheel of size $1$.  Because the homology $H_*(\config(n, w))$ is generated by concatenations of wheels and filters, we can count the number of barriers in each generator.

The following proposition implies that counting barriers is not well-defined on arbitrary homology classes in $H_*(\config(n, w))$.  As mentioned above, this suggests that the structure of $H_*(\config(n, w))$ is not well-described by $\FI_d$--modules.

\begin{prop}\label{prop-not-fid}
There is a nontrivial element of $H_4(\config(8, 3))$ that is simultaneously a sum of non-barrier generators, a sum of one-barrier generators, and a sum of two-barrier generators.
\end{prop}

As a warm-up before proving the proposition, we consider the case of $n=4$, $w = 3$, and $\sigma = 2\ 1\ 4\ 3$.  Then $P(n-\#\sigma, \mathcal{W}(\sigma), w)$ is a no--$(w+1)$--equal space with two points of weight $2$, which we denote $[2\ 1]$ and $[4\ 3]$.  Then because $i_{\sigma}$ is a chain map we have
\[i_\sigma(\del([2\ 1][4\ 3])) = \del(i_{\sigma}([2\ 1][4\ 3])).\]
The left-hand side is the sum (with some signs) of the two ways to concatenate the wheels $2\ 1$ and $4\ 3$, whereas the right-hand side is the sum (with some signs) of the four cycles $\del(2\ 1\ 4\ 3)$, $\del(2\ 1\ 3\ 4)$, $\del(1\ 2\ 4\ 3)$, and $\del(1\ 2\ 3\ 4)$, each of which can be thought of as an image under the $S_4$--relabeling of $\del(1\ 2\ 3\ 4)$, which for width $w = 3$ is a filter with four wheels each containing one disk, and is a barrier.

\begin{proof}[Proof of Proposition~\ref{prop-not-fid}]
Let $\sigma = 2\ 1\ 4\ 3\ 6\ 5\ 8\ 7$.  Then $P(n-\#\sigma, \mathcal{W}(\sigma), w)$ is a no--$(w+1)$--equal space with four points of weight $2$, which we denote by $[2\ 1]$, $[4\ 3]$, $[6\ 5]$, and $[8\ 7]$.  We can apply any element $\tau \in S_4$ to any chain in $C_*(P(n-\#\sigma, \mathcal{W}(\sigma), w))$ by permuting these four labels of points.

To construct our cycle, we take
\[z = \sum_{\tau \in S_4} \sgn(\tau) \cdot i_{\sigma}(\tau([2\ 1]\vert [4\ 3]\vert [6\ 5]\vert [8\ 7])).\]
That is, we take the signed sum of all of the ways to concatenate the four wheels $2\ 1$, $4\ 3$, $6\ 5$, and $8\ 7$.  Each of these summands is already a basis element, although they differ as to which pairs are considered filters of two wheels.  Thus, their sum is nonzero in homology, and none of these wheels is a barrier so it is a sum of non-barrier generators.

To  write our element $z$ as a sum of one-barrier cycles, we pair up all elements $\tau$ in $S_4$ that differ by swapping the middle two wheels.  One such pair gives
\[i_{\sigma}([2\ 1] \vert \del([4\ 3][6\ 5])\vert [8\ 7]),\]
which concatenates the wheels $2\ 1$ (on the left) and $8\ 7$ (on the right) to the cycle
\[i_\sigma(\del([4\ 3][6\ 5])) = \del(i_{\sigma}([4\ 3][6\ 5])),\]
which is a sum (with some signs) of relabelings of the barrier filter $\del(3\ 4\ 5\ 6)$.  In this way we can write our element $z$ as a sum of generators, each one a concatenation of a $2$--wheel, a barrier filter, and another $2$--wheel.

To write our element $z$ as a sum of two-barrier cycles, we group the elements $\tau$ in $S_4$ into quadruples: permutations get grouped together if they differ by swapping the first two and/or the last two wheels.  One such quadruple gives
\[i_{\sigma}(\del([2\ 1][4\ 3])\vert \del([6\ 5][8\ 7])) = \pm \del(i_{\sigma}([2\ 1][4\ 3])) \vert \del(i_{\sigma}([6\ 5][8\ 7])),\]
which, similar to the computation above, is a sum of generators, each one a concatenation of two barrier filters.
\end{proof}

\section{Configuration spaces of unordered disks}\label{sec:unordered}

The configuration space of $n$ unordered disks of diameter $1$ in a strip of width $w$ is the quotient of $\config(n, w)$ by the action of $S_n$ that permutes the disk labels, and it is homotopy equivalent to the quotient of $\cel(n, w)$ by the action of $S_n$.  Because the action is cellular and free, this quotient is a cell complex which we call $\ucel(n, w)$.  In this section of the paper, we compute the homology of $\ucel(n, w)$ (and thus of the configuration space of unordered disks in a strip) with field coefficients, using the discrete Morse theory methods from Section~\ref{S:nokequal}.  In the version of discrete Morse theory that applies here, we do not need to assume that the coefficients of the boundary map are $\pm 1$, as is true for polyhedral cell complexes, but only that these coefficients are units.  We use field coefficients throughout this section, so all nonzero coefficients are automatically units.

The concatenation product
\[H_j(\ucel(n,w)) \otimes H_{j'}(\ucel(n',w)) \to H_{j+j'}(\ucel(n+n',w))\]
is well-defined since there is no need to choose labels for the disks.
Therefore, for any ring $R$, $H_*(\ucel(*,w);R)$ forms a noncommutative bigraded
$R$-algebra.

The cells of $\ucel(n, w)$ are labeled by symbols as in $\cel(n, w)$, but the
numbers in the symbols become indistinguishable; the only remaining information
is the sizes of blocks, which we also refer to as weights.  We notate them all by ${\circ}$, 
and use exponents to denote the weights of the blocks.  For instance, in $\cel(3)$
we have
\[\del(1\ 2\ 3) = -1 \mid 2\ 3\ + 2 \mid 1\ 3\ - 3 \mid 1\ 2 + 2\ 3 \mid 1 - 1\ 3 \mid 2 + 1\ 2 \mid 3,\]
and the corresponding relation in $\ucel(3)$ is
\[\del({\circ}^3) = -{\circ} \mid {\circ\circ} + {\circ} \mid {\circ\circ} - {\circ} \mid {\circ\circ} + {\circ\circ} \mid {\circ} - {\circ\circ} \mid {\circ} + {\circ\circ} \mid {\circ} = -{\circ} \mid {\circ\circ} + {\circ\circ} \mid {\circ}.\]
The following lemma describes all the coefficients of these boundary maps.

\begin{lem}
In $\ucel(n)$, the coefficient of the face ${\circ}^k  \mid {\circ}^{n-k}$ in the boundary of the cell ${\circ}^n$ can be described as follows:
\begin{itemize}
\item If $k$ and $n-k$ are both odd, the coefficient is $0$.
\item If $n = 2n'$ is even and $k = 2k'$ is even, the coefficient is $\binom{n'}{k'}$.
\item If $n = 2n' + 1$ is odd and $k = 2k'$ is even, the coefficient is $\binom{n'}{k'}$.
\item If $n = 2n' + 1$ is odd and $k = 2k' + 1$ is odd, the coefficient is $-\binom{n'}{k'}$.
\end{itemize}
\end{lem}

\begin{proof}
Consider the cell $1\ 2\ \cdots\ n$ in $\cel(n)$.  In its boundary, there are $\binom{n}{k}$ cells that project to ${\circ}^k \mid {\circ}^{n-k}$ in $\ucel(n)$, and our task is to add up all of their signs.  The sign of each such face is $(-1)^k$ times the sign of the corresponding permutation.

We pair up the numbers $1$ and $2$, $3$ and $4$, and so on, pairing up $n-1$ and $n$ if $n$ is even, or leaving only $n$ unpaired if $n$ is odd.  We can match and cancel faces of $1\ 2\ \cdots\ n$ with opposite signs in the following way.  Given a face, if $1$ and $2$ are in different blocks, then swapping them gives another face with opposite sign.  Similarly, if $1$ and $2$ are in the same block, but $3$ and $4$ are in different blocks, then swapping $3$ and $4$ gives another face with opposite sign.  In this way, for each face for which a pair of numbers is split up, we match and cancel it with another such face by finding the first pair of numbers that is split up, and swapping those numbers.

The remaining faces have $1$ and $2$ in the same block, $3$ and $4$ in the same block, and so on, and each one corresponds to an even permutation, so the total sign is $(-1)^k$.  If $k$ and $n-k$ are both odd, then there are no such faces.  If $k = 2k'$ is even and there are $n'$ pairs, then the faces all have positive sign, and there are $\binom{n'}{k'}$ of them.  And, if $k = 2k' + 1$ is odd and there are $n'$ pairs, then the faces all have negative sign, and there are $\binom{n'}{k'}$ of them.
\end{proof}

\subsection{Discrete Morse theory on $\ucel(n, w)$ with $\mathbb{Q}$--coefficients}

When ordering the cells to produce a discrete gradient vector field, part of the ordering will be chosen later to be lexicographical.  Thus, we need to compute the lexicographically least way to split each block.

\begin{lem}
The lexicographically least face of the cell ${\circ}^n$ in $\ucel(n)$ with nonzero coefficient is given as follows.  If $n$ is odd, the least face is ${\circ}^1 \mid {\circ}^{n-1}$.  If $n$ is even and greater than $2$, the least face is ${\circ}^2 \mid {\circ}^{n-2}$.  In the case $n = 2$, the cell ${\circ}^2$ is a cycle.
\end{lem}

\begin{proof}
If $n$ is odd, then the coefficient of the face ${\circ}^1 \mid {\circ}^{n-1}$ is $-\binom{(n-1)/2}{0} = -1$.  If $n > 2$ is even, then the coefficient of the face ${\circ}^1 \mid {\circ}^{n-1}$ is zero, but the coefficient of the face ${\circ}^2 \mid {\circ}^{n-2}$ is $\binom{(n-2)/2}{1} = (n-2)/2 \neq 0$.
\end{proof}

\begin{defn}
Two consecutive blocks in a symbol in $\ucel(n, w)$ form a \textbf{\textit{leader--follower pair in characteristic $0$}} if they have the form ${\circ}^1 \mid {\circ}^{2k'}$ or ${\circ}^{2} \mid {\circ}^{2k'}$ for some $k'$. 
\end{defn}

\begin{thm}\label{thm:ubasis-q}
There is a discrete gradient vector field on $\ucel(n, w)$ such that the critical cells are in bijection with a basis for $H_*(\ucel(n, w); \mathbb{Q})$.  The critical cells are described as follows:
\begin{itemize}
\item For $w = 2$, every cell.
\item For $w = 3$, the concatenation of zero or more blocks ${\circ}^{2}$, followed by zero or more singletons ${\circ}^1$.
\item For $w > 2$ even, the concatenation of zero or more copies of ${\circ}^2 \mid {\circ}^{w}$ or strings ending in ${\circ}^1 \mid {\circ}^w$, followed by the concatenation of zero or more singletons ${\circ}^1$.  For each string ending in ${\circ}^1 \mid {\circ}^w$, it consists of an optional ${\circ}^2$, followed by zero or more singletons ${\circ}^1$, followed by the pair ${\circ}^1 \mid {\circ}^w$.
\item For $w > 3$ odd, the concatenation of zero or more pairs ${\circ}^2 \mid {\circ}^{w-1}$, followed by an optional ${\circ}^2$, followed by zero or more singletons ${\circ}^1$.
\end{itemize}
\end{thm}

\begin{proof}
The proof is similar to the proof of Theorem~\ref{thm-weighted-basis}, which finds a basis for homology of weighted no-$(w+1)$-equal spaces.  Informally, we construct a discrete vector field that pairs cells as follows: given a symbol, we read it from left to right, and find the first place where either there is a block of weight greater than $2$, in which case we match down by breaking it into a leader--follower pair; or, there is a leader-follower pair of total weight at most $w$, in which case we match up by merging it into one block.  

The cells that remain unpaired are those for which every leader--follower pair has total weight greater than $w$, and all other blocks are either ${\circ}^1$ or ${\circ}^2$.  We observe that if $w$ is even, the only possibilities for leader--follower pairs in critical cells are ${\circ}^1 \mid {\circ}^w$ and ${\circ}^2 \mid {\circ}^w$, because the total weight must exceed $w$ while the weight of the follower block must be even and at most $w$. For the same reason, if $w$ is odd the only possible leader--follower pair in a critical cell is ${\circ}^2 \mid {\circ}^{w-1}$.  We also observe that there are no instances of either ${\circ}^1 \mid {\circ}^2$ for $w \geq 3$, or ${\circ}^2 \mid {\circ}^2$ for $w \geq 4$, because these would be leader--follower pairs.  Combining these observations, we deduce that the critical cells must have the form given in the theorem statement.

More formally, to check that this discrete vector field is gradient, we exhibit an ordering that produces it.  We order the cells of $\ucel(n, w)$ such that if $f$ and $g$ are two cells, we find the first block where they differ, and order according to this first differing block in $f$ and $g$:
\begin{itemize}
\item a follower block is less than a non-follower;
\item two follower blocks are ordered in increasing order of weight; and
\item two non-follower blocks are ordered in decreasing order of weight.
\end{itemize}
Then, we pair up two cells $f$ and $g$, with $f$ a face of $g$, if $f$ is the greatest face of $g$ and $g$ is the least coface of $f$.  One can check that this pairing agrees with the pairing described informally above, and thus that the critical cells match the desired description.

To check that the resulting critical cells correspond to a basis, by Lemma~\ref{lem-max-basis} it suffices to find a cycle $z(e)$ for each critical cell $e$, such that $e$ is the greatest cell in $z(e)$.  Because we are using coefficients in $\mathbb{Q}$, we do not have to worry about whether the coefficient is a unit, as long as it is nonzero.  To construct $z(e)$, we take the concatenation product of the following cycles: for each block that is not in a leader--follower pair, it is already a cycle, and for each leader--follower pair, as our cycle we take the boundary of the block resulting from merging the pair.  Note that every cell of this boundary that has a nonzero coefficient has block weight at most $w$, because the original leader-follower pair is one of the two faces with the most unbalanced block weights in this boundary.  Because our leader--follower pair is the lexicographically least face of the merged block, and our ordering is the reverse of lexicographical for non-followers, we see that $e$ is the greatest cell in $z(e)$.  This implies that the cells $z(e)$ form a basis for $H_*(\ucel(n, w); \mathbb{Q})$.
\end{proof}

\begin{cor}
The homology $H_*(\ucel(*, w); \mathbb{Q})$ forms a bigraded algebra over $\mathbb{Q}$ under concatenation product.  It has the following generators:
\begin{enumerate}
\item The singleton block ${\circ}^1$;
\item The block ${\circ}^{2}$;
\item For $w > 3$, the cycle $\partial({\circ}^{w+1})$; and
\item For $w > 2$ even, the cycle $\partial({\circ}^{w+2})$.
\end{enumerate}
It has the following relations:
\begin{enumerate}
\item The singleton block ${\circ}^1$ commutes with ${\circ}^{2}$ for all $w \geq 3$;
\item The singleton block ${\circ}^1$ commutes with $\partial({\circ}^{w+1})$ for all odd $w > 3$; and
\item The symbol ${\circ}^{2}  \mid {\circ}^{2}$ is null-homologous for all $w \geq 4$;
\end{enumerate}
\end{cor}

\begin{proof}
The description of the basis shows that the specified generators do generate.

Relation (1) is true because the boundary of the cell ${\circ}^{3}$ is $-{\circ}^1 \mid {\circ}^{2} + {\circ}^{2} \mid {\circ}^1$.  Relation (2) comes from the relation $\del^2({\circ}^{w+2}) = 0$ on $\ucel(n)$; when $w$ is odd, expanding $\del({\circ}^{w+2})$ gives $-{\circ}^1 \mid {\circ}^{w+1} + {\circ}^{w+1} \mid {\circ}^1$ plus a sum of cells in $\ucel(n, w)$, so applying $\del$ again gives our desired homology relation.  Relation (3) is true because ${\circ}^{2} \mid {\circ}^{2}$ is the only face of ${\circ}^{4}$ with nonzero coefficient.

The relations are enough to transform an arbitrary product of generators into one of our basis cycles.
\end{proof}

\begin{cor}\label{cor:q-growth}
For fixed $j$ and $w$, the Betti numbers $\beta_j(\ucel(n, w); \mathbb{Q})$ as a function of $n$ grow with an upper bound of $O(n^q)$, where $q = \left\lfloor\frac{j}{w-1}\right\rfloor$.  If $w$ is odd, the Betti numbers either are $0$ for all $n$ or are $1$ for all sufficiently large $n$.  If $w$ is even, the Betti numbers either are $0$ for all $n$ or grow as $\Theta(n^q)$, and the latter case holds for all $j \geq (w-1)(w-3)$.
\end{cor}

\begin{proof}
Deleting non-leader singleton blocks ${\circ}^1$ from a critical cell gives a critical cell with smaller $n$, and for each $j$ and $w$, there are finitely many ways to form one of these ``skyline'' critical cells that have no non-leader singletons.  For each skyline critical cell, the only places to insert singletons are at the end (that is, on the right side) and immediately before the leader--follower pair ${\circ}^1 \mid {\circ}^w$, which exists only if $w$ is even.  Thus, if $w$ is odd, then for all sufficiently large $n$ the Betti number $\beta_j(\ucel(n, w); \mathbb{Q})$ is constant, equal to the number of skyline critical cells for $j$ and $w$, which is $1$ if $j$ is congruent to $0$ or $1$ mod $w-1$, or $0$ otherwise.

If $w$ is even, for each $j$ we claim that either there is a skyline with $\left\lfloor\frac{j}{w-1}\right\rfloor$ instances of ${\circ}^1 \mid {\circ}^w$, or there is no skyline at all.  We write $j$ as $q(w-1) + r$, with $0 \leq r \leq w-2$.  To construct the critical cell, if $r \leq q$ we concatenate $r$ instances of ${\circ}^2 \mid {\circ}^1 \mid {\circ}^w$ and then $q-r$ instances of ${\circ}^1 \mid {\circ}^w$.  If $r = q+1$ we concatenate $q$ instances of ${\circ}^2 \mid {\circ}^1 \mid {\circ}^w$, followed by one instance of ${\circ}^2$.  If $r > q+1$ there is no way to build a critical cell of dimension $j$, because all blocks that contribute to the dimension are either ${\circ}^w$, contributing $w-1$, or ${\circ}^2$, contributing $1$, and there can be at most $q+1$ instances of ${\circ}^2$.

Given a skyline critical cell with $n'$ disks and $k$ instances of ${\circ}^1 \mid {\circ}^w$, the number of critical cells with $n$ disks arising from this skyline is $\binom{n-n'+k}{k}$, corresponding to the number of ways to arrange $n-n'$ additional singletons and $k$ dividers.  This is a polynomial in $n$ of degree $k$.  If $w$ is even, for each $j$ either there is no skyline, or there is a skyline with $q = \left\lfloor \frac{j}{w-1}\right\rfloor$ instances of ${\circ}^1 \mid {\circ}^w$, which is the largest possible $k$.  Thus, the Betti numbers grow like either $0$ or $\Theta(n^q)$.

In the case where $w$ is even and $j \geq (w-1)(w-3)$, the quotient $q$ is at least $w-3$ and the remainder $r$ is at most $w-2$, so the case $r > q+1$ is impossible, and the Betti numbers grow like $\Theta(n^q)$.
\end{proof}

\subsection{Discrete Morse theory on $\ucel(n, w)$ with $\mathbb{F}_p$--coefficients}

Using coefficients mod $p$, our strategy for computing the homology is the same as with $\mathbb{Q}$ coefficients, but the answer becomes more complicated because we need to account for divisibility of binomial coefficients.

\begin{lem}
For any prime $p$, the lexicographically least face of the cell ${\circ}^n$ in $\ucel(n)$ with coefficient not divisible by $p$ is given as follows.  If $n$ is odd, the least face is ${\circ}^1 \mid {\circ}^{n-1}$.  If $n$ is even, then we write $n$ as $2p^k \cdot a$, where $a$ is not divisible by $p$, and the least face is ${\circ}^{2p^k} \mid {\circ}^{2p^k(a-1)}$.
\end{lem}

\begin{proof}
If $n$ is odd, then the coefficient of the face ${\circ}^1 \mid {\circ}^{n-1}$ is $-\binom{(n-1)/2}{0} = -1$, which is a unit in any $\mathbb{F}_p$.

If $n = 2p^k\cdot a$ is even, then the faces of ${\circ}^n$ have coefficients $\binom{p^ka}{k'}$ for various $k'$.  These are the coefficients of $(x+y)^{p^ka} \equiv (x^{p^k} + y^{p^k})^a$ mod $p$, using the Frobenius homomorphism.  These coefficients are $0$ unless $p^k$ divides $k'$, and the coefficient $\binom{p^ka}{p^k}$ is congruent mod $p$ to $\binom{a}{1} = a$, which is a unit in $\mathbb{F}_p$.
\end{proof}

\begin{defn}
Two consecutive blocks in a symbol in $\ucel(n, w)$ form a \textbf{\textit{leader--follower pair in characteristic $p$}} if they have the form ${\circ}^1  \mid {\circ}^{2k'}$ for some $k'$, or the form ${\circ}^{2p^k} \mid {\circ}^{2p^k(a-1)}$ for some $k \geq 0$ and $a$ not divisible by $p$.  We assign the pairs disjointly from left to right, so that once a block is a follower in a pair with the previous block, it cannot also be a leader in a pair with the next block.
\end{defn}

\begin{thm}
For each prime $p$, there is a discrete gradient vector field on $\ucel(n, w)$ with $\mathbb{F}_p$ coefficients, such that the critical cells are in bijection with a basis for $H_*(\ucel(n, w); \mathbb{F}_p)$.  The critical cells are those with the properties that every leader--follower pair has total weight greater than $w$, and every block that is not a follower is either ${\circ}^1$ or has the form ${\circ}^{2p^k}$ for some $k\geq 0$.  These properties imply that consecutive blocks that are not followers appear in decreasing order of weight, weakly decreasing if $p = 2$ and strictly decreasing if $p \neq 2$, except for singleton blocks ${\circ}^1$ which may occur consecutively for any $p$.
\end{thm}

\begin{proof}
The proof is exactly analogous to that of Theorem~\ref{thm:ubasis-q}, which addresses the case of $\mathbb{Q}$ coefficients.  As in that proof, our discrete vector field is informally described by reading each symbol from left to right, breaking down any block of weight other than $1$ or $2p^k$ into a leader--follower pair, and combining any leader--follower pair of total weight at most $w$.

We can describe the resulting critical cells more concretely as follows.  To find all possibilities for leader--follower pairs in critical cells, for each $k \geq 0$ such that $2p^k \leq w$, we find the least multiple of $2p^k$ greater than $w$.  If this multiple is not divisible by $2p^{k+1}$, it has the form $2p^ka$ for $a$ not divisible by $p$, and the leader--follower pair ${\circ}^{2p^k} \mid {\circ}^{2p^k(a-1)}$ may appear in a critical cell.  We know that $2p^k(a-1)$ is at most $w$, otherwise it would contradict the selection of $2p^ka$ as the least multiple of $2p^k$ greater than $w$.  If the least multiple of $2p^k$ greater than $w$ is divisible by $2p^{k+1}$, there is no leader--follower pair beginning with $2p^k$ that may appear in a critical cell.  In addition, if $w$ is even, the leader--follower pair ${\circ}^1 \mid {\circ}^w$ may appear in a critical cell.

For each critical cell, we can imagine dividing the symbol into strings, where the followers are the dividers.  Each string consists of blocks ${\circ}^1$ and/or ${\circ}^{2p^k}$ for various $k \geq 0$, with constraints on the multiplicities and order because they may not form leader--follower pairs.  The pair of blocks ${\circ}^{2p^k} \mid {\circ}^{2p^k}$ forms a leader--follower pair if $p \neq 2$, so it may not appear in one of these strings; if $p = 2$, it does not form a leader--follower pair, so it may appear.  The pairs ${\circ}^1 \mid {\circ}^{2p^k}$ for $k \geq 0$ and ${\circ}^{2p^k} \mid {\circ}^{2p^{\ell}}$ for $k < \ell$ do form leader--follower pairs regardless of $p$, so they may not appear in one of these strings.  Thus, if $p = 2$, each string is an arbitrary sequence of blocks ${\circ}^1$ and ${\circ}^{2p^k}$, any number of each, in weakly decreasing order.  If $p \neq 2$, each string is an arbitrary sequence of any number of blocks ${\circ}^1$ and at most one of each block ${\circ}^{2p^k}$, in decreasing order.  Note that given a leader--follower pair, in the string preceding it, the blocks cannot have smaller weight than the leader, because of the condition that the entire string including the leader should be in decreasing order.

To complete the proof formally, we use the same ordering on cells of $\ucel(n, w)$ as in the proof of Theorem~\ref{thm:ubasis-q}, except with the characteristic $p$ definition of leader--follower pairs.  The resulting discrete gradient vector field agrees with the informal description above.  We define cycles $z(e)$ as in the proof of Theorem~\ref{thm:ubasis-q}: for each leader--follower pair, we take the boundary of the block in $\ucel(n)$ resulting from merging the pair, for each block not in a leader--follower pair it is already a cycle, and we take the concatenation product of all these cycles to get $z(e)$.  Applying Lemma~\ref{lem-max-basis}, we conclude that the cells $z(e)$ form a basis for $H_*(\ucel(n, w); \mathbb{F}_p)$.
\end{proof}

\begin{cor}
For each prime $p$, the homology $H_*(\ucel(*, w); \mathbb{F}_p)$ forms a bigraded algebra over $\mathbb{F}_p$ under concatenation product.  It has the following generators:
\begin{enumerate}
\item The singleton block ${\circ}^1$;
\item The block ${\circ}^{2p^k}$, for each $k \geq 0$ with $2p^k \leq w$;
\item If $w$ is even but not equal to $2p^k$ for any $k \geq 0$, the cycle $\partial({\circ}^{w+1})$; and
\item The cycle $\partial({\circ}^{n'})$, where $n'$ is the least multiple of $2p^k$ greater than $w$, for each $k \geq 0$ such that $2p^k \leq w/2$.  In this case we let $k(n')$ denote the power of $p$ in the prime factorization of $n'$.
\end{enumerate}
It has the following relations:
\begin{enumerate}
\item The singleton block ${\circ}^1$ commutes with ${\circ}^{2p^k}$ whenever $1 + 2p^k \leq w$;
\item The singleton block ${\circ}^1$ commutes with $\partial({\circ}^{n'})$ whenever $1 + n' - 2p^{k(n')} \leq w$;
\item If $p \neq 2$, then ${\circ}^{2p^k}  \mid {\circ}^{2p^k}$ is null-homologous whenever $4p^k \leq w$;
\item We have ${\circ}^{2p^\ell}  \mid {\circ}^{2p^k} = -({\circ}^{2p^k}  \mid {\circ}^{2p^\ell})$ whenever $2p^l + 2p^k \leq w$ and $k \neq \ell$; and
\item The block ${\circ}^{2p^\ell}$ commutes with $\partial ({\circ}^{n'})$ whenever $2p^\ell + n' - 2p^{k(n')} \leq w$.
\end{enumerate}
\end{cor}

\begin{proof}
The description of the basis shows that the specified generators do generate.

Relation (1) is true because the boundary of the cell ${\circ}^{1+2p^k}$ mod $p$ is $-{\circ}^1 \mid {\circ}^{2p^k} + {\circ}^{2p^k} \mid {\circ}^1$.  All other faces have coefficients of $0$, because $\binom{p^k}{k'} \equiv 0$ mod $p$ unless $k'$ is $0$ or $p^k$.  Relation (2) comes from the relation $\del^2({\circ}^{1+n'}) = 0$ on $\ucel(n)$; expanding $\del({\circ}^{1+n'})$ gives $-{\circ}^1 \mid {\circ}^{n'} + {\circ}^{n'} \mid {\circ}^1$ plus a sum of cells in $\ucel(n, w)$, so applying $\del$ again gives our desired homology relation.  Relation (3) is true because ${\circ}^{2p^k} \mid {\circ}^{2p^k}$ is the only face of ${\circ}^{4p^k}$ with nonzero coefficient mod $p$ when $p \neq 2$. (If $p =2$, then ${\circ}^{4p^k}$ is a cycle.)  To see that relation (4) is true, if $\ell < k$ the faces of ${\circ}^{2p^\ell + 2p^k}$ have coefficients $\binom{p^\ell(1 + p^{k - \ell})}{p^\ell k'} \equiv \binom{1+p^{k - \ell}}{k'}$ for various $k'$.  The coefficients for $k' = 1$ and $k' = p^{k - \ell}$ are both $1$ mod $p$ and are the coefficients of ${\circ}^{2p^\ell} \mid {\circ}^{2p^k}$ and ${\circ}^{2p^k} \mid {\circ}^{2p^\ell}$.  The other coefficients are all $0$ mod $p$, because the Pascal's triangle identity gives $\binom{1+p^{k - \ell}}{k'} = \binom{p^{k-\ell}}{k' - 1} + \binom{p^{k-\ell}}{k'}$, which is $0$ mod $p$ unless $k'$ is $0$, $1$, $p^{k-\ell} - 1$, or $p^{k-\ell}$.  Relation (5) comes from the relation $\del^2({\circ}^{2p^\ell + n'}) = 0$; using similar reasoning to relation (4), we find that the only faces of $\del({\circ}^{2p^\ell + n'})$ with nonzero coefficient have the form ${\circ}^{k'} \mid {\circ}^{2p^\ell + n'-k'}$ where $k'$ is congruent to either $0$ or $2p^\ell$ mod $2p^k$.  Thus, $\del({\circ}^{2p^\ell + n'})$ is ${\circ}^{2p^\ell} \mid {\circ}^{n'} + {\circ}^{n'} \mid {\circ}^{2p^\ell}$ plus a sum of cells in $\ucel(n, w)$, and then we may apply $\del$ again, keeping in mind that the sign convention for the Leibniz rule gives a negative sign to the term ${\circ}^{2p^\ell} \mid \del({\circ}^{n'})$.

The relations are enough to transform an arbitrary product of generators so that consecutive non-follower blocks are in decreasing order, strictly decreasing if $p \neq 2$.  The resulting cycle is in our basis.
\end{proof}

\begin{cor}
For fixed $j$, $w$, and prime $p$, the Betti numbers $\beta_j(\ucel(n, w); \mathbb{F}_p)$ as a function of $n$ grow with an upper bound of $O(n^q)$, where $q = \left\lfloor\frac{j}{w-1}\right\rfloor$.  If $w$ is odd, the Betti numbers become eventually constant in $n$; if $w$ is even and $j \geq (w-1)(w-3)$, the Betti numbers grow as $\Theta(n^q)$.
\end{cor}

\begin{proof}
As in Corollary~\ref{cor:q-growth}, we can delete non-leader singleton blocks ${\circ}^1$ from each critical cell to form one of finitely many ``skylines''.  To recover all critical cells with a given skyline, we insert singletons ${\circ}^1$ either on the far right, or immediately preceding a leader--follower pair ${\circ}^1 \mid {\circ}^w$, which can only exist if $w$ is even.  In this case, the maximum possible number of such pairs is $q = \left\lfloor \frac{j}{w-1}\right\rfloor$, so summing over all skylines, we find that the total Betti number is eventually equal to the number of skylines if $w$ is odd, and is bounded above by $O(n^q)$ if $w$ is even.  

In the case where $j \geq (w-1)(w-3)$, we can construct a skyline critical cell with $q$ instances of ${\circ}^1 \mid {\circ}^w$ in exactly the same way as in Corollary~\ref{cor:q-growth}. The existence of such a skyline implies that $\beta_j(\ucel(n, w); \mathbb{F}_p)$ grows with a lower bound of $\Omega(n^q)$, matching the upper bound.
\end{proof}

\section{Open questions and further directions}\label{sec:open}

\begin{enumerate}
\item One generalization of the disks in a strip configuration spaces is the following.  Let $E \rightarrow B$ be a locally trivial bundle, and consider the configuration space of $n$ distinct points in $E$, such that each fiber of the bundle may contain at most $w$ of them.  What do our methods say about the homology of this configuration space?  We predict that if some neighborhoods in $B$ are $1$-dimensional, then the homology exhibits the same non-commutativity as that of disks in a strip, but that if $B$ is everywhere at least $2$-dimensional, then the homology of the configuration space is a finitely generated $\FI$-module.
\item The representation stability properties of the configuration space of $n$ points on a given manifold come in some sense from the special case where the manifold is Euclidean space; the special case can be considered a local model.  In particular, when the manifold has an end, the homology of the manifold configuration space, considered for all $n$ at once, is a module over the twisted commutative algebra given by the homology of the Euclidean configuration space.  The algebra acts by inserting cycles near infinity in the end of the manifold.  Does our disks in a strip configuration space act as a local model for other configuration spaces, for which the homology exhibits similar finite generation properties due to its being a module?  For instance, what can we say about the homology of the configuration space of disks in the product of an interval with a non-compact $1$-complex, such as the union of three rays with a common starting point?
\item Our proof that $H_*(\config(n, w))$ is finitely generated as a twisted non-commutative algebra relies on fully computing $H_*(\config(n, w))$ and exhibiting the generators.  Is there a more abstract algebraic framework that would prove finite generation without computing the homology?
\item Having described $H_*(\config(n, w))$, but not at all equivariantly, we can ask about its $S_n$-action by permuting the disk labels.  In its decomposition into irreducible representations of $S_n$, how do the multiplicities grow in $n$?  For finitely generated twisted noncommutative algebras in general, by what patterns can the multiplicities grow?  In particular, can one use presentations by
  generators and relations as in \S\ref{S:FId} to recover this information?
\end{enumerate}

\appendix
\section{Computer calculations for small $n$}

In \cite{AKM}, the Betti numbers of $\config(n,w)$ were computed for $n \leq 8$
using off-the-shelf software for computing persistent homology.  This involved
running the software on the complex $\cel(8)$, which has over 5 million cells.
In general, $\cel(n)$ has $2^{n-1}n!$ cells.

We wrote a Python script that harnesses Theorems \ref{thm:basis} and \ref{thm:PH}
to compute the persistence diagram of $H_*(\cel(n,*))$.  Although the runtime
still grows as $n!$, this is faster than the above method by an exponential
factor; for $n=12$, the script ran on a laptop in less than 90 minutes.
Figures \ref{pdiag1}, \ref{pdiag2}, and \ref{pdiag3} show a graphical
representation of the resulting persistence diagram.  We would like to thank
Matthew Kahle for making the first version of this series of figures.


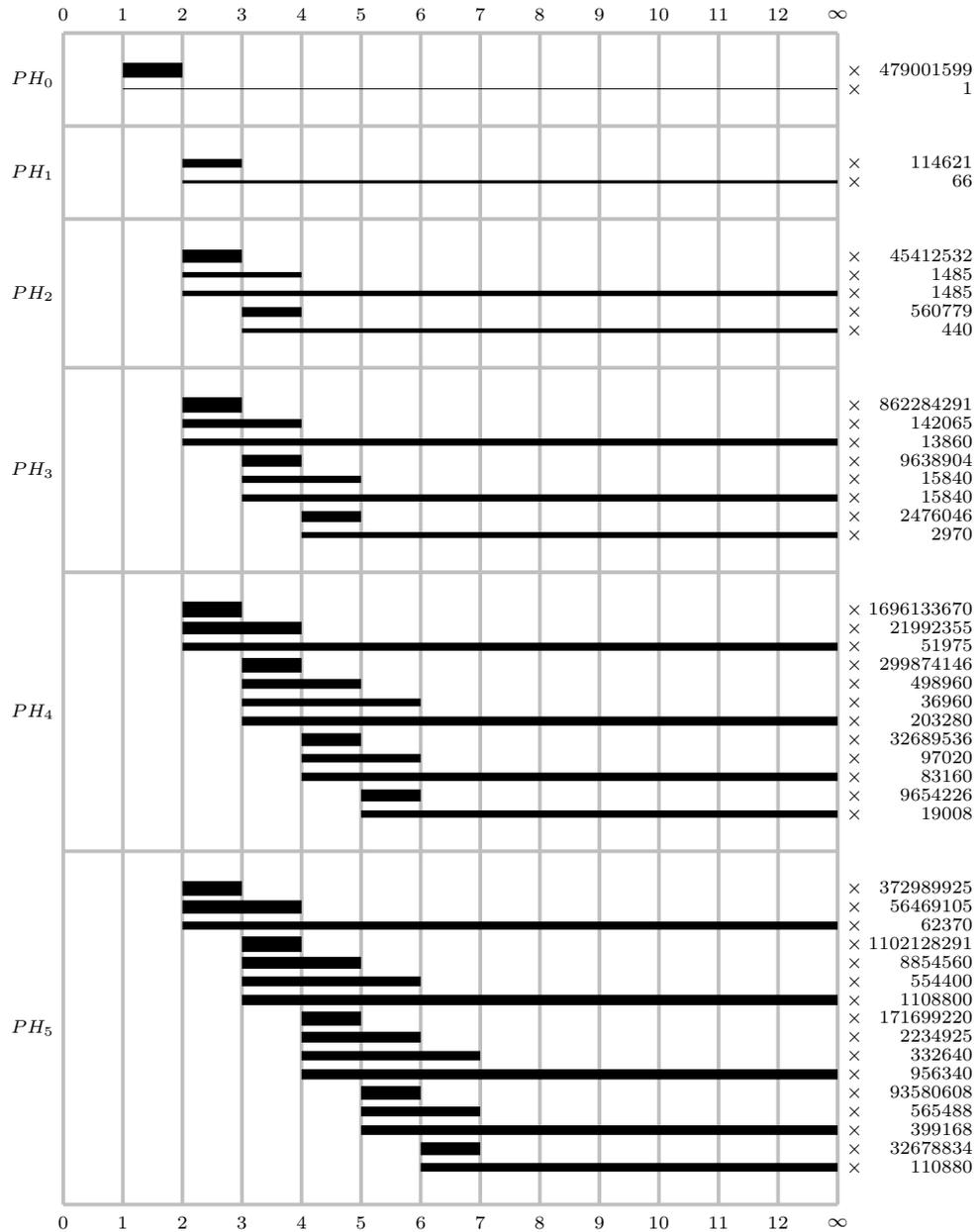
\begin{figure}
\begin{tikzpicture}[line width=0.5mm,xscale=0.8,yscale=1]
\tiny

\foreach \x in {0,...,13}
{
\draw[lightgray] (\x,0.5)--(\x,-15.25);
}

\foreach \x in {0,...,12}
{
\node at (\x,0.75){\x};
}

\node at (13,0.75){$\infty$};

\draw[lightgray] (0, 0.5000 )-- (13, 0.5000 );
\draw[line width= 1.9987214493574 mm] ( 1 , 0.0000 )-- ( 2 , 0.0000 );
\node[anchor=west,text width=\tendigitwidth] at (13, 0.0000 ) {$\times$\hspace{\fill}479001599};
\draw[line width= 0.1000 mm] ( 1 , -0.2500 )-- ( 13 , -0.2500 );
\node[anchor=west,text width=\tendigitwidth] at (13, -0.2500 ) {$\times$\hspace{\fill}1 };
\node[anchor=east] at (0, -0.125 ) {$PH_0$};

\draw[lightgray] (0, -0.7500 )-- (13, -0.7500 );
\draw[line width= 1.1649386312549 mm] ( 2 , -1.250 )-- ( 3 , -1.250 );
\node[anchor=west,text width=\tendigitwidth] at (13, -1.250 ) {$\times$\hspace{\fill}114621 };
\draw[line width= 0.41896547420264 mm] ( 2 , -1.500 )-- ( 13 , -1.500 );
\node[anchor=west,text width=\tendigitwidth] at (13, -1.500 ) {$\times$\hspace{\fill}66 };
\node[anchor=east] at (0, -1.375 ) {$PH_1$};

\draw[lightgray] (0, -2.000 )-- (13, -2.000 );
\draw[line width= 1.7631298660163 mm] ( 2 , -2.500 )-- ( 3 , -2.500 );
\node[anchor=west,text width=\tendigitwidth] at (13, -2.500 ) {$\times$\hspace{\fill}45412532 };
\draw[line width= 0.73031700512368 mm] ( 2 , -2.750 )-- ( 4 , -2.750 );
\node[anchor=west,text width=\tendigitwidth] at (13, -2.750 ) {$\times$\hspace{\fill}1485 };
\draw[line width= 0.73031700512368 mm] ( 2 , -3.000 )-- ( 13 , -3.000 );
\node[anchor=west,text width=\tendigitwidth] at (13, -3.000 ) {$\times$\hspace{\fill}1485 };
\draw[line width= 1.3237082167496 mm] ( 3 , -3.250 )-- ( 4 , -3.250 );
\node[anchor=west,text width=\tendigitwidth] at (13, -3.250 ) {$\times$\hspace{\fill}560779 };
\draw[line width= 0.60867747269123 mm] ( 3 , -3.500 )-- ( 13 , -3.500 );
\node[anchor=west,text width=\tendigitwidth] at (13, -3.500 ) {$\times$\hspace{\fill}440 };
\node[anchor=east] at (0, -3) {$PH_2$};

\draw[lightgray] (0, -4.000 )-- (13, -4.000 );
\draw[line width= 2.0575095578199 mm] ( 2 , -4.500 )-- ( 3 , -4.500 );
\node[anchor=west,text width=\tendigitwidth] at (13, -4.500 ) {$\times$\hspace{\fill}862284291 };
\draw[line width= 1.1864039978328 mm] ( 2 , -4.750 )-- ( 4 , -4.750 );
\node[anchor=west,text width=\tendigitwidth] at (13, -4.750 ) {$\times$\hspace{\fill}142065 };
\draw[line width= 0.95367622727439 mm] ( 2 , -5.000 )-- ( 13 , -5.000 );
\node[anchor=west,text width=\tendigitwidth] at (13, -5.000 ) {$\times$\hspace{\fill}13860 };
\draw[line width= 1.6081317967177 mm] ( 3 , -5.250 )-- ( 4 , -5.250 );
\node[anchor=west,text width=\tendigitwidth] at (13, -5.250 ) {$\times$\hspace{\fill}9638904 };
\draw[line width= 0.96702936653684 mm] ( 3 , -5.500 )-- ( 5 , -5.500 );
\node[anchor=west,text width=\tendigitwidth] at (13, -5.500 ) {$\times$\hspace{\fill}15840 };
\draw[line width= 0.96702936653684 mm] ( 3 , -5.750 )-- ( 13 , -5.750 );
\node[anchor=west,text width=\tendigitwidth] at (13, -5.750 ) {$\times$\hspace{\fill}15840 };
\draw[line width= 1.4722173490966 mm] ( 4 , -6.000 )-- ( 5 , -6.000 );
\node[anchor=west,text width=\tendigitwidth] at (13, -6.000 ) {$\times$\hspace{\fill}2476046 };
\draw[line width= 0.79963172317967 mm] ( 4 , -6.250 )-- ( 13 , -6.250 );
\node[anchor=west,text width=\tendigitwidth] at (13, -6.250 ) {$\times$\hspace{\fill}2970 };
\node[anchor=east] at (0, -5.375 ) {$PH_3$};

\draw[lightgray] (0, -6.750 )-- (13, -6.750 );
\draw[line width= 2.1251617186069 mm] ( 2 , -7.250 )-- ( 3 , -7.250 );
\node[anchor=west,text width=\tendigitwidth] at (13, -7.250 ) {$\times$\hspace{\fill}1696133670 };
\draw[line width= 1.6906205450930 mm] ( 2 , -7.500 )-- ( 4 , -7.500 );
\node[anchor=west,text width=\tendigitwidth] at (13, -7.500 ) {$\times$\hspace{\fill}21992355 };
\draw[line width= 1.0858518112726 mm] ( 2 , -7.750 )-- ( 13 , -7.750 );
\node[anchor=west,text width=\tendigitwidth] at (13, -7.750 ) {$\times$\hspace{\fill}51975 };
\draw[line width= 1.9518873431267 mm] ( 3 , -8.000 )-- ( 4 , -8.000 );
\node[anchor=west,text width=\tendigitwidth] at (13, -8.000 ) {$\times$\hspace{\fill}299874146 };
\draw[line width= 1.3120281211200 mm] ( 3 , -8.250 )-- ( 5 , -8.250 );
\node[anchor=west,text width=\tendigitwidth] at (13, -8.250 ) {$\times$\hspace{\fill}498960 };
\draw[line width= 1.0517591525756 mm] ( 3 , -8.500 )-- ( 6 , -8.500 );
\node[anchor=west,text width=\tendigitwidth] at (13, -8.500 ) {$\times$\hspace{\fill}36960 };
\draw[line width= 1.2222339617994 mm] ( 3 , -8.750 )-- ( 13 , -8.750 );
\node[anchor=west,text width=\tendigitwidth] at (13, -8.750 ) {$\times$\hspace{\fill}203280 };
\draw[line width= 1.7302565584657 mm] ( 4 , -9.000 )-- ( 5 , -9.000 );
\node[anchor=west,text width=\tendigitwidth] at (13, -9.000 ) {$\times$\hspace{\fill}32689536 };
\draw[line width= 1.1482672421799 mm] ( 4 , -9.250 )-- ( 6 , -9.250 );
\node[anchor=west,text width=\tendigitwidth] at (13, -9.250 ) {$\times$\hspace{\fill}97020 };
\draw[line width= 1.1328521741972 mm] ( 4 , -9.500 )-- ( 13 , -9.500 );
\node[anchor=west,text width=\tendigitwidth] at (13, -9.500 ) {$\times$\hspace{\fill}83160 };
\draw[line width= 1.6082906304914 mm] ( 5 , -9.750 )-- ( 6 , -9.750 );
\node[anchor=west,text width=\tendigitwidth] at (13, -9.750 ) {$\times$\hspace{\fill}9654226 };
\draw[line width= 0.98526152221624 mm] ( 5 , -10.00 )-- ( 13 , -10.00 );
\node[anchor=west,text width=\tendigitwidth] at (13, -10.00 ) {$\times$\hspace{\fill}19008 };
\node[anchor=east] at (0, -8.625 ) {$PH_4$};

\draw[lightgray] (0, -10.50 )-- (13, -10.50 );
\draw[line width= 1.9737061966519 mm] ( 2 , -11.00 )-- ( 3 , -11.00 );
\node[anchor=west,text width=\tendigitwidth] at (13, -11.00 ) {$\times$\hspace{\fill}372989925 };
\draw[line width= 1.7849204232400 mm] ( 2 , -11.25 )-- ( 4 , -11.25 );
\node[anchor=west,text width=\tendigitwidth] at (13, -11.25 ) {$\times$\hspace{\fill}56469105 };
\draw[line width= 1.1040839669520 mm] ( 2 , -11.50 )-- ( 13 , -11.50 );
\node[anchor=west,text width=\tendigitwidth] at (13, -11.50 ) {$\times$\hspace{\fill}62370 };
\draw[line width= 2.0820508957417 mm] ( 3 , -11.75 )-- ( 4 , -11.75 );
\node[anchor=west,text width=\tendigitwidth] at (13, -11.75 ) {$\times$\hspace{\fill}1102128291 };
\draw[line width= 1.5996443138524 mm] ( 3 , -12.00 )-- ( 5 , -12.00 );
\node[anchor=west,text width=\tendigitwidth] at (13, -12.00 ) {$\times$\hspace{\fill}8854560 };
\draw[line width= 1.3225641726858 mm] ( 3 , -12.25 )-- ( 6 , -12.25 );
\node[anchor=west,text width=\tendigitwidth] at (13, -12.25 ) {$\times$\hspace{\fill}554400 };
\draw[line width= 1.3918788907418 mm] ( 3 , -12.50 )-- ( 13 , -12.50 );
\node[anchor=west,text width=\tendigitwidth] at (13, -12.50 ) {$\times$\hspace{\fill}1108800 };
\draw[line width= 1.8961254783050 mm] ( 4 , -12.75 )-- ( 5 , -12.75 );
\node[anchor=west,text width=\tendigitwidth] at (13, -12.75 ) {$\times$\hspace{\fill}171699220 };
\draw[line width= 1.4619718228420 mm] ( 4 , -13.00 )-- ( 6 , -13.00 );
\node[anchor=west,text width=\tendigitwidth] at (13, -13.00 ) {$\times$\hspace{\fill}2234925 };
\draw[line width= 1.2714816103092 mm] ( 4 , -13.25 )-- ( 7 , -13.25 );
\node[anchor=west,text width=\tendigitwidth] at (13, -13.25 ) {$\times$\hspace{\fill}332640 };
\draw[line width= 1.3770868777341 mm] ( 4 , -13.50 )-- ( 13 , -13.50 );
\node[anchor=west,text width=\tendigitwidth] at (13, -13.50 ) {$\times$\hspace{\fill}956340 };
\draw[line width= 1.8354333740496 mm] ( 5 , -13.75 )-- ( 6 , -13.75 );
\node[anchor=west,text width=\tendigitwidth] at (13, -13.75 ) {$\times$\hspace{\fill}93580608 };
\draw[line width= 1.3245444354154 mm] ( 5 , -14.00 )-- ( 7 , -14.00 );
\node[anchor=west,text width=\tendigitwidth] at (13, -14.00 ) {$\times$\hspace{\fill}565488 };
\draw[line width= 1.2897137659886 mm] ( 5 , -14.25 )-- ( 13 , -14.25 );
\node[anchor=west,text width=\tendigitwidth] at (13, -14.25 ) {$\times$\hspace{\fill}399168 };
\draw[line width= 1.7302238148005 mm] ( 6 , -14.50 )-- ( 7 , -14.50 );
\node[anchor=west,text width=\tendigitwidth] at (13, -14.50 ) {$\times$\hspace{\fill}32678834 };
\draw[line width= 1.1616203814424 mm] ( 6 , -14.75 )-- ( 13 , -14.75 );
\node[anchor=west,text width=\tendigitwidth] at (13, -14.75 ) {$\times$\hspace{\fill}110880 };
\node[anchor=east] at (0, -12.875 ) {$PH_5$};

\draw[lightgray] (0, -15.25 )-- (13, -15.25 );
\foreach \x in {0,...,12}
{
\node at (\x,-15.5){\x};
}
\node at (13,-15.5){$\infty$};

\end{tikzpicture}
\caption{The persistent homology for the configuration space of 12 disks of unit diameter in a strip of width $w$. The thickness of a bar in the barcode is proportional to the logarithm of the multiplicity, and the exact multiplicity is in the rightmost column.} \label{pdiag1}
\end{figure}

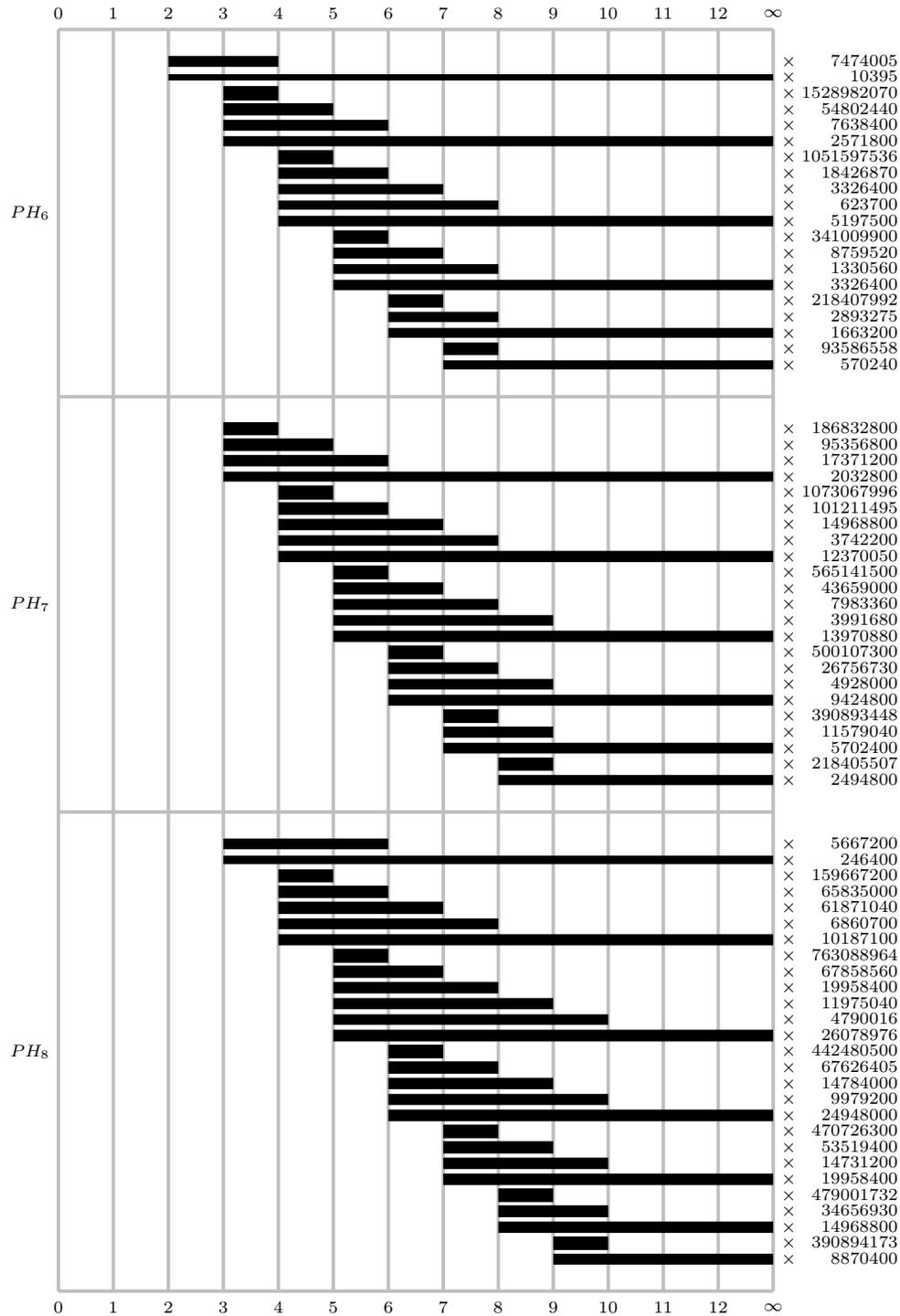
\begin{figure}
\begin{tikzpicture}[line width=0.5mm,xscale=0.8,yscale=0.93]
\tiny

\foreach \x in {0,...,13}
{
\draw[lightgray] (\x,0.5)--(\x,-19.25);
}

\foreach \x in {0,...,12}
{
\node at (\x,0.75){\x};
}

\node at (13,0.75){$\infty$};

\draw[lightgray] (0, 0.5000 )-- (13, 0.5000 );
\draw[line width= 1.5826941558013 mm] ( 2 , 0.0000 )-- ( 4 , 0.0000 );
\node[anchor=west,text width=\tendigitwidth] at (13, 0.0000 ) {$\times$\hspace{\fill}7474005 };
\draw[line width= 0.92490802002921 mm] ( 2 , -0.2500 )-- ( 13 , -0.2500 );
\node[anchor=west,text width=\tendigitwidth] at (13, -0.2500 ) {$\times$\hspace{\fill}10395 };
\draw[line width= 2.1147868037206 mm] ( 3 , -0.5000 )-- ( 4 , -0.5000 );
\node[anchor=west,text width=\tendigitwidth] at (13, -0.5000 ) {$\times$\hspace{\fill}1528982070 };
\draw[line width= 1.7819245276474 mm] ( 3 , -0.7500 )-- ( 5 , -0.7500 );
\node[anchor=west,text width=\tendigitwidth] at (13, -0.7500 ) {$\times$\hspace{\fill}54802440 };
\draw[line width= 1.5848698715127 mm] ( 3 , -1.000 )-- ( 6 , -1.000 );
\node[anchor=west,text width=\tendigitwidth] at (13, -1.000 ) {$\times$\hspace{\fill}7638400 };
\draw[line width= 1.4760116600818 mm] ( 3 , -1.250 )-- ( 13 , -1.250 );
\node[anchor=west,text width=\tendigitwidth] at (13, -1.250 ) {$\times$\hspace{\fill}2571800 };
\draw[line width= 2.0773576307721 mm] ( 4 , -1.500 )-- ( 5 , -1.500 );
\node[anchor=west,text width=\tendigitwidth] at (13, -1.500 ) {$\times$\hspace{\fill}1051597536 };
\draw[line width= 1.6729320483427 mm] ( 4 , -1.750 )-- ( 6 , -1.750 );
\node[anchor=west,text width=\tendigitwidth] at (13, -1.750 ) {$\times$\hspace{\fill}18426870 };
\draw[line width= 1.5017401196086 mm] ( 4 , -2.000 )-- ( 7 , -2.000 );
\node[anchor=west,text width=\tendigitwidth] at (13, -2.000 ) {$\times$\hspace{\fill}3326400 };
\draw[line width= 1.3343424762514 mm] ( 4 , -2.250 )-- ( 8 , -2.250 );
\node[anchor=west,text width=\tendigitwidth] at (13, -2.250 ) {$\times$\hspace{\fill}623700 };
\draw[line width= 1.5463688298714 mm] ( 4 , -2.500 )-- ( 13 , -2.500 );
\node[anchor=west,text width=\tendigitwidth] at (13, -2.500 ) {$\times$\hspace{\fill}5197500 };
\draw[line width= 1.9647422067084 mm] ( 5 , -2.750 )-- ( 6 , -2.750 );
\node[anchor=west,text width=\tendigitwidth] at (13, -2.750 ) {$\times$\hspace{\fill}341009900 };
\draw[line width= 1.5985651666891 mm] ( 5 , -3.000 )-- ( 7 , -3.000 );
\node[anchor=west,text width=\tendigitwidth] at (13, -3.000 ) {$\times$\hspace{\fill}8759520 };
\draw[line width= 1.4101110464212 mm] ( 5 , -3.250 )-- ( 8 , -3.250 );
\node[anchor=west,text width=\tendigitwidth] at (13, -3.250 ) {$\times$\hspace{\fill}1330560 };
\draw[line width= 1.5017401196086 mm] ( 5 , -3.500 )-- ( 13 , -3.500 );
\node[anchor=west,text width=\tendigitwidth] at (13, -3.500 ) {$\times$\hspace{\fill}3326400 };
\draw[line width= 1.9201875394572 mm] ( 6 , -3.750 )-- ( 7 , -3.750 );
\node[anchor=west,text width=\tendigitwidth] at (13, -3.750 ) {$\times$\hspace{\fill}218407992 };
\draw[line width= 1.4877899636475 mm] ( 6 , -4.000 )-- ( 8 , -4.000 );
\node[anchor=west,text width=\tendigitwidth] at (13, -4.000 ) {$\times$\hspace{\fill}2893275 };
\draw[line width= 1.4324254015526 mm] ( 6 , -4.250 )-- ( 13 , -4.250 );
\node[anchor=west,text width=\tendigitwidth] at (13, -4.250 ) {$\times$\hspace{\fill}1663200 };
\draw[line width= 1.8354397320024 mm] ( 7 , -4.500 )-- ( 8 , -4.500 );
\node[anchor=west,text width=\tendigitwidth] at (13, -4.500 ) {$\times$\hspace{\fill}93586558 };
\draw[line width= 1.3253812603825 mm] ( 7 , -4.750 )-- ( 13 , -4.750 );
\node[anchor=west,text width=\tendigitwidth] at (13, -4.750 ) {$\times$\hspace{\fill}570240 };
\draw[lightgray] (0, -5.250 )-- (13, -5.250 );
\node[anchor=east] at (0, -2.375 ) {$PH_6$};

\draw[line width= 1.9045724657210 mm] ( 3 , -5.750 )-- ( 4 , -5.750 );
\node[anchor=west,text width=\tendigitwidth] at (13, -5.750 ) {$\times$\hspace{\fill}186832800 };
\draw[line width= 1.8373136203671 mm] ( 3 , -6.000 )-- ( 5 , -6.000 );
\node[anchor=west,text width=\tendigitwidth] at (13, -6.000 ) {$\times$\hspace{\fill}95356800 };
\draw[line width= 1.6670324220460 mm] ( 3 , -6.250 )-- ( 6 , -6.250 );
\node[anchor=west,text width=\tendigitwidth] at (13, -6.250 ) {$\times$\hspace{\fill}17371200 };
\draw[line width= 1.4524924710988 mm] ( 3 , -6.500 )-- ( 13 , -6.500 );
\node[anchor=west,text width=\tendigitwidth] at (13, -6.500 ) {$\times$\hspace{\fill}2032800 };
\draw[line width= 2.0793787668578 mm] ( 4 , -6.750 )-- ( 5 , -6.750 );
\node[anchor=west,text width=\tendigitwidth] at (13, -6.750 ) {$\times$\hspace{\fill}1073067996 };
\draw[line width= 1.8432722895324 mm] ( 4 , -7.000 )-- ( 6 , -7.000 );
\node[anchor=west,text width=\tendigitwidth] at (13, -7.000 ) {$\times$\hspace{\fill}101211495 };
\draw[line width= 1.6521478592862 mm] ( 4 , -7.250 )-- ( 7 , -7.250 );
\node[anchor=west,text width=\tendigitwidth] at (13, -7.250 ) {$\times$\hspace{\fill}14968800 };
\draw[line width= 1.5135184231742 mm] ( 4 , -7.500 )-- ( 8 , -7.500 );
\node[anchor=west,text width=\tendigitwidth] at (13, -7.500 ) {$\times$\hspace{\fill}3742200 };
\draw[line width= 1.6330788786398 mm] ( 4 , -7.750 )-- ( 13 , -7.750 );
\node[anchor=west,text width=\tendigitwidth] at (13, -7.750 ) {$\times$\hspace{\fill}12370050 };
\draw[line width= 2.0152586700233 mm] ( 5 , -8.000 )-- ( 6 , -8.000 );
\node[anchor=west,text width=\tendigitwidth] at (13, -8.000 ) {$\times$\hspace{\fill}565141500 };
\draw[line width= 1.7591920004564 mm] ( 5 , -8.250 )-- ( 7 , -8.250 );
\node[anchor=west,text width=\tendigitwidth] at (13, -8.250 ) {$\times$\hspace{\fill}43659000 };
\draw[line width= 1.5892869933440 mm] ( 5 , -8.500 )-- ( 8 , -8.500 );
\node[anchor=west,text width=\tendigitwidth] at (13, -8.500 ) {$\times$\hspace{\fill}7983360 };
\draw[line width= 1.5199722752880 mm] ( 5 , -8.750 )-- ( 9 , -8.750 );
\node[anchor=west,text width=\tendigitwidth] at (13, -8.750 ) {$\times$\hspace{\fill}3991680 };
\draw[line width= 1.6452485721375 mm] ( 5 , -9.000 )-- ( 13 , -9.000 );
\node[anchor=west,text width=\tendigitwidth] at (13, -9.000 ) {$\times$\hspace{\fill}13970880 };
\draw[line width= 2.0030333233363 mm] ( 6 , -9.250 )-- ( 7 , -9.250 );
\node[anchor=west,text width=\tendigitwidth] at (13, -9.250 ) {$\times$\hspace{\fill}500107300 };
\draw[line width= 1.7102296588448 mm] ( 6 , -9.500 )-- ( 8 , -9.500 );
\node[anchor=west,text width=\tendigitwidth] at (13, -9.500 ) {$\times$\hspace{\fill}26756730 };
\draw[line width= 1.5410443784195 mm] ( 6 , -9.750 )-- ( 9 , -9.750 );
\node[anchor=west,text width=\tendigitwidth] at (13, -9.750 ) {$\times$\hspace{\fill}4928000 };
\draw[line width= 1.6058855070914 mm] ( 6 , -10.00 )-- ( 13 , -10.00 );
\node[anchor=west,text width=\tendigitwidth] at (13, -10.00 ) {$\times$\hspace{\fill}9424800 };
\draw[line width= 1.9783945569303 mm] ( 7 , -10.25 )-- ( 8 , -10.25 );
\node[anchor=west,text width=\tendigitwidth] at (13, -10.25 ) {$\times$\hspace{\fill}390893448 };
\draw[line width= 1.6264707125118 mm] ( 7 , -10.50 )-- ( 9 , -10.50 );
\node[anchor=west,text width=\tendigitwidth] at (13, -10.50 ) {$\times$\hspace{\fill}11579040 };
\draw[line width= 1.5556397696819 mm] ( 7 , -10.75 )-- ( 13 , -10.75 );
\node[anchor=west,text width=\tendigitwidth] at (13, -10.75 ) {$\times$\hspace{\fill}5702400 };
\draw[line width= 1.9201864016719 mm] ( 8 , -11.00 )-- ( 9 , -11.00 );
\node[anchor=west,text width=\tendigitwidth] at (13, -11.00 ) {$\times$\hspace{\fill}218405507 };
\draw[line width= 1.4729719123634 mm] ( 8 , -11.25 )-- ( 13 , -11.25 );
\node[anchor=west,text width=\tendigitwidth] at (13, -11.25 ) {$\times$\hspace{\fill}2494800 };
\draw[lightgray] (0, -11.75 )-- (13, -11.75 );
\node[anchor=east] at (0, -8.5 ) {$PH_7$};

\draw[line width= 1.5550205726571 mm] ( 3 , -12.25 )-- ( 6 , -12.25 );
\node[anchor=west,text width=\tendigitwidth] at (13, -12.25 ) {$\times$\hspace{\fill}5667200 };
\draw[line width= 1.2414711510642 mm] ( 3 , -12.50 )-- ( 13 , -12.50 );
\node[anchor=west,text width=\tendigitwidth] at (13, -12.50 ) {$\times$\hspace{\fill}246400 };
\draw[line width= 1.8888602206994 mm] ( 4 , -12.75 )-- ( 5 , -12.75 );
\node[anchor=west,text width=\tendigitwidth] at (13, -12.75 ) {$\times$\hspace{\fill}159667200 };
\draw[line width= 1.8002662169773 mm] ( 4 , -13.00 )-- ( 6 , -13.00 );
\node[anchor=west,text width=\tendigitwidth] at (13, -13.00 ) {$\times$\hspace{\fill}65835000 };
\draw[line width= 1.7940562776805 mm] ( 4 , -13.25 )-- ( 7 , -13.25 );
\node[anchor=west,text width=\tendigitwidth] at (13, -13.25 ) {$\times$\hspace{\fill}61871040 };
\draw[line width= 1.5741320035313 mm] ( 4 , -13.50 )-- ( 8 , -13.50 );
\node[anchor=west,text width=\tendigitwidth] at (13, -13.50 ) {$\times$\hspace{\fill}6860700 };
\draw[line width= 1.6136632771957 mm] ( 4 , -13.75 )-- ( 13 , -13.75 );
\node[anchor=west,text width=\tendigitwidth] at (13, -13.75 ) {$\times$\hspace{\fill}10187100 };
\draw[line width= 2.0452885180093 mm] ( 5 , -14.00 )-- ( 6 , -14.00 );
\node[anchor=west,text width=\tendigitwidth] at (13, -14.00 ) {$\times$\hspace{\fill}763088964 };
\draw[line width= 1.8032936096936 mm] ( 5 , -14.25 )-- ( 7 , -14.25 );
\node[anchor=west,text width=\tendigitwidth] at (13, -14.25 ) {$\times$\hspace{\fill}67858560 };
\draw[line width= 1.6809160665314 mm] ( 5 , -14.50 )-- ( 8 , -14.50 );
\node[anchor=west,text width=\tendigitwidth] at (13, -14.50 ) {$\times$\hspace{\fill}19958400 };
\draw[line width= 1.6298335041548 mm] ( 5 , -14.75 )-- ( 9 , -14.75 );
\node[anchor=west,text width=\tendigitwidth] at (13, -14.75 ) {$\times$\hspace{\fill}11975040 };
\draw[line width= 1.5382044309674 mm] ( 5 , -15.00 )-- ( 10 , -15.00 );
\node[anchor=west,text width=\tendigitwidth] at (13, -15.00 ) {$\times$\hspace{\fill}4790016 };
\draw[line width= 1.7076640030448 mm] ( 5 , -15.25 )-- ( 13 , -15.25 );
\node[anchor=west,text width=\tendigitwidth] at (13, -15.25 ) {$\times$\hspace{\fill}26078976 };
\draw[line width= 1.9907906953645 mm] ( 6 , -15.50 )-- ( 7 , -15.50 );
\node[anchor=west,text width=\tendigitwidth] at (13, -15.50 ) {$\times$\hspace{\fill}442480500 };
\draw[line width= 1.8029509071255 mm] ( 6 , -15.75 )-- ( 8 , -15.75 );
\node[anchor=west,text width=\tendigitwidth] at (13, -15.75 ) {$\times$\hspace{\fill}67626405 };
\draw[line width= 1.6509056072864 mm] ( 6 , -16.00 )-- ( 9 , -16.00 );
\node[anchor=west,text width=\tendigitwidth] at (13, -16.00 ) {$\times$\hspace{\fill}14784000 };
\draw[line width= 1.6116013484754 mm] ( 6 , -16.25 )-- ( 10 , -16.25 );
\node[anchor=west,text width=\tendigitwidth] at (13, -16.25 ) {$\times$\hspace{\fill}9979200 };
\draw[line width= 1.7032304216628 mm] ( 6 , -16.50 )-- ( 13 , -16.50 );
\node[anchor=west,text width=\tendigitwidth] at (13, -16.50 ) {$\times$\hspace{\fill}24948000 };
\draw[line width= 1.9969787379040 mm] ( 7 , -16.75 )-- ( 8 , -16.75 );
\node[anchor=west,text width=\tendigitwidth] at (13, -16.75 ) {$\times$\hspace{\fill}470726300 };
\draw[line width= 1.7795554762959 mm] ( 7 , -17.00 )-- ( 9 , -17.00 );
\node[anchor=west,text width=\tendigitwidth] at (13, -17.00 ) {$\times$\hspace{\fill}53519400 };
\draw[line width= 1.6505478251516 mm] ( 7 , -17.25 )-- ( 10 , -17.25 );
\node[anchor=west,text width=\tendigitwidth] at (13, -17.25 ) {$\times$\hspace{\fill}14731200 };
\draw[line width= 1.6809160665314 mm] ( 7 , -17.50 )-- ( 13 , -17.50 );
\node[anchor=west,text width=\tendigitwidth] at (13, -17.50 ) {$\times$\hspace{\fill}19958400 };
\draw[line width= 1.9987214771235 mm] ( 8 , -17.75 )-- ( 9 , -17.75 );
\node[anchor=west,text width=\tendigitwidth] at (13, -17.75 ) {$\times$\hspace{\fill}479001732 };
\draw[line width= 1.7361008263603 mm] ( 8 , -18.00 )-- ( 10 , -18.00 );
\node[anchor=west,text width=\tendigitwidth] at (13, -18.00 ) {$\times$\hspace{\fill}34656930 };
\draw[line width= 1.6521478592862 mm] ( 8 , -18.25 )-- ( 13 , -18.25 );
\node[anchor=west,text width=\tendigitwidth] at (13, -18.25 ) {$\times$\hspace{\fill}14968800 };
\draw[line width= 1.9783947424026 mm] ( 9 , -18.50 )-- ( 10 , -18.50 );
\node[anchor=west,text width=\tendigitwidth] at (13, -18.50 ) {$\times$\hspace{\fill}390894173 };
\draw[line width= 1.5998230449098 mm] ( 9 , -18.75 )-- ( 13 , -18.75 );
\node[anchor=west,text width=\tendigitwidth] at (13, -18.75 ) {$\times$\hspace{\fill}8870400 };
\draw[lightgray] (0, -19.25 )-- (13, -19.25 );
\node[anchor=east] at (0, -15.5 ) {$PH_8$};

\foreach \x in {0,...,12}
{
\node at (\x,-19.5){\x};
}
\node at (13,-19.5){$\infty$};

\end{tikzpicture}
\caption{The persistent homology for the configuration space of 12 disks of unit diameter in a strip of width $w$. The thickness of a bar in the barcode is proportional to the logarithm of the multiplicity, and the exact multiplicity is in the rightmost column.} \label{pdiag2}
\end{figure}

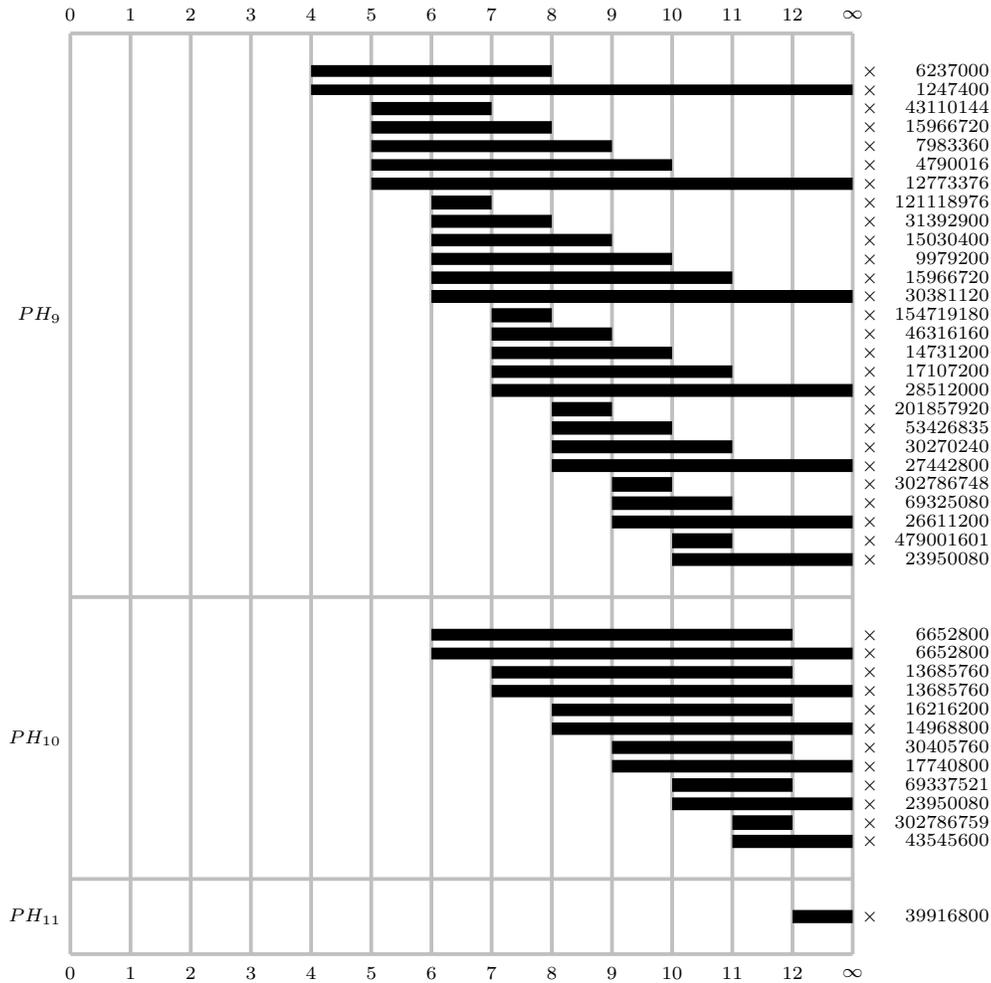
\begin{figure}
\begin{tikzpicture}[line width=0.5mm,xscale=0.8,yscale=1]
\tiny

\foreach \x in {0,...,13}
{
\draw[lightgray] (\x,0.5)--(\x,-11.75);
}

\foreach \x in {0,...,12}
{
\node at (\x,0.75){\x};
}

\node at (13,0.75){$\infty$};

\draw[lightgray] (0, 0.5000 )-- (13, 0.5000 );

\draw[line width= 1.5646009855508 mm] ( 4 , 0.0000 )-- ( 8 , 0.0000 );
\node[anchor=west,text width=\tendigitwidth] at (13, 0.0000 ) {$\times$\hspace{\fill}6237000 };
\draw[line width= 1.4036571943074 mm] ( 4 , -0.2500 )-- ( 13 , -0.2500 );
\node[anchor=west,text width=\tendigitwidth] at (13, -0.2500 ) {$\times$\hspace{\fill}1247400 };
\draw[line width= 1.7579268887010 mm] ( 5 , -0.5000 )-- ( 7 , -0.5000 );
\node[anchor=west,text width=\tendigitwidth] at (13, -0.5000 ) {$\times$\hspace{\fill}43110144 };
\draw[line width= 1.6586017114000 mm] ( 5 , -0.7500 )-- ( 8 , -0.7500 );
\node[anchor=west,text width=\tendigitwidth] at (13, -0.7500 ) {$\times$\hspace{\fill}15966720 };
\draw[line width= 1.5892869933440 mm] ( 5 , -1.000 )-- ( 9 , -1.000 );
\node[anchor=west,text width=\tendigitwidth] at (13, -1.000 ) {$\times$\hspace{\fill}7983360 };
\draw[line width= 1.5382044309674 mm] ( 5 , -1.250 )-- ( 10 , -1.250 );
\node[anchor=west,text width=\tendigitwidth] at (13, -1.250 ) {$\times$\hspace{\fill}4790016 };
\draw[line width= 1.6362873562686 mm] ( 5 , -1.500 )-- ( 13 , -1.500 );
\node[anchor=west,text width=\tendigitwidth] at (13, -1.500 ) {$\times$\hspace{\fill}12773376 };
\draw[line width= 1.8612283893192 mm] ( 6 , -1.750 )-- ( 7 , -1.750 );
\node[anchor=west,text width=\tendigitwidth] at (13, -1.750 ) {$\times$\hspace{\fill}121118976 };
\draw[line width= 1.7262092310661 mm] ( 6 , -2.000 )-- ( 8 , -2.000 );
\node[anchor=west,text width=\tendigitwidth] at (13, -2.000 ) {$\times$\hspace{\fill}31392900 };
\draw[line width= 1.6525585374815 mm] ( 6 , -2.250 )-- ( 9 , -2.250 );
\node[anchor=west,text width=\tendigitwidth] at (13, -2.250 ) {$\times$\hspace{\fill}15030400 };
\draw[line width= 1.6116013484754 mm] ( 6 , -2.500 )-- ( 10 , -2.500 );
\node[anchor=west,text width=\tendigitwidth] at (13, -2.500 ) {$\times$\hspace{\fill}9979200 };
\draw[line width= 1.6586017114000 mm] ( 6 , -2.750 )-- ( 11 , -2.750 );
\node[anchor=west,text width=\tendigitwidth] at (13, -2.750 ) {$\times$\hspace{\fill}15966720 };
\draw[line width= 1.7229331920812 mm] ( 6 , -3.000 )-- ( 13 , -3.000 );
\node[anchor=west,text width=\tendigitwidth] at (13, -3.000 ) {$\times$\hspace{\fill}30381120 };
\draw[line width= 1.8857122289759 mm] ( 7 , -3.250 )-- ( 8 , -3.250 );
\node[anchor=west,text width=\tendigitwidth] at (13, -3.250 ) {$\times$\hspace{\fill}154719180 };
\draw[line width= 1.7651001486238 mm] ( 7 , -3.500 )-- ( 9 , -3.500 );
\node[anchor=west,text width=\tendigitwidth] at (13, -3.500 ) {$\times$\hspace{\fill}46316160 };
\draw[line width= 1.6505478251516 mm] ( 7 , -3.750 )-- ( 10 , -3.750 );
\node[anchor=west,text width=\tendigitwidth] at (13, -3.750 ) {$\times$\hspace{\fill}14731200 };
\draw[line width= 1.6655009985487 mm] ( 7 , -4.000 )-- ( 11 , -4.000 );
\node[anchor=west,text width=\tendigitwidth] at (13, -4.000 ) {$\times$\hspace{\fill}17107200 };
\draw[line width= 1.7165835609253 mm] ( 7 , -4.250 )-- ( 13 , -4.250 );
\node[anchor=west,text width=\tendigitwidth] at (13, -4.250 ) {$\times$\hspace{\fill}28512000 };
\draw[line width= 1.9123074641551 mm] ( 8 , -4.500 )-- ( 9 , -4.500 );
\node[anchor=west,text width=\tendigitwidth] at (13, -4.500 ) {$\times$\hspace{\fill}201857920 };
\draw[line width= 1.7793823705795 mm] ( 8 , -4.750 )-- ( 10 , -4.750 );
\node[anchor=west,text width=\tendigitwidth] at (13, -4.750 ) {$\times$\hspace{\fill}53426835 };
\draw[line width= 1.7225675609609 mm] ( 8 , -5.000 )-- ( 11 , -5.000 );
\node[anchor=west,text width=\tendigitwidth] at (13, -5.000 ) {$\times$\hspace{\fill}30270240 };
\draw[line width= 1.7127614396432 mm] ( 8 , -5.250 )-- ( 13 , -5.250 );
\node[anchor=west,text width=\tendigitwidth] at (13, -5.250 ) {$\times$\hspace{\fill}27442800 };
\draw[line width= 1.9528539313709 mm] ( 9 , -5.500 )-- ( 10 , -5.500 );
\node[anchor=west,text width=\tendigitwidth] at (13, -5.500 ) {$\times$\hspace{\fill}302786748 };
\draw[line width= 1.8054317303450 mm] ( 9 , -5.750 )-- ( 11 , -5.750 );
\node[anchor=west,text width=\tendigitwidth] at (13, -5.750 ) {$\times$\hspace{\fill}69325080 };
\draw[line width= 1.7096842737766 mm] ( 9 , -6.000 )-- ( 13 , -6.000 );
\node[anchor=west,text width=\tendigitwidth] at (13, -6.000 ) {$\times$\hspace{\fill}26611200 };
\draw[line width= 1.9987214497750 mm] ( 10 , -6.250 )-- ( 11 , -6.250 );
\node[anchor=west,text width=\tendigitwidth] at (13, -6.250 ) {$\times$\hspace{\fill}479001601 };
\draw[line width= 1.6991482222108 mm] ( 10 , -6.500 )-- ( 13 , -6.500 );
\node[anchor=west,text width=\tendigitwidth] at (13, -6.500 ) {$\times$\hspace{\fill}23950080 };
\draw[lightgray] (0, -7.000 )-- (13, -7.000 );
\node[anchor=east] at (0, -3.25 ) {$PH_9$};

\draw[line width= 1.5710548376646 mm] ( 6 , -7.500 )-- ( 12 , -7.500 );
\node[anchor=west,text width=\tendigitwidth] at (13, -7.500 ) {$\times$\hspace{\fill}6652800 };
\draw[line width= 1.5710548376646 mm] ( 6 , -7.750 )-- ( 13 , -7.750 );
\node[anchor=west,text width=\tendigitwidth] at (13, -7.750 ) {$\times$\hspace{\fill}6652800 };
\draw[line width= 1.6431866434172 mm] ( 7 , -8.000 )-- ( 12 , -8.000 );
\node[anchor=west,text width=\tendigitwidth] at (13, -8.000 ) {$\times$\hspace{\fill}13685760 };
\draw[line width= 1.6431866434172 mm] ( 7 , -8.250 )-- ( 13 , -8.250 );
\node[anchor=west,text width=\tendigitwidth] at (13, -8.250 ) {$\times$\hspace{\fill}13685760 };
\draw[line width= 1.6601521300536 mm] ( 8 , -8.500 )-- ( 12 , -8.500 );
\node[anchor=west,text width=\tendigitwidth] at (13, -8.500 ) {$\times$\hspace{\fill}16216200 };
\draw[line width= 1.6521478592862 mm] ( 8 , -8.750 )-- ( 13 , -8.750 );
\node[anchor=west,text width=\tendigitwidth] at (13, -8.750 ) {$\times$\hspace{\fill}14968800 };
\draw[line width= 1.7230142622113 mm] ( 9 , -9.000 )-- ( 12 , -9.000 );
\node[anchor=west,text width=\tendigitwidth] at (13, -9.000 ) {$\times$\hspace{\fill}30405760 };
\draw[line width= 1.6691377629658 mm] ( 9 , -9.250 )-- ( 13 , -9.250 );
\node[anchor=west,text width=\tendigitwidth] at (13, -9.250 ) {$\times$\hspace{\fill}17740800 };
\draw[line width= 1.8054496746212 mm] ( 10 , -9.500 )-- ( 12 , -9.500 );
\node[anchor=west,text width=\tendigitwidth] at (13, -9.500 ) {$\times$\hspace{\fill}69337521 };
\draw[line width= 1.6991482222108 mm] ( 10 , -9.750 )-- ( 13 , -9.750 );
\node[anchor=west,text width=\tendigitwidth] at (13, -9.750 ) {$\times$\hspace{\fill}23950080 };
\draw[line width= 1.9528539350038 mm] ( 11 , -10.00 )-- ( 12 , -10.00 );
\node[anchor=west,text width=\tendigitwidth] at (13, -10.00 ) {$\times$\hspace{\fill}302786759 };
\draw[line width= 1.7589319222864 mm] ( 11 , -10.25 )-- ( 13 , -10.25 );
\node[anchor=west,text width=\tendigitwidth] at (13, -10.25 ) {$\times$\hspace{\fill}43545600 };
\draw[lightgray] (0, -10.75 )-- (13, -10.75 );
\node[anchor=east] at (0, -8.875 ) {$PH_{10}$};

\draw[line width= 1.7502307845874 mm] ( 12 , -11.25 )-- ( 13 , -11.25 );
\node[anchor=west,text width=\tendigitwidth] at (13, -11.25 ) {$\times$\hspace{\fill}39916800 };
\node[anchor=east] at (0, -11.25 ) {$PH_{11}$};
\draw[lightgray] (0, -11.75 )-- (13, -11.75 );

\foreach \x in {0,...,12}
{
\node at (\x,-12){\x};
}
\node at (13,-12){$\infty$};

\end{tikzpicture}

\caption{The persistent homology for the configuration space of 12 disks of unit diameter in a strip of width $w$. The thickness of a bar in the barcode is proportional to the logarithm of the multiplicity, and the exact multiplicity is in the rightmost column.} \label{pdiag3}
\end{figure}

\bibliographystyle{amsalpha}
\bibliography{harddisks}

\end{document}